\documentclass[a4paper,oneside]{amsart}
\usepackage{amssymb, amsmath, mathtools}
\usepackage{mathrsfs} 
\usepackage{stmaryrd} 
\usepackage[DIV=11]{typearea}
\usepackage[shortlabels]{enumitem}
\usepackage{booktabs}
\usepackage{placeins}
\usepackage{pgfplots}
\usepackage{pgfplotstable}
\pgfplotsset{compat=1.18}
\usetikzlibrary{pgfplots.groupplots}
\usetikzlibrary{external}
\usepackage[colorlinks,linkcolor={blue},citecolor={blue},urlcolor={black},bookmarksdepth=2]{hyperref}

\definecolor{PTblue}{RGB}{68,119,170}
\definecolor{PTcyan}{RGB}{102,204,238}
\definecolor{PTgreen}{RGB}{34,136,51}
\definecolor{PTyellow}{RGB}{204,187,68}
\definecolor{PTred}{RGB}{238,102,119}
\definecolor{PTpurple}{RGB}{170,51,119}
\definecolor{PTgrey}{RGB}{187,187,187}
\pgfplotsset{
    table/search path={data},
    every axis/.append style={font=\normalsize, semithick},
    every axis plot/.append style={very thick, mark=-},
    axis lines=left,
    tick align=inside, yticklabel style={/pgf/number format/fixed},
    xmajorgrids, ymajorgrids, grid style={draw opacity=0.15,color=black},
    colormap={mycolors}{rgb255=(100,143,255) rgb255=(120,94,240) rgb255=(254,97,0) rgb255=(255,176,0)},
    colormap={PTbright}{color=(PTblue) color=(PTred) color=(PTgreen) color=(PTyellow) color=(PTcyan) color=(PTpurple)},
    width=0.7\textwidth, height=0.7\textwidth,
}

\theoremstyle{plain}
\newtheorem{theorem}{Theorem}[section]
\theoremstyle{remark}
\newtheorem{remark}[theorem]{Remark}

\theoremstyle{plain}
\newtheorem{corollary}[theorem]{Corollary}
\newtheorem{lemma}[theorem]{Lemma}
\newtheorem{proposition}[theorem]{Proposition}
\newtheorem{definition}[theorem]{Definition}
\newtheorem{notation}[theorem]{Notation}

\newtheorem{assumption}{Assumption}

\numberwithin{equation}{section}

\newcommand{\N}{\mathbb{N}}
\newcommand{\Z}{\mathbb{Z}}

\newcommand{\R}{\mathbb{R}}
\newcommand{\C}{\mathbb{C}}
\newcommand{\T}{\mathbb{T}} 

\newcommand{\rmd}{\mathrm{d}}

\newcommand{\ce}{\coloneqq}
\newcommand{\ec}{\eqqcolon}
\DeclarePairedDelimiter{\abs}{\lvert}{\rvert}
\DeclarePairedDelimiter{\norm}{\lVert}{\rVert}
\DeclarePairedDelimiter{\p}{\lparen}{\rparen}
\DeclarePairedDelimiter{\br}{\lbrack}{\rbrack}
\DeclarePairedDelimiter{\floorh}{\lfloor}{\rfloor}
\DeclarePairedDelimiter{\ceilh}{\lceil}{\rceil}
\NewDocumentCommand{\tnorm}{O{}m}{\mathopen{#1|\tnormkern #1|\tnormkern #1|} #2 \mathclose{#1|\tnormkern #1|\tnormkern #1|}}
\NewDocumentCommand{\qnorm}{O{}m}{\mathopen{#1|\tnormkern #1|\tnormkern #1|\tnormkern #1|} #2 \mathclose{#1|\tnormkern #1|\tnormkern #1|\tnormkern #1|}}
\newcommand{\tnormkern}{\mkern-1.5mu\relax}
\DeclarePairedDelimiter{\bnorm}{\llbracket}{\rrbracket}
\DeclareMathOperator{\Id}{I}

\DeclareMathOperator{\dom}{D}
\newcommand{\bfE}{\mathbf{E}}
\newcommand{\bfH}{\mathbf{H}}
\newcommand{\calB}{\mathcal{B}}

\newcommand{\calL}{\mathscr{L}}
\newcommand{\calO}{\mathcal{O}}
\newcommand{\calP}{\mathcal{P}}
\newcommand{\calR}{\mathcal{R}}
\newcommand{\calX}{\mathcal{X}}
\newcommand{\1}{\mathbf{1}}

\newcommand{\LHX}{{\calL_2(H,X)}} 
\newcommand{\LHY}{{\calL_2(H,Y)}}
\newcommand{\LHZ}{{\calL_2(H,Z)}}
\newcommand{\bilinHX}{{\calL_2^{(2)}(H,X)}} 
\newcommand{\bilinHY}{{\calL_2^{(2)}(H,Y)}} 
\newcommand{\hra}{\hookrightarrow}
\newcommand{\seq}{\subseteq}
\newcommand{\CY}{C_Y} 
\newcommand{\LFX}{L_{F,X}} 
\newcommand{\LGX}{L_{G,X}}
\newcommand{\LGpGX}{L_{G'G,X}}
\newcommand{\LFY}{L_{F,Y}} 
\newcommand{\LGY}{L_{G,Y}}
\newcommand{\LGpGY}{L_{G'G,Y}}
\newcommand{\CalphaF}{C_{\alpha,F}} 
\newcommand{\CalphaG}{C_{\alpha,G}}
\newcommand{\CFp}{L_{F,X}} 
\newcommand{\CGp}{L_{G,X}} 
\newcommand{\CbetaF}{C_{\alpha,F'}} 
\newcommand{\CbetaG}{C_{\alpha,G'}}
\newcommand{\CWPpX}{C_{p,X}^{\mathrm{WP}}} 
\newcommand{\CWPpY}{C_{p,Y}^{\mathrm{WP}}} 
\newcommand{\CWPqY}{C_{q,Y}^{\mathrm{WP}}} 
\newcommand{\CWPqZ}{C_{q,Z}^{\mathrm{WP}}}
\newcommand{\CstabpY}{C_{p,Y}^{\mathrm{stab}}}

\newcommand{\LpXone}{L_{p,X,1}} 
\newcommand{\LpXtwo}{L_{p,X,2}}

\newcommand{\floors}{\lfloor s \rfloor}

\newcommand{\dr}{\,\rmd r}
\newcommand{\ds}{\,\rmd s}
\newcommand{\dt}{\,\rmd t}
\newcommand{\dx}{\,\rmd x}
\newcommand{\dzeta}{\,\rmd \zeta}
\newcommand{\dW}{\,\rmd W}
\newcommand{\dWHr}{\,\rmd W_r} 
\newcommand{\dWHs}{\,\rmd W_s}

\newcommand{\e}{{\mathrm{e}}}
\newcommand{\iu}{{\mathrm{i}}}
\newcommand{\E}{\mathbb{E}}
\newcommand{\F}{\mathscr{F}}
\newcommand{\PP}{\mathbb{P}}

\DeclareMathOperator{\curl}{curl}
\DeclareMathOperator{\diag}{diag}

\ExplSyntaxOn
\NewDocumentCommand{\maybeParen}{m}{
    \str_if_in:nnTF {#1} {+} {(#1)} {
        \str_if_in:nnTF {#1} {-} {(#1)} {#1}
    }
}
\ExplSyntaxOff

\newcommand{\h}{h}              
\newcommand{\init}{\xi}         
\renewcommand{\t}[1]{t_{#1}} 
\newcommand{\sg}[1]{S(#1)}      
\newcommand{\ind}{\mathbf{1}}   
\newcommand{\sol}[1]{U_{#1}}
\newcommand{\app}[1]{u_{#1}}

\makeatletter
\let\ftype@table\ftype@figure
\makeatother

\allowdisplaybreaks

\begin{document}

\author{Felix Kastner}
\address{Department of Mathematics\\ University of Hamburg\\ Bundesstraße 55\\ 20146 Hamburg\\ Germany} \email{felix.kastner@uni-hamburg.de}

\author{Katharina Klioba}
\address{Delft Institute of Applied Mathematics\\ TU Delft\\ Mekelweg 4\\
2628 CD Delft\\ the Netherlands.} \email{k.klioba-1@tudelft.nl}

\thanks{The second author gratefully acknowledges support by the Alexander von Humboldt Foundation through a Feodor Lynen Research Fellowship. The authors are also grateful for support from the VICI subsidy VI.C.212.027 of the Netherlands Organisation for Scientific Research (NWO)}

\date\today
\title[Milstein-type schemes for hyperbolic SPDEs]{Milstein-type schemes for hyperbolic SPDEs}


\begin{abstract}

This article studies the temporal approximation of hyperbolic semilinear stochastic evolution equations with multiplicative Gaussian noise by Milstein-type schemes. 
We take the term hyperbolic to mean that the leading operator generates a contractive, not necessarily analytic $C_0$-semigroup.
Optimal convergence rates are derived for the pathwise uniform strong error
\[
    E_h^\infty \coloneqq \Big(\mathbb{E}\Big[\max_{1\le j \le M}\|U_{t_j}-u_j\|_X^p\Big]\Big)^{1/p}
\]
on a Hilbert space $X$ for $p\in [2,\infty)$. Here, $U$ is the mild solution and $u_j$ its Milstein approximation at time $t_j=jh$ with step size $h>0$ and final time $T=Mh>0$. For sufficiently regular nonlinearity and noise, we establish strong convergence of order one, with the error satisfying $E_h^\infty\lesssim h\sqrt{\log(T/h)}$ for rational Milstein schemes and $E_h^\infty \lesssim h$ for exponential Milstein schemes. This extends previous results from parabolic to hyperbolic SPDEs and from exponential to rational Milstein schemes. Moreover, root-mean-square error estimates are strengthened to pathwise uniform estimates. Numerical experiments validate the convergence rates for the stochastic Schrödinger equation. Further applications to Maxwell's and transport equations are included. 

\end{abstract}

\maketitle

\section{Introduction}
\label{sec:intro}

In this article, we study the temporal approximation of semilinear stochastic evolution equations of the form
\begin{equation}\label{eq:introSEE}
	\begin{cases}
		\rmd U + AU \dt &= F(U)\dt + G(U) \dW\quad \text{on }(0,T], \\
		U_0&=\init \in L^p(\Omega;X).
	\end{cases}
\end{equation}
Here, $-A$ is the generator of a contractive but in general not analytic $C_0$-semigroup on a Hilbert space $X$, the nonlinearity $F$ and the multiplicative noise $G$ are globally Lipschitz, $W$ is a cylindrical Brownian motion, $\init$ is the initial data, and $p\in [2,\infty)$.
In general, $F$ and $G$ can be time-dependent and random but for simplicity we restrict this introductory discussion to autonomous, deterministic coefficients.

The aim of this paper is to obtain strong convergence rates exceeding $1/2$ for a class of time discretisation schemes, the so-called \emph{Milstein schemes}, in the
hyperbolic setting. We present regularity assumptions on $F$, $G$, and $\init$ under which the optimal rate $1$ is attained, such as linear growth of $F$ and $G$ on $D(A)$ for the classical Milstein scheme or on $D(A^2)$ for rational versions thereof.

Hyperbolic SPDEs such as Schrödinger or Maxwell's equations have attracted considerable attention in recent years  (see \cite{AC18, BLM21, BC23, CCHS20, CL22, CLS13, CQS16, Cui25NLS, HHS22, KLP20} and references therein). Strong convergence rates have been studied in this setting for the exponential Euler method for some non-parabolic equations~\cite{AC18,CCHS20}, which was generalised to rational schemes with a unifying approach by one of the authors in~\cite{KliobaVeraar24Rate}. However, it is known that the rate $1/2$ up to a logarithmic correction factor is optimal among all schemes solely using the Wiener increments as information from the Brownian motion. Indeed, already for SDEs, the error of such schemes has to grow at least like $\sqrt{\log(T/h)}h^{1/2}$ for step size $h\to 0$~\cite{Muller-Gronbach}. 

\subsection*{The Milstein scheme for SPDEs}
The Milstein scheme overcomes this limitation by incorporating more information from the Brownian motion. Originally developed for SDEs \cite{Milstein75} (see also~\cite[Sec.~10.3]{KloedenPlaten99}) via an Itô--Taylor expansion of first order, Jentzen and Röckner extended it to SPDEs in their seminal paper \cite{JentzenRoeckner15_Milstein}.

The central idea behind the Milstein scheme is as follows. Assume for simplicity that the mild solution of \eqref{eq:introSEE} is even a strong one, i.e.\ it can be written as an Itô process $\sol{h}=\init+\int_0^h (-A\sol{s}+ F(\sol{s}))\ds + \int_0^h G(\sol{s})\dWHs$, where $h>0$ is the time step. To construct a first-order approximation, it suffices to approximate the temporal integrand $-A\sol{s}+F(\sol{s})\approx -A\sol{0}+F(\sol{0})$ to order $0$. In contrast, since the stochastic integral is only of order $1/2$, the stochastic integrand should be approximated to order $1/2$ rather than $0$. From Itô's formula, we deduce $G(\sol{s}) \approx G(\sol{0})+\int_0^s G'(\sol{r})\,\rmd\sol{r}$ and thus, using \eqref{eq:introSEE} and omitting all terms of order greater than $1/2$,
\begin{align*}
    G(\sol{s})
    &\approx G(\sol{0})+\int_0^s G'(\sol{r})\p[\big]{-A\sol{r}+F(\sol{r})}\dr +\int_0^s G'(\sol{r})G(\sol{r})\dWHr \\
    &\approx G(\sol{0})+\int_0^s G'(\sol{0})G(\sol{0})\dWHr.
\end{align*} 
Inserting these approximations in the solution formula results in the expansion
\begin{align*}
    \sol{h} \approx \init-hA\sol{0}+h F(\sol{0}) + \int_0^h G(\sol{0}) \dWHs + \int_0^h \int_0^s G'(\sol{0})G(\sol{0})\dWHr \dWHs
\end{align*}
with an iterated stochastic integral characteristic of Milstein schemes. An analogous reasoning applies to mild (rather than strong) solutions via mild Itô formulas, cf.\ \cite[Sec.~3.3]{mildItoFormula}, where convolutions with the semigroup $(\sg{t})_{t\ge 0}$ arise. Motivated by this observation, we define the \emph{rational Milstein scheme} $u=(\app{j})_{j=0,\ldots,M}$ by $\app{0}\ce\init$ and for $1 \le j \le M$
    \begin{align*}
        \app{j} &\ce R_\h^j \init + \h \sum_{i=0}^{j-1} R_\h^{j-i}F(\app{i}) + \sum_{i=0}^{j-1} R_\h^{j-i}\p[\Big]{\int_{\t{i}}^{\t{i+1}}G(\app{i})\dWHs + \int_{\t{i}}^{\t{i+1}} G'(\app{i})\br[\Big]{\int_{\t{i}}^s G(\app{i}) \dWHr}\dWHs}.
    \end{align*}
    Here, $\t{j}=jh$, $0\le j \le M$, and $R=(R_h)_{h>0}$ is a time discretisation scheme approximating the semigroup $S$, i.e.\ $R_h \approx S(h)$.
    If $R=S$, it  
    is called the \emph{(exponential) Milstein scheme}.

    The Milstein scheme has been studied extensively for parabolic SPDEs. Convergence of the root-mean-square error at rate $1$ was shown in \cite{JentzenRoeckner15_Milstein} under a commutativity condition on the noise, which has subsequently been lifted in \cite{vonHallernRoessler20}. Further developments include its analysis for SPDEs driven by non-continuous martingale noise \cite{BarthLang12MilsteinNonCont}, its $L^p$- and almost sure convergence for advection-diffusion equations \cite{BarthLang13Milstein}, and the space-time-discretisation by a Milstein--Galerkin scheme \cite{Kruse2014MilsteinGalerkin}. Rational Milstein schemes have been formulated as mild Itô processes in \cite{mildItoFormula} and a software package for the numerical approximation of the iterated stochastic integrals  
    is available  \cite{KastnerRoessler23}. Finally, derivative-free Milstein-type schemes have been proposed \cite{ vonHallernRoessler23DerivativeFreeMilstein, WangGan13RungeKutta}, offering an easier implementation while preserving high convergence orders.

    However, all of these higher-order convergence rates pertain to the parabolic case, where regularisation phenomena of the underlying analytic semigroup can be used to improve convergence rates. These are not applicable to \eqref{eq:introSEE} if the associated semigroup is merely contractive. Establishing higher-order convergence of the Milstein scheme for hyperbolic stochastic evolution equations has already been raised as an open research direction over a decade ago in \cite[p.~324]{JentzenRoeckner15_Milstein}, together with the generalisation from exponential to rational Milstein schemes. Numerical simulations suggest convergence at rate $1$ of the Milstein scheme for stochastic Schrödinger equations \cite{LiLi20PhysRevE}. Recently, convergence rates of up to $3/2$ have been established for the stochastic wave equation with finite-dimensional noise  using a so-called $(\hat{\alpha},\beta)$-scheme (with $(\hat{\alpha},\beta)=(1,0)$)~\cite{FengPandaProhl24}. Further recent developments for hyperbolic SPDEs include \cite{BC23}, where rate $1$ was shown for a splitting scheme for the nonlinear Schrödinger equation with additive noise and in \cite{Cui25NLS} with multiplicative noise.

    Often in the literature, the error considered is the {\em pointwise strong error}
\begin{align*}
\max_{0 \le j\le M}\E\br[\big]{\|\sol{\t{j}}-\app{j}\|_X^p},
\end{align*}
where $U$ is the mild solution to \eqref{eq:introSEE} and $u$ its temporal discretisation. However, smallness of the pointwise strong error does not imply convergence of the path of the approximations or that numerical simulations converge to the mild solution. Hence, we investigate convergence rates of the {\em pathwise uniform strong error}
\begin{align*}
\E \br[\Big]{\max_{0 \le j\le M} \|\sol{\t{j}} - \app{j}\|_X^p}
\end{align*}
instead, where now the maximum over $j$ is inside the expectation. In general, the pathwise uniform error cannot be estimated by means of the pointwise strong error without deteriorating the convergence rate in the non-deterministic case. If the pathwise uniform strong error decays at rate $\alpha>0$ for all $p\in [2,\infty)$, a Borel--Cantelli argument yields \emph{almost sure convergence} at rate $\alpha-\varepsilon$ for all $\varepsilon\in (0,\alpha)$. That is, asymptotically, one has $\max_{0\le j \le M}\|\sol{\t{j}} - \app{j}\|_X\lesssim h^{\alpha-\varepsilon}$ almost surely.

\subsection*{Main result}

In the spirit of the Kato setting \cite{Kato75}, let $X$ and $Y$ be Hilbert spaces such that $Y\hra X$ continuously and let $(S(t))_{t\ge 0}$ be a $C_0$-semigroup on $X$. A time discretisation scheme $R\colon [0,\infty) \to \calL(X),~h\mapsto R_h \ce R(h)$ \emph{approximates $S$ to order $\alpha\in(0,1]$ on $Y$} if for all $T>0$
    \begin{equation*}
        \norm{(S(\t{j})-R_h^j)u}_X \lesssim_{T,\alpha} h^\alpha \norm{u}_Y
    \end{equation*}
    for all $u \in Y$, $h>0$, and $j \in \N$ such that $\t{j}=jh \in [0,T]$, where $R_h^j=(R_h)^j$. It is called \emph{contractive} on $X$ if $\norm{R_h}_{\calL(X)}\le 1$ for all $h>0$. Let $H$ be a Hilbert space and denote by $\calL_2(H,X)$ and $\calL_2^{(2)}(H,X)$ the spaces of linear and bilinear Hilbert--Schmidt operators, respectively. Our main result establishing convergence rates for the pathwise uniform error of the rational Milstein scheme is as follows.

\begin{theorem}
\label{thm:intro}
    Let $X,Y$ be Hilbert spaces and $\alpha \in (\frac{1}{2},1]$ such that $Y \hookrightarrow \dom(A^\alpha)$ continuously. Suppose that $-A$ generates a $C_0$-contraction semigroup $S=(S(t))_{t\ge 0}$ on both $X$ and $Y$.  Let $R=(R_\h)_{\h>0}$ be a time discretisation scheme that approximates the semigroup $S$ to rate $\alpha$ on $Y$ and that is contractive on both $X$ and $Y$.
    Suppose that $F$ and $G$ in \eqref{eq:introSEE} satisfy that
    \begin{itemize}
        \item $F\colon X \to X$ and $G\colon X \to \LHX$ are Lipschitz continuous and Gâteaux differentiable,
        \item $F\colon Y\to Y$ and $G\colon Y \to \LHY$ are of linear growth,
        \item the Gâteaux derivatives $F'\colon Y \to \calL(Y,X)$ and $G'\colon Y \to \calL(Y,\LHX)$ are $(2\alpha-1)$-Hölder continuous,
        \item $G'\circ G\colon X \to \bilinHX$ is Lipschitz continuous,
        \item and $G'\circ G\colon Y \to \bilinHY$ is of linear growth.
    \end{itemize}
    Let $p\in[2,\infty)$ and $\xi \in L^{2\alpha p}(\Omega;Y)$. Denote by $U$ the solution of \eqref{eq:introSEE} and by $u=(\app{j})_{j=0,\ldots,M}$ the rational Milstein scheme. 
    
    Then, for $M\ge 2$, there is a constant $C_T\ge 0$ independent of $\init$ and $h$ such that
    \begin{equation}
    \label{eq:mainEstIntro}
        \norm[\Big]{\max_{0 \le j \le M} \|\sol{\t{j}}-\app{j}\|_X }_{L^p(\Omega)} \le C_T (1+\|\init\|_{L^{2\alpha p}(\Omega;Y)})h^\alpha\sqrt{\log\p[\Big]{\frac{T}{h}}}.
    \end{equation} 
    In particular, the rational Milstein scheme converges at rate $\alpha$ up to a logarithmic correction factor as $h \to 0$. For the exponential Milstein scheme, \eqref{eq:mainEstIntro} also holds without the logarithmic 
    factor.
\end{theorem}

Theorem~\ref{thm:intro} is obtained as a special case of the main results in Theorem~\ref{thm:main} and, for the exponential Milstein scheme, Theorem~\ref{thm:expMilstein}. They also allow for $F$ and $G$ to be time-dependent and random. 
If the assumptions are satisfied for $p=2$, the same bound is obtained for the $r$-th moment, $r\in [1,2]$. For a suitable continuous-time extension $(\bar{u}_t)_{t\in [0,T]}$ of the rational Milstein scheme, an analogous estimate to \eqref{eq:mainEstIntro} holds for the supremum of $\sol{t}-\bar{u}_t$ over the full time interval $[0,T]$, cf. Theorem \ref{thm:extensionInterval}.
Common choices for $Y$ are suitable intermediate spaces between $X$ and $\dom(A^2)$ such as domains of fractional powers of $A$. 
In the special case that $Y = \dom(A)$, the optimal rate of convergence $1$ is achieved for the exponential Milstein scheme, and, if one can choose $Y=\dom(A^2)$, also for most commonly used rational Milstein schemes. 
Table \ref{tab:convrate} illustrates this dependence of the convergence rate on the choice of $Y$ in more detail.
\begin{table}[htb]
    \centering
    \renewcommand{\arraystretch}{1.2}
    \caption{Convergence rates $\alpha$ of the exponential and two rational Milstein schemes in case $Y = \dom(A^{\beta})$ in Theorem \ref{thm:intro} for some $\beta>0$}
     \label{tab:convrate}
    \begin{tabular}{@{}llll@{}}
      \toprule
      & Expon. Milstein & Implicit Euler Milstein & Crank--Nicolson Milstein \\
      \midrule
      scheme $R_h$ & $S(h)$ & $(1+hA)^{-1}$ &  $(2-hA)(2+hA)^{-1}$ \\
      \midrule
      rate $\alpha$ & $\beta\wedge 1$ & $\frac{\beta}{2}\wedge 1$ & $\frac{2\beta}{3}\wedge 1$ \\
      \bottomrule
    \end{tabular}
\end{table}

The implicit Euler and the Crank--Nicolson scheme are two common choices of rational schemes to approximate the semigroup, but Theorem \ref{thm:intro} is not limited to these. From functional calculus, it follows that any scheme $R_h=r(-hA)$ is admissible if it is induced by a rational, holomorphic function $r\colon\C_-\to\C$ satisfying $|r|\le 1$ on the left open half-plane $\C_-$ and approximating the exponential function (cf.\ Proposition \ref{prop:functionalcalculus}). 
In particular, this includes A-acceptable and consistent rational A-stable schemes such as the higher-order implicit Runge--Kutta methods Lobatto IIIA, IIIB, and IIIC, Radau methods, and certain DIRK schemes (see \cite[Sections IV.5 and IV.6]{HairerWanner1996} for definitions and properties of these schemes).

The error estimate \eqref{eq:mainEstIntro} is optimal in the sense that it matches the convergence rate of the initial-value term on its own, up to a logarithmic factor for rational Milstein schemes. Moreover, in the case of sufficient regularity, we achieve the optimal rate $1$, which coincides with the order of the Itô--Taylor expansion in the SDE case. 

To the best of the authors' knowledge, the present work provides the first rigorous error analysis of the Milstein scheme for hyperbolic SPDEs, both in terms of an abstract framework and for concrete model equations listed below. Our primary contributions are:
\begin{itemize}
    	\item first optimal convergence rates for Milstein schemes for \emph{hyperbolic SPDEs}, partially addressing an open problem raised in \cite[p.~324,~third~problem]{JentzenRoeckner15_Milstein} and \cite[Rem.~6.7]{KliobaVeraar24Rate},
    	\item treatment of \emph{rational Milstein schemes} based on rational semigroup approximations $R_h\neq S(h)$ such as the Crank--Nicolson Milstein scheme, answering the open problem posed for parabolic SPDEs in \cite[p.~324,~first~problem]{JentzenRoeckner15_Milstein} in the case of hyperbolic SPDEs, 
        \item \emph{maximal estimates} in \emph{$p$-th moment}, $p\in [2,\infty)$, leading to \emph{pathwise uniform} convergence rates rather than pointwise root-mean-square estimates,
        \item error estimates \textit{on the full time interval} for a suitable \textit{continuous-time extension} of Milstein-type schemes for hyperbolic SPDEs.
    \end{itemize}
For concrete equations like Schrödinger or Maxwell's equations, our results improve results from the literature for rational schemes to higher rates $\alpha\in (1/2,1]$ for Milstein schemes. Numerical simulations confirm these rates.

While our results also apply to additive noise, they are not novel in this case. This can be attributed to the observation that the Milstein scheme then reduces to the exponential Euler method or the corresponding rational scheme, for which convergence at rate $1$ was shown in \cite[Sec.~3]{KliobaVeraar24Rate}. Likewise, we do not achieve improved rates for wave equations compared to \cite[Thm.~7.6]{KliobaVeraar24Rate}, but our result covers nonlinearities and noise depending on both position and velocity components of the solution as opposed to just the position. In the latter case, rate $3/2$ was achieved for finite-dimensional noise in \cite{FengPandaProhl24} for a so-called $(\hat{\alpha},\beta)$-scheme (with $(\hat{\alpha},\beta)=(1,0)$), which seems out of reach in our setting, since the decay of the semigroup difference $\|[\sg{t}-\sg{s}]x\|_X$ is limited by $(t-s)^1$ regardless of the smoothness of $x$.

As shown in \cite{CHJNW}, pathwise uniform error bounds can be derived from pointwise ones via Hölder continuity of the $p$-th moment and the Kolmogorov--Chentsov theorem. However, for fixed integrability $p$ of the initial values, this decreases the convergence rate by $1/p$ 
\cite[Rem.~6.5]{KliobaVeraar24Rate}. 

To render the above results applicable to an implementable numerical scheme for SPDEs, a
spatial discretisation is required. As the main findings of the present work concern the temporal discretisation, we refer the interested reader to the literature on space discretisation, e.g.\ via spectral Galerkin methods as in our numerical simulation or \cite{JentzenRoeckner15_Milstein, KamBlo}, finite differences \cite{CQS16,  GyMi09}, finite elements~\cite{CLS13, KLP20} or a discontinuous Galerkin approach \cite{BLM21, HHS22}.

For nonlinear Nemytskii operators $F$ and $G$, the linear growth condition on $Y$ limits one's choice of $Y$ to $H^1$, prohibiting optimal convergence rates for second order equations with Nemytskii-type nonlinearities. However, optimal rates can be achieved for some non-Nemytskii nonlinearities (cf.~Subsection~\ref{subsec:schroedingerConv}). Relaxing the framework to allow for polynomial growth on $Y$ would overcome this restriction and is left for future work. 
The analysis of the Milstein scheme in the globally Lipschitz setting constitutes a first step towards treating locally Lipschitz nonlinearities, which occur more frequently in model equations. A systematic error analysis of uniform strong errors for rational schemes in the locally Lipschitz setting would then provide a natural starting point for the analysis of Milstein schemes, both of which remain open.

\subsection*{Method of proof}

We extend Kato's framework \cite{Kato75}, which was first used for SPDEs in \cite{KliobaVeraar24Rate}, to systematically treat hyperbolic problems beyond convergence rate $1/2$. This entails working with a pair of Hilbert spaces $Y\hra X$ such that $F$ and $G$ in \eqref{eq:introSEE} are, as in Theorem \ref{thm:intro}, Lipschitz continuous on $X$ and of linear growth on $Y$. Lipschitz continuity on higher order Sobolev spaces typically fails for Nemytskii operators and is thus not assumed. Although some parabolic equations are covered by the Kato setting, alternative approaches based on regularisation phenomena yield better results. Consequently, we focus on hyperbolic problems. 

Achieving higher-order convergence relies on suitable Taylor expansions, which require Gâteaux differentiability and continuity of the derivative, which in turn imply Fréchet differentiability. Care must be taken in specifying between which spaces this differentiability assumption is imposed. Indeed, assuming Fréchet differentiability of $F\colon X\to X$ for $F$ a Nemytskii operator on $X=L^2$ immediately limits the setting to affine-linear `nonlinearities' $F$. More precisely, consider the Nemytskii operator $F\colon L^2(\calO)\to L^2(\calO), F(u)\ce \phi \circ u$ associated to a Lipschitz continuous $C^1$-function $\phi\colon \R\to\R$, and $\calO\seq \R^d$ a bounded domain. If $F$ is Fréchet differentiable at some $\tilde{u}\in L^2(\calO)$, there exist $a,b\in \R$ such that $F(u)=au+b$ for all $u\in L^2(\calO)$ (see also Proposition \ref{prop:NemytskiiGateauxAffine}).

This was recently pointed out as a prevalent problem in the numerical analysis literature \cite[Rem.~1.3]{DjurdjevacGerencserKremp24}. We avoid this pitfall by merely assuming Gâteaux differentiability of $F$ as a map on $X$. This is satisfied for any $\phi$ as above and, together with Lipschitz continuity of $F$ on $X$, ensures uniform boundedness of the Gâteaux derivative $F'\colon X \to \calL(X)$. 
Furthermore, we assume Hölder continuity of the Gâteaux derivative $F'\colon Y \to \calL(Y,X)$, which ensures the existence of a suitable form of Taylor expansion.
The latter also implies Fréchet differentiability of $F\colon Y \to X$. However, since this only requires controlling the norm of the derivative of $F$ uniformly over a unit ball in $Y$ rather than $X$, genuinely nonlinear Nemytskii operators are admissible provided that $Y$ is sufficiently regular.
For instance, in one spatial dimension, $Y=H^s$ for $s>\frac{1}{2}$ suffices due to the embedding into $L^\infty(\calO)$. One then has the Taylor expansion
\begin{align*}
    F(u) &= F(v) + F'(v)\br{u-v}+ \int_{0}^{1} \big( F'(v+\zeta(u-v))-F'(v) \big)\br{u-v} \dzeta
\end{align*}
for all $u,v\in Y$. Together with Lipschitz continuity of $F'\colon Y \to \calL(Y,X)$, which implies Fréchet differentiability of $F\colon Y \to X$, this suffices to prove convergence at rate $1$. Contrary to \cite{JentzenRoeckner15_Milstein}, we do not assume $F$ and $G$ to be twice Fréchet differentiable.

To consider pathwise uniform errors in $L^p(\Omega)$ for general $p\in [2,\infty)$ instead of root-mean-square errors (i.e.\ $p=2$) only as in \cite{JentzenRoeckner15_Milstein}, a stochastic Fubini argument is employed in one of the terms arising from the Taylor expansion and maximal inequalities for stochastic convolutions are used. A crucial ingredient that enables the extension from exponential to rational Milstein schemes is a logarithmic square function estimate (\cite[Prop.~2.3]{KliobaVeraar24Rate} and \cite[Thm.~3.1]{coxVanWInden2024logSquare}) of the form 
\begin{equation*}
    \E \br[\bigg]{\sup_{i\in \{1, \ldots, n\}} \sup_{t\geq 0}\Big\|\int_0^t  \Phi_i(s) \dWHs\Big\|_X^p} \lesssim \sqrt{\log(n)^p} \|(\Phi_i)_{i=1}^n\|^p
\end{equation*}
for a suitable square function norm of $\Phi$ (see Proposition \ref{prop:log-maximal-inequality} below).

Furthermore, regularity estimates are needed. It is only possible to show $1/2$-Hölder continuity of the $p$-th moment of the mild solution $U$ to \eqref{eq:introSEE} in $X$ but not in $Y$. A different splitting of the error permits us to circumvent this: Rather than $(\E[\|\sol{t}-\sol{s}\|_X^p])^{1/p}$ for $s\le t$, we estimate differences with the semigroup $(\E[\|\sol{t}-\sg{t-s}\sol{s}\|_X^p])^{1/p}$ in terms of $(t-s)^{1/2}$. Since some terms then vanish, the remaining terms can be estimated via linear growth and thus also in $Y$, as illustrated in Lemma~\ref{lem:regularity-UtDiffVst-Z}. On $X$, this decay can be improved to $(t-s)^1$ by additionally considering the difference with the stochastic integral term $\int_s^t G(\sg{t-s}\sol{s})\dWHr$ (cf.\ Lemma \ref{lem:regularity-DiffWithStInt-X}). The proof is finished by an application of a discrete Grönwall inequality.

\subsection*{Overview}
Section \ref{sec:prelim} recalls some facts from stochastic integration, Fréchet derivatives, and semigroup approximation required subsequently. We introduce our setting and assumptions in full generality in Section \ref{sec:settingAssumptions}. Section \ref{sec:WPStability} then recalls the underlying well-posedness results from the literature and contains a proof of pointwise strong stability of the Milstein scheme. Pathwise uniform convergence rates for the Milstein scheme are stated and proven in Section~\ref{sec:convergence}, where the main result in Theorem~\ref{thm:main} extends Theorem~\ref{thm:intro} above. Improvements for the exponential Milstein scheme in Theorem~\ref{thm:expMilstein} and the linear case in Corollary~\ref{cor:linearCase}, an extension to the full time interval in Theorem~\ref{thm:extensionInterval} as well as possible generalisations are discussed. Our results are illustrated for the stochastic Schrödinger, Maxwell's, and transport equations in Section~\ref{sec:examples}. Numerical simulations for different versions of the stochastic Schrödinger equation in Section~\ref{sec:simulations} validate our theoretical findings.

\subsubsection*{Acknowledgements}

The authors thank Foivos Evangelopoulos-Ntemiris, Mark Ver\-aar, and Joris van Winden for helpful discussions and comments. They would also like to thank the anonymous referees for helpful comments that led to improving the manuscript and adding Theorem \ref{thm:extensionInterval}.

\section{Preliminaries}
\label{sec:prelim}

\subsection*{Notation}

Denote the natural numbers by $\N\ce \{1,2,3,\ldots\}$. 
Throughout the paper, $H$, $X$, and $Y$ are separable Hilbert spaces, $(\Omega, \F, \PP)$ is a fixed probability space with filtration $(\F_t)_{t \in [0,T]}$ satisfying the usual conditions, and $(W_t)_{t\ge 0}$ denotes an $H$-cylindrical Brownian motion. 
For $p\in[2,\infty)$, we abbreviate by $\norm{\cdot}_p$ the canonical norm in $L^p(\Omega)$ and write $\calB(X)$ for the Borel $\sigma$-algebra of $X$ and $\calP$ for the predictable $\sigma$-algebra on $\Omega\times[0,T]$.
We use the subscript $L_{\mathcal{G}}^p$ to denote $\mathcal{G}$-measurable elements or processes in $L^p$.
Denote by $C^{\alpha}(I;X)$, or simply $C^{\alpha}(I)$ if $X=\R$, the space of bounded and $\alpha$-Hölder continuous functions $\phi\colon I\to X$ for $\alpha \in (0,1]$ and $I\seq \R$.
The notation $f(x) \lesssim_{a,b} g(x)$ is used if there is a constant $C \ge 0$ depending on $a,b$ such that for all $x$ in the respective set $f(x) \le C g(x)$.

Fix a final time $T>0$. We  consider a uniform time grid $\{\t{j} = jh:j=0,\ldots,M\}$ on $[0,T]$ with $M\in \N$ steps and time step size $h\ce T/M>0$. For $t \in [0,T]$, the time grid point before $t$ is given by $\lfloor t \rfloor \ce \max\{t_j:\,t_j\le t\}$. We approximate the exact solution $\sol{}=(\sol{t})_{t\in [0,T]}$ of a given evolution equation governed by a $C_0$-semigroup $S=(S(t))_{t\geq 0}$ by a numerical solution $\app{}=(\app{j})_{j=0,\ldots,M}$ given by the Milstein scheme. The computation of $\app{}$ relies on a time discretisation scheme $R=(R_h)_{h>0}$ that approximates the semigroup $S$.

\subsection{Stochastic integration in Hilbert spaces}
\label{subsec:stochasticIntegration}
By $\calL_2(H,X)$, we denote the space of Hilbert--Schmidt operators from $H$ to $X$ and by $\bilinHX$ the space of bilinear Hilbert--Schmidt operators, which consists of all bilinear bounded operators $\Phi \colon H \times H \to X$ such that 
\begin{equation*}
    \|\Phi\|_\bilinHX \ce \p[\Big]{\sum_{m \in \N} \sum_{n\in\N} \|\Phi(h_m,h_n)\|_X^2}^{1/2}<\infty,
\end{equation*}
where $(h_n)_{n\in\N}$ is an orthonormal basis of $H$. 
It is easily checked that $\bilinHX \cong \calL_2(H;\LHX)$ is an isometric isomorphism and thus also $\bilinHX \cong \calL_2(\overline{H\otimes H},X)$ \cite[Thm.~9.4.10, Prop.~9.1.9, and Ex.~7.5.2]{AnalysisBanachSpacesII}. Here, $\overline{H\otimes H}$ denotes the Hilbert space tensor product, which is given by the completion of the algebraic tensor product with respect to the canonical inner product in $H\otimes H$.

For $\phi \in \LHX$ and a sequence $\gamma = (\gamma_n)_{n\in\N}$ of centred i.i.d.\ normally distributed random variables we define
\begin{equation}\label{eq:convradonW}
    \phi \gamma \ce \sum_{n\in\N} \gamma_n \phi h_n,
\end{equation}
where the convergence is in $L^p(\Omega;X)$ for every $p \in [1,\infty)$ and almost surely (see \cite[Cor.~6.4.12]{AnalysisBanachSpacesII}). 

We recall some fundamental definitions and statements on stochastic integration in Hilbert spaces from~\cite{DaPratoZabczyk14}. Given an $\calL_2(H,X)$-valued integrand, we take stochastic integrals w.r.t.\ an \emph{$H$-cylindrical Brownian motion} as integrator, which is a bounded linear mapping $W_H\colon L^2(0,T;H) \to L^2(\Omega)$ such that
\begin{enumerate}[label=(\roman*)]
    \item $W_H b$ is centred Gaussian for all $b \in L^2(0,T;H)$,
    \item $\E[W_H b_1 \cdot W_H b_2] = \langle b_1,b_2 \rangle_{L^2(0,T;H)}$ for all $b_1, b_2 \in L^2(0,T;H)$,
    \item $W_H b$ is $\F_t$-measurable for all $b\in L^2(0,T;H)$ supported in $[0,t]$,
    \item $W_H b$ is independent of $\F_s$ for all $b\in L^2(0,T;H)$ supported in $[s,T]$.
\end{enumerate}
A complex $H$-cylindrical Brownian motion is defined analogously with a complex conjugate on $W_H b_2$. For each fixed $h\in H$ of unit norm, $(W_H(t)h)_{t \in[0,T]}\ce (W_H(\ind_{(0,t)} \otimes h))_{t \in[0,T]}$ is a  (standard) Brownian motion. In the special case $H=\R$, we recover the classical real-valued Brownian motion. 

Given a linear, bounded, positive self-adjoint operator $Q\in \calL(H)$ of trace class, coloured noise can be described with the help of a \emph{$Q$-Wiener process}.  
This is an equivalent description in the sense that $W_Q$ is a $Q$-Wiener process if and only if $Q^{1/2}W_H(t)\ce\sum_{n\geq 1} Q^{1/2} h_n W_H(t) h_n = W_Q(t)$ for some $H$-cylindrical Brownian motion $W_H$ for all $t\in[0,T]$, where we used \eqref{eq:convradonW}. One can reduce the study of~\eqref{eq:introSEE} with a $Q$-Wiener process $W_Q$ to the case of cylindrical Brownian motion by replacing $G$ by $G Q^{1/2}$. Subsequently, we will omit the index $H$ from $W_H$ for the sake of readability.
We now recall a standard property of stochastic integrals from \cite[Prop.~4.30]{DaPratoZabczyk14}.

\begin{lemma}
\label{lem:linearityHSstochInt}
    Let $p \in [2,\infty)$, $0\le a<b\le T$, $Z_1$ and $Z_2$ be Hilbert spaces, and let $(W_t)_{t\ge 0}$ be an $H$-cylindrical Brownian motion. Further, let $B\in L^\infty(\Omega; \calL(Z_1,Z_2))$ be $\F_a$-measurable and $\phi\in L^p(\Omega;L^2(0,T;\calL_2(H,Z_1)))$ be progressively measurable. 
    Then
    \begin{equation*}
        B\Big(\int_a^b \phi(r) \dWHr\Big) = \int_a^b (B \circ \phi)(r) \dWHr \quad\text{ in }L^p(\Omega;Z_2), 
    \end{equation*}
    where $B\circ \phi \in L^p(\Omega;L^2(0,T;\calL_2(H,Z_2)))$ is given by $(B\circ \phi)(\omega,r)h \ce B(\omega)(\phi(\omega,r)h)$.
\end{lemma}

A central estimate for stochastic integrals is the following maximal inequality for stochastic convolutions with a quasi-contractive semigroup.

\begin{definition}
\label{def:quasiContractive}
    A $C_0$-semigroup $(S(t))_{t \ge 0}$ on $X$ is called \emph{quasi-contractive} with parameter $\lambda\ge 0$ if $\|S(t)\|_{\calL(X)} \le e^{\lambda t}$ for all $t \ge 0$ and \emph{contractive} if this holds with $\lambda=0$.
\end{definition}

\begin{theorem}
    \label{thm:maximal-inequality}
    Let $(\sg{t})_{t\geq0}$ be a quasi-contractive semigroup on $X$ with parameter $\lambda\geq0$.
    Then for $p\in[2,\infty)$
    \begin{equation*}
        \norm[\bigg]{ \sup_{t\in[0,T]} \norm[\Big]{ \int_{0}^{t} \sg{t-s}g(s) \dWHs }_X }_{L^p(\Omega)}
        \le e^{\lambda T} B_p \norm{g}_{L^2(0,T;L^{p}(\Omega;\LHX))},
    \end{equation*}
    where one can take $B_2=2$ and $B_p=4\sqrt{p}$ for $p\in(2,\infty)$.
\end{theorem}

If the semigroup considered is the identity, the classical \emph{Burkholder--Davis--Gundy inequalities} are recovered. In this case, the inequality remains valid in the absence of a supremum, and is referred to as \emph{Itô's isomorphism}.
\begin{proof}
    The contractive case with the norm of $g$ in $L^p(\Omega;L^2(0,T;\LHX))$ follows from \cite{HausSei}. Via a scaling argument, this can be extended to the quasi-contractive case. The admissibility of the constants is explained in \cite[p.~2068]{KliobaVeraar24Rate}. Lastly, noting that $\frac{p}{2}\ge 1$, Minkowski's integral inequality allows us to further estimate the norm of $g$ by
    \begin{align*}
        \|&g\|_{L^p(\Omega;L^2(0,T;\LHX))} \le \norm{g}_{L^2(0,T;L^{p}(\Omega;\LHX))}. \qedhere
    \end{align*} 
\end{proof}

Moreover, the following logarithmic square function estimate is an essential tool in the convergence estimate of the rational Milstein scheme. Its original version \cite[Prop.~2.7]{vNeeVer-Pin} with $\log(M)$-scaling was improved to $\sqrt{\log(M)}$-scaling in \cite[Prop.~2.3]{KliobaVeraar24Rate} and generalised in \cite[Thm.~3.1]{coxVanWInden2024logSquare}.
\begin{proposition}
    \label{prop:log-maximal-inequality}
    Let $p\in[2,\infty)$ and $M\ge 2$.
    Let $\Phi \ce \p{\Phi^{(j)}}_{j=1}^{M}$ be a finite sequence in\\ $L_\calP^p(\Omega; L^2(0,T; \LHX))$.
    Then with $K=4\exp(1+\frac{1}{2\mathrm{e}})$ it holds that
    \begin{align*}
        \norm[\bigg]{ &\sup_{t\in[0,T], j\in\{1,\dots,M\}} \norm[\Big]{ \int_{0}^{t} \Phi^{(j)}_s \dWHs }_X }_{L^p(\Omega)}\leq K \sqrt{\max\{\log(M), p\}}\|\Phi\|_{L^p(\Omega;\ell_M^\infty(L^2(0,T;\LHX)))}.
    \end{align*}
\end{proposition}

\subsection{Gâteaux and Fréchet differentiability}
\label{subsec:gateauxFrechet}
The analysis of the Milstein scheme relies on both Gâteaux and Fréchet derivatives, whose definition we recall. 
In the following, let $Z_1$ and $Z_2$ be two arbitrary Banach spaces.

\begin{definition}
    For an open subset $O\seq Z_1$, an operator $B\colon O\to Z_2$ is said to be \emph{Gâteaux differentiable at $x \in O$} if there exists $R\in \calL(Z_1,Z_2)$ such that $\frac{1}{\varepsilon}(B(x+\varepsilon z)-B(x))\to Rz$ in $Z_2$ as $\varepsilon\to 0$ for all $z\in Z_1$. If this holds for all $x\in O$, $B$ is \emph{Gâteaux differentiable on $O$} and $R$ is called the \emph{Gâteaux derivative of $B$ at $x$} denoted by $B'\colon O \to \calL(Z_1,Z_2), x \mapsto B'(x)$ with $B'(x)\colon  y \mapsto B'(x)[y]$.   
    We say that $B$ is \emph{Fréchet differentiable at $x \in O$} if there exists $R\in \calL(Z_1,Z_2)$ such that
    \begin{equation*}
        \lim_{\|y\|_{Z_1}\to 0} \frac{\norm{B(x+y)-B(x)-Ry}_{Z_2}}{\|y\|_{Z_1}}=0.
    \end{equation*}
    If the above holds for all $x\in O$, we say that \emph{$B$ is Fréchet differentiable on $O$} and write $B\in C^1(O,Z_2)$.
\end{definition}

Both Gâteaux and Fréchet derivatives are unique whenever they exist by linearity and uniqueness of limits in $Z_2$. We point out that the limit in the definition of Gâteaux differentiability only involves the norm on the codomain $Z_2$, whereas for the Fréchet derivative, also the norm on the domain $Z_1$ appears. By setting $y=\varepsilon z$ in the Fréchet derivative, one immediately sees that every Fréchet differentiable operator is Gâteaux differentiable and the derivatives agree. The converse is false in general, since only directional derivatives are considered for Gâteaux differentiability and no uniform norm control is required. However, under a continuity assumption on the Gâteaux derivatives, Fréchet differentiability is assured.

\begin{proposition}[Theorem 1.9 in \cite{Ambrosetti1995}]
\label{prop:GateauxContIsFrechet} 
    Suppose that $B\colon O\seq Z_1 \to Z_2$ is Gâteaux differentiable on $O$ and $B'\colon  O \to \calL(Z_1,Z_2)$ is continuous in some $x \in O$. Then $B\colon O\seq Z_1 \to Z_2$ is Fréchet differentiable at $x$ and the Fréchet derivative at $x$ coincides with the Gâteaux derivative $B'(x)$. 
\end{proposition}

Nevertheless, Fréchet differentiability is not required to obtain uniform bounds on the Gâteaux derivative. Indeed, Lipschitz continuity is sufficient.

\begin{lemma}
    \label{lem:bounded-deriv}
    Suppose that $B\colon Z_1 \to Z_2$ is Lipschitz continuous and Gâteaux differentiable on a nonempty subset $O\seq Z_1$.
    Then the Gâteaux derivative $B'\colon O\to\calL(Z_1,Z_2)$ is uniformly bounded by the Lipschitz constant of $B$. 
\end{lemma}
\begin{proof}
    Denote the Lipschitz constant of $B$ by $L>0$ and let $x\in O$. The Gâteaux derivative in $x\in O$ can be bounded independently of $x$ by
    \begin{align*}
        \norm{B'(x)}_{\calL(Z_1,Z_2)} &= \sup_{\norm{y}_{Z_1}=1} \lim_{\varepsilon\to0} \frac{\norm{B(x+\varepsilon y)-B(x)}_{Z_2}}{\abs{\varepsilon}}
        \leq \sup_{\norm{y}_{Z_1}=1} \lim_{\varepsilon\to0} \frac{L\norm{\varepsilon y}_{Z_1}}{\abs{\varepsilon}}
        = L.\qedhere
    \end{align*}       
\end{proof}

The case of Nemytskii operators $B$ on $Z_1=Z_2=L^2(\calO;\R)$ for a bounded domain $\calO \seq \R^d$, $d\in \N$, illustrates that Fréchet differentiability on all of $L^2$ is a restrictive assumption.

\begin{proposition}[Theorem 2.7 and Proposition 2.8 in \cite{Ambrosetti1995}]
\label{prop:NemytskiiGateauxAffine}
    Let $d\in \N$, $\calO \seq \R^d$ a bounded domain, and $\phi\colon \calO \times \R\to\R$ be such that $\phi(x,\cdot)$ is continuously differentiable for almost all $x\in \calO$ with uniformly bounded partial derivatives $|\partial_s\phi(x,s)|\le C$ for almost all $x\in\calO$ and all $s\in \R$ and assume that $\phi(\cdot,s)$ and $\partial_s\phi(\cdot,s)$ are measurable. Further, let $B\colon L^2(\calO;\R)\to L^2(\calO;\R), (B(u))(x)\ce \phi(x,u(x))$ for $x \in \calO$ be the corresponding Nemytskii operator. Then $B\colon L^2(\calO;\R) \to L^2(\calO;\R)$ is Gâteaux differentiable. Moreover, if $B\colon L^2(\calO;\R)\to L^2(\calO;\R)$ is Fréchet differentiable at some $\tilde{u}\in L^2(\calO;\R)$ then it is affine linear. That is, there are measurable functions $a,b\colon \calO \to \R$ such that for every $u \in L^2(\calO;\R)$, $(B(u))(x)=a(x) u(x) +b(x)$ for almost all $x\in\calO$.
\end{proposition}

In case $\phi(x,s)\equiv \phi(s)$, the conditions on $\phi$ in the proposition above reduce to $\phi\in C^1(\R)$ being Lipschitz continuous. Replacing $Z_1$ by a subspace with a stronger norm, Fréchet differentiability no longer implies that $B$ is affine linear.

\subsection{Approximation of semigroups}
\label{subsec:approxSemigroup}

A key step in solving a stochastic evolution equation numerically is to approximate the associated semigroup describing the behaviour of the linear part.

\begin{definition}
\label{def:orderScheme}
Let $Y\hra X$ continuously and let $S=(S(t))_{t \ge 0}$ be a $C_0$-semigroup on $X$. A \emph{time discretisation scheme}, or simply \emph{scheme}, is a function $R\colon  [0,\infty) \to \calL(X),~h\mapsto R_h \ce R(h)$. It is said to \emph{approximate $S$ to order $\alpha>0$ on $Y$}  if for all $T>0$ there is a constant $C_\alpha \ge 0$ such that
\begin{equation*}
    \|(S(\t{j})-R_h^j)u\|_X \le C_\alpha h^\alpha\|u\|_Y
\end{equation*}
for all $u \in Y$, $h>0$, and $j \in \N$ such that $\t{j}=jh \in [0,T]$.
Here, $R_h^j = (R_h)^j$.
Equivalently, we say that \emph{$R$ converges of order $\alpha$ on $Y$}. 
An $\calL(Z)$-valued scheme $R$ is called \emph{contractive on $Z$} for $Z\hra X$ continuously if $\|R_h\|_{\calL(Z)} \le 1$ for all $h \ge 0$.
\end{definition}

Common choices of schemes approximating the semigroup $S$ generated by $-A$ with their respective approximation orders include (cf.\ Corollary 4.4 in \cite{Kovacs07} with $q=1$ for IE and $q=2$ for CN):

\begin{itemize}
\item exponential Euler (EXE): $R_h = S(h)$, any order $\alpha >0$ on $X$;
\item implicit Euler (IE): $R_h = (1+hA)^{-1}$, order $\alpha \in (0,1]$ on $\dom(A^{2\alpha})$;
\item Crank--Nicolson (CN): $R_h = (2-hA)(2+hA)^{-1}$, order $\alpha \in (0,2]$ on $\dom(A^{3\alpha/2})$ provided that $S$ is contractive.
\end{itemize}

An essential assumption in our hyperbolic setting consists of contractivity of the semigroup and of the scheme approximating it. For EXE, contractivity of the scheme clearly is equivalent to contractivity of the semigroup. 
Both IE and CN are instances of \emph{rational schemes}, which are schemes $R$ that are induced by a rational function $r\colon \C\to\C\cup \{\infty\}$ via $R_h=r(h\cdot(-A))$ for all sufficiently small $h>0$. This definition uses the bounded $H^\infty$-calculus of $A$ as the negative generator of a $C_0$-contraction semigroup to define $R_h$. Under additional assumptions on $r$ on the open left half-plane $\C_-$, contractivity of the scheme $R$ is recovered as a consequence of \cite[Thm.~10.2.24]{AnalysisBanachSpacesII}.

\begin{proposition}
\label{prop:functionalcalculus}
    Let $-A$ generate a $C_0$-contraction semigroup on $X$. For a holomorphic function $r\colon \C_-\to \C$ satisfying $\abs{r(z)}\leq 1$ for all $z \in \C_{-}$, let $R_h = r(-hA)$ for $h>0$. Then $R$ is contractive.
\end{proposition}

This proposition ensures the contractivity of IE, CN, but also higher-order implicit Runge--Kutta methods including Radau methods, Lobatto IIIA, IIIB, and IIIC as well as some DIRK schemes. More generally, all A-stable schemes \cite{DahlquistAstable} and all A-acceptable schemes \cite{brenner1979rational} satisfy the condition above on $r$, where the latter require it to hold on $\overline{\C_-}$ and an additional consistency condition. As is the case for these examples, a common choice for the spaces $Y$ on which the semigroup $S$ is approximated to some rate depending on $\alpha>0$ consists of domains $\dom(A^\alpha)$ of fractional powers of $A$. It is useful to know that the embedding $\dom(A^{\alpha}) \hookrightarrow \dom_A(\alpha, \infty) \ce (X,\dom(A))_{\alpha,\infty}$
into the real interpolation space with parameter $\infty$ holds. For details on interpolation spaces, we refer the interested reader to \cite{Lun,Tr1}.

\subsection{A discrete Grönwall inequality}
The following variant of a discrete Grönwall inequality \cite[Lem.~2.7]{KliobaVeraar24Rate} based on \cite[Lem.~A.3]{Kruse2014book} proves to be useful.
\begin{lemma}\label{lem:gronwall}
    Let $(\varphi_j)_{j \ge 0}$ be a non-negative sequence, $M\in \N \cup \{\infty\}$, and $\alpha,\beta\in [0,\infty)$ constants. Suppose that for $j =0,\ldots,M$
    \[
    \varphi_j \le \alpha + \beta\bigg(\sum_{i=0}^{j-1} \varphi_i^2\bigg)^{1/2}.
    \]
    Then for $j =0,\ldots,M$
    \[\varphi_j \le \alpha (1+\beta^2 j)^{1/2} \exp\Big(\frac{1+\beta^2j}{2}\Big).\]
\end{lemma}

\section{Setting and Assumptions}
\label{sec:settingAssumptions}

We consider the stochastic evolution equation
\begin{equation}\label{eq:SEE}
	\begin{cases}
		\rmd U + AU \dt &= F(t,U)\dt + G(t,U) \dW\quad \text{on }(0,T], \\
		\sol{0}&=\init \in L_{\F_0}^p(\Omega;X).
	\end{cases}
\end{equation}
on a Hilbert space $X$ for some $p \in [2,\infty)$ and $T>0$. Here, $-A$ generates a contractive $C_0$-semigroup, $F$ is a nonlinearity, $G$ multiplicative noise, and $(W_t)_{t\ge 0}$ is an $H$-cylindrical Brownian motion for some separable Hilbert space $H$. 
Assumptions \ref{ass:spacesSemigroup}, \ref{ass:FG_rate_I}\ref{assCond:FG_rate_Lipschitz}, and \ref{ass:FG_rate_I}\ref{assCond:FG_rate_Yinvariance} ensure that \eqref{eq:SEE} is well-posed (see Theorem \ref{thm:wellposedness-Z} for the precise statement).

\begin{assumption}[Spaces and semigroup]
    \label{ass:spacesSemigroup}
    Let $X$, $Y$, and $H$ be separable Hilbert spaces such that $Y \hra X$ continuously.  
    Let $-A\colon \dom(A)\seq X\to X$ be the generator of a $C_0$-contraction semigroup $S=(\sg{t})_{t \ge 0}$ on both $X$ and $Y$. 
    Let $\alpha \in (0,1]$ and suppose that $Y \hookrightarrow \dom_A(\alpha,\infty)$ continuously if $\alpha \in (0,1)$ or $Y \hookrightarrow \dom(A)$ continuously if $\alpha=1$, with embedding constant $\CY\ge 0$. 
\end{assumption}

The embedding of $Y$ into a real interpolation space allows us to obtain decay rates for semigroup differences via interpolation \cite[Formula~(2.4)]{KliobaVeraar24Rate}.
\begin{lemma}
    \label{lem:semigroup-interpolation}
    Under Assumption \ref{ass:spacesSemigroup}, we have
    \begin{equation*}
        \norm{\sg{t}-\sg{s}}_{\calL(Y,X)} \leq 2\CY (t-s)^\alpha,\qquad 0 \le s \le t.
    \end{equation*}
\end{lemma}

The conditions \ref{assCond:FG_rate_Lipschitz}--\ref{assCond:FG_rate_lineargrowthY} of the following assumption are a slight modification of \cite[Ass.~6.1]{KliobaVeraar24Rate} for schemes without Milstein terms. 
\begin{assumption}[Nonlinearities I]
    \label{ass:FG_rate_I}
    Let $q\in [2,\infty)$ and $\alpha \in (0,1]$. 
    Let $F\colon\Omega \times [0,T] \times X \to X$ and $G\colon\Omega \times [0,T] \times X \to \LHX$ be strongly $\calP \otimes \calB(X)$-measurable, define $f\colon\Omega \times [0,T]\to X$ by $f \ce F(\cdot,\cdot,0)$ as well as $g\colon\Omega \times [0,T]\to \LHX$ by $g \ce G(\cdot,\cdot,0)$, and suppose that
    \begin{enumerate}[label=(\alph*)]
        \item\label{assCond:FG_rate_Lipschitz} \emph{(global Lipschitz continuity on $X$)} there exist constants $\LFX, \LGX \ge 0$ such that for all $\omega \in \Omega, t \in [0,T]$, and $x,y\in X$,
        \begin{align*}
            \|F(\omega,t,x)-F(\omega,t,y)\|_X &\le \LFX\|x-y\|_X,\quad
            \|G(\omega,t,x)-G(\omega,t,y)\|_\LHX \le \LGX\|x-y\|_X,
        \end{align*}
        \item \label{assCond:FG_rate_HolderContTime}\emph{(temporal $\alpha$-Hölder continuity)} it holds that
        \begin{align*}
            C_{\alpha,F,q}\ce \sup_{0\leq s<t\leq T} \frac{1}{(t-s)^\alpha}\norm[\Big]{\sup_{x\in X} \norm{F(\cdot,t,x)-F(\cdot,s,x)}_X}_q <\infty,\\
            C_{\alpha,G,q}\ce \sup_{0\leq s<t\leq T} \frac{1}{(t-s)^\alpha}\norm[\Big]{\sup_{x\in X} \norm{G(\cdot,t,x)-G(\cdot,s,x)}_\LHX}_q <\infty,
        \end{align*}
                
        \item \label{assCond:FG_rate_Yinvariance} \emph{($Y$-invariance)} $F\colon \Omega \times [0,T] \times Y \to Y$ and $G\colon \Omega \times [0,T] \times Y \to \LHY$ are strongly $\calP\otimes \calB(Y)$-measurable and $f \in C([0,T];L^q(\Omega; Y))$ as well as $g \in C([0,T];L^q(\Omega;\LHY))$,
        
        \item \emph{(linear growth on $Y$)} \label{assCond:FG_rate_lineargrowthY} there exist constants $\LFY, \LGY \ge 0$ such that for all $\omega \in \Omega, t \in [0,T]$, and $x\in Y$, it holds that
        \begin{align*}
            \|F(\omega,t,x)-f(\omega,t)\|_Y &\le \LFY(1+\|x\|_Y),\quad
            \|G(\omega,t,x)-g(\omega,t)\|_\LHY \le \LGY(1+\|x\|_Y),
        \end{align*}
        
        \item \emph{(partial Gâteaux differentiability)} \label{assCond:FG_rate_GateauxDiffble} and for all $\omega \in \Omega$ and $t \in [0,T]$, the maps $F(\omega,t,\cdot)\colon X\to X$ and $G(\omega,t,\cdot)\colon X\to\LHX$ are Gâteaux differentiable. 
    \end{enumerate}
\end{assumption}

\begin{notation}
\label{not:GpG} 
    \begin{enumerate}[label=(\roman*)]
    \item When there is no risk of confusion, we omit the explicit dependence on $\omega$.
    \item We abbreviate $\CalphaF\ce C_{\alpha,F,p}$ and $\CalphaG\ce C_{\alpha,G,p}$.
    \item For Hilbert spaces $Z_1,Z_2$, the Gâteaux derivative of $B\colon\Omega \times [0,T] \times Z_1 \to Z_2$ at $u\in Z_1$ is denoted by $B'(t,u)\in \calL(Z_1,Z_2)$, where the dependence on $\omega$ is omitted from the notation. It corresponds to the map $ v \mapsto B'(t,u)[v]$. In our setting, we thus consider $F'\colon \Omega\times [0,T]\times X \to \calL(X)$ and $G'\colon \Omega\times [0,T]\times X \to \calL(X,\LHX)$. 
    \item The shorthand notation $G'G\colon \Omega\times[0,T]\times X\to \bilinHX$ is used to denote the composition $(G'G)(\omega,t,u) \ce G'(\omega,t,u) \circ G(\omega,t,u)$ for $\omega\in\Omega$, $t\in[0,T]$, and $u\in X$. 
    It can be easily checked that for every such $(\omega,t,u)$ this defines a bilinear Hilbert--Schmidt operator 
    \begin{equation*}
            (G'G)(\omega,t,u)\colon H\times H \to X,
            (h_1,h_2) \mapsto G'(\omega,t,u)\br{G(\omega,t,u)h_1}h_2.
    \end{equation*}
    \item Analogously, we write $F'(\omega,t,u)[\Phi]\ce F'(\omega,t,u)\circ \Phi \in \calL_2(H,X)$ for $\Phi\in \LHX$, $\omega\in\Omega$, $t\in[0,T]$, and $u\in X$. This defines a linear operator on $\LHX$ (rather than $X$), which, by a slight abuse of notation, we also denote by $F'(\omega,t,u)$, e.g.\ in \eqref{eq:TF32}. 
\end{enumerate}
\end{notation}

\begin{remark}
\label{rem:weakenedHoelderAss}
    For $\alpha\in (\frac{1}{2},1]$, the temporal Hölder continuity from Assumption \ref{ass:FG_rate_I}\ref{assCond:FG_rate_HolderContTime} can be weakened to
        \begin{equation}
        \label{eq:weakenedTemporalHoelderLinGrowth}
            \tilde{C}_{\alpha,F}\ce \sup_{0\leq s<t\leq T} \norm[\bigg]{\sup_{x\in X}  \frac{\norm{F(\cdot,t,x)-F(\cdot,s,x)}_X}{(1+\norm{x}_X)(t-s)^\alpha} }_{\frac{2\alpha p}{2\alpha-1}} < \infty,
        \end{equation}
        and likewise for $G$. Under this assumption, subsequent estimates relying on temporal Hölder continuity require an additional Hölder's inequality in $\Omega$ with parameters $\frac{2\alpha p}{2\alpha-1}$ and $2\alpha p$. Further assuming $\init \in L_{\F_0}^{2\alpha p}(\Omega;X)$, $f\in L_{\calP}^{2\alpha p}(\Omega;L^1(0,T;X))$, and $g \in L_{\calP}^{2\alpha p}(\Omega;L^2(0,T;\LHX))$ ensures well-posedness in $L^{2\alpha p}(\Omega;C([0,T];X))$, cf.\ Theorem \ref{thm:wellposedness-Z}. The estimates in Theorem~\ref{thm:main} then hold with $\CalphaF$ replaced by $\tilde{C}_{\alpha,F}C_{2\alpha p,X}^{\mathrm{WP}}$, where the latter constant is defined in \eqref{eq:defCWPqZ}.
\end{remark}

\begin{lemma}
\label{lem:GateauxDerivUnifBdd}
    Suppose that Lipschitz continuity and Gâteaux differentiability as in Assumptions~\ref{ass:FG_rate_I}\ref{assCond:FG_rate_Lipschitz} and~\ref{assCond:FG_rate_GateauxDiffble} hold. Then 
    $F'$ and $G'$ are uniformly bounded by $\LFX$ and $\LGX$, i.e.\
    \begin{equation*}
    \begin{split}
        \sup_{\omega \in \Omega, t \in [0,T],u \in X} \norm{G'(\omega,t,u)}_{\calL(X,\LHX)} &\le \LGX < \infty.
    \end{split}
    \end{equation*}
    and likewise for $F'$.
\end{lemma}
\begin{proof}
    The statement follows from Lemma \ref{lem:bounded-deriv} with $O=Z_1=X$ as well as $Z_2=X$ and $Z_2=\LHX$ for $F'$ and $G'$, respectively.
\end{proof}

We consider contractive approximations of the semigroup associated with the stochastic evolution equation and recall Definition \ref{def:orderScheme}.

\begin{assumption}[Discretisation Scheme]
    \label{ass:schemeR}
    Let $R=(R_\h)_{\h>0}$ be a time discretisation scheme and $X,Y$ as in Assumption \ref{ass:spacesSemigroup}. Suppose that $R$ approximates the semigroup $S=(\sg{t})_{t\ge 0}$ to rate $\alpha \in (0,1]$ on $Y$ and that $R$ is contractive on both $X$ and $Y$.
\end{assumption}

Admissible choices for $R$ include the exponential Euler (EXE) method $R=S$, the implicit Euler (IE) method $R_h=(1+hA)^{-1}$, the Crank--Nicolson (CN) method $R_h=(2-hA)(2+hA)^{-1}$, and other A-stable rational schemes, as discussed in Subsection \ref{subsec:approxSemigroup}. In order to approximate the mild solution $U$ of \eqref{eq:SEE} in time, the rational and exponential Milstein schemes are considered.

\begin{definition}[Rational Milstein scheme]
\label{def:rationalMilstein}
    Let $M \in \N$. The \emph{rational Milstein scheme} on an equidistant time grid $\{\t{j}=j\h: j=0,\dots,M\}$ with step size $\h=\frac{T}{M}$ is defined as the time-discrete stochastic process $\app{}=(\app{j})_{j=0,\ldots,M}$ given by $\app{0}\ce\init$ and, for a scheme $R$ as in Assumption \ref{ass:schemeR},
    \begin{equation}
    \label{eq:defRationalMilstein}
    \begin{split}
        \app{j} &\ce R_\h^j \init + \h \sum_{i=0}^{j-1} R_\h^{j-i}F(\t{i},\app{i}) 
        + \sum_{i=0}^{j-1} R_\h^{j-i}\br[\big]{G(\t{i},\app{i})\Delta W_{i+1} + (G'G)(\t{i},\app{i})\Delta_2W_{i+1}}
    \end{split}
    \end{equation}
    for $1 \le j \le M$, where $\Delta W_{i+1}\ce W(\t{i+1})-W(\t{i})$ denotes the usual Wiener increments, $\Delta_2 W_{i+1} \ce \int_{\t{i}}^{\t{i+1}}\int_{\t{i}}^s \dWHr\dWHs$ the iterated stochastic integrals, and $G'G$ is as defined in Notation \ref{not:GpG}. The penultimate term is to be understood in the sense of \eqref{eq:convradonW} and the last one in the sense of
    \begin{equation}\label{eq:convradonWbilin}
        \Phi \Delta_2 W_{i+1} \ce \sum_{m,n\in\N} \Phi(h_m,h_n)\int_{t_i}^{t_{i+1}}\int_{t_i}^s\,\rmd \beta_r^m\,\rmd \beta_s^n
    \end{equation}
    for $\Phi\in\bilinHX$, $(h_n)_{n\in \N}$ an orthonormal basis of $H$, and where $\beta_t^n \ce W_H(t)h_n$ defines a real-valued Brownian motion $(\beta_t^n)_{t\ge 0}$.
\end{definition}
The series in \eqref{eq:convradonWbilin} shall be understood as the $L^p$-limit of rectangular partial sums. Indeed, they converge in $L^p(\Omega;X)$ for every $p\in [1,\infty)$ by Itô's isometry for $p=2$, monotonicity for $p\in[1,2)$, and for $p>2$ by Nelson's hypercontractivity estimate for the second Wiener chaos \cite[p.~62]{Nualart06}.

\begin{definition}[Exponential Milstein scheme]
\label{def:exponentialMilstein}
    The \emph{exponential Milstein scheme}, or simply \emph{Milstein scheme}, is given by \eqref{eq:defRationalMilstein} with $R = S$.
\end{definition}

Recalling Notation \ref{not:GpG} and Lemma \ref{lem:linearityHSstochInt}, we can reformulate the noise-related terms of the rational Milstein scheme as
\begin{align}
	\label{eq:rewriteGpG}
   &\phantom{\le }\sum_{i=0}^{j-1} R_\h^{j-i}\p[\Big]{\int_{\t{i}}^{\t{i+1}}G(\t{i},\app{i})\dWHs + \int_{\t{i}}^{\t{i+1}} G'(\t{i},\app{i})\br[\Big]{\int_{\t{i}}^s G(\t{i},\app{i}) \dWHr}\dWHs}\nonumber\\
    &= \sum_{i=0}^{j-1} R_\h^{j-i}\p[\Big]{\int_{\t{i}}^{\t{i+1}}G(\t{i},\app{i})\dWHs + \int_{\t{i}}^{\t{i+1}} \int_{\t{i}}^s (G'G)(\t{i},\app{i}) \dWHr\dWHs}.
\end{align}

We focus on the case $\alpha>\frac{1}{2}$.
Otherwise, the exponential Euler scheme or rational schemes, which are both using only the Wiener increments of the noise, are easier to implement and allow for the same pathwise uniform convergence rate $\alpha \in (0,\frac{1}{2}]$ under less restrictive assumptions and with a shorter proof, see \cite[Thm.~6.4]{KliobaVeraar24Rate}.

\begin{assumption}[Nonlinearities II]
\label{ass:FG_rate_II}
    Suppose that Assumption \ref{ass:FG_rate_I} holds for some $q\in [2,\infty)$. 
    Define $\tilde{g}\colon\Omega\times[0,T]\to\bilinHX$ by $ \tilde{g}(\omega,t)\ce (G'G)(\omega,t,0)$. Suppose that
    \begin{enumerate}[label=(\alph*)]
        \item\label{assCond:FG_rate_GateauxHoelder} \emph{(spatial $(2\alpha-1)$-Hölder continuity of Gâteaux derivatives)} $\alpha \in (\frac{1}{2},1]$ and $F'\colon \Omega\times[0,T]\times Y \to \calL(Y,X)$ as well as $G'\colon \Omega\times[0,T]\times Y \to \calL(Y,\LHX)$ are $(2\alpha-1)$-Hölder continuous. That is, there are constants $\CbetaF,\CbetaG \ge 0$ such that for all $\omega\in\Omega$, $t\in[0,T]$ and $x,y\in Y$, 
        \begin{align*}
            \|F'(\omega,t,x)-F'(\omega,t,y)\|_{\calL(Y,X)} &\le \CbetaF\|x-y\|_Y^{2\alpha-1},\\
            \|G'(\omega,t,x)-G'(\omega,t,y)\|_{\calL(Y,\LHX)} &\le \CbetaG\|x-y\|_Y^{2\alpha-1}.
        \end{align*}
        
        \item\label{assCond:FG_rate_GpGLipschitz} \emph{(Lipschitz continuity of $G'G$ on $X$)} There is $\LGpGX\ge 0$ such that for all $\omega\in\Omega,t\in[0,T]$ and $x,y \in X$,
        \begin{equation*}
            \norm{(G'G)(\omega,t,x)-(G'G)(\omega,t,y)}_{\bilinHX} \le \LGpGX \norm{x-y}_X.
        \end{equation*}

        \item\label{assCond:FG_rate_GpGYinvariance} \emph{($Y$-invariance of $G'G$)} $G'G\colon \Omega\times[0,T]\times Y\to \bilinHY$ is strongly $\calP\otimes\calB(Y)$-measurable and $\tilde{g}\in C([0,T];L^q(\Omega;\bilinHY))$ is progressively measurable.
        
        \item\label{assCond:FG_rate_GpGLinGrowth} \emph{($G'G-\tilde{g}$ of linear growth on $Y$)} There is $\LGpGY\ge 0$ such that for all $\omega\in\Omega$, $t\in[0,T]$ and $x \in Y$,
        \begin{equation*}
            \norm{(G'G)(\omega,t,x)-\tilde{g}(\omega,t)}_{\bilinHY} \le \LGpGY(1+\|x\|_Y).
         \end{equation*}
    \end{enumerate}
\end{assumption}

\begin{proposition}
\label{prop:TaylorExpansionFrechet}
    Let $F$ and $G$ satisfy Assumptions \ref{ass:FG_rate_I}\ref{assCond:FG_rate_GateauxDiffble} and \ref{ass:FG_rate_II}\ref{assCond:FG_rate_GateauxHoelder}. 
    Then for all $a,b \in Y$, $t\in[0,T]$ and $\omega\in\Omega$, the Taylor expansion with remainder
    \begin{equation*}
        F(\omega,t,b)=F(\omega,t,a)+F'(\omega,t,a)[b-a]+\int_0^1 \p[\big]{F'(\omega,t,a+\zeta(b-a))-F'(\omega,t,a)}[b-a]\dzeta
    \end{equation*}
    holds in $X$. 
    An analogous Taylor expansion for $G$ holds in $\LHX$.
\end{proposition}
\begin{proof}
    We prove the claim for $F$; it follows analogously for $G$.
    The statement holds pointwise for every $t\in[0,T]$ and $\omega\in\Omega$, thus we suppress these variables in the following.
    
    The $(2\alpha-1)$-Hölder continuity of Assumption \ref{ass:FG_rate_II}\ref{assCond:FG_rate_GateauxHoelder} implies in particular the continuity of $(F|_Y)'\colon Y \to \calL(Y,X)$.
    This is enough to conclude that $F'\colon Y\times Y \to X, (y_1,y_2)\mapsto F'(y_1)[y_2]$ is jointly continuous.
    Since $a,b\in Y$, also $a+\zeta(b-a)\in Y$ for all $\zeta\in [0,1]$ and thus the result follows from~\cite[Thm.~3.2.2]{Hamilton1982}.
\end{proof}

\begin{remark}
    \begin{enumerate}[(i)]
        \item Note that Gâteaux differentiability of $F|_Y\colon Y \to X$ and continuity of the Gâteaux derivatives $(F|_Y)'\colon Y \to \calL(Y,X)$ imply Fréchet differentiability of $F|_Y\colon Y \to X$ by~Proposition~\ref{prop:GateauxContIsFrechet}.

        \item It is essential to assume $(2\alpha-1)$-Hölder continuity of $F'$ as a map from $Y$ to $\calL(Y,X)$ in Assumption \ref{ass:FG_rate_II}\ref{assCond:FG_rate_GateauxHoelder} and likewise for $G$. If one were to require Hölder continuity of $F'$ as a map from $X$ to $\calL(X)$, reasoning as above would imply Fréchet differentiability of $F\colon X\to X$. However, in the case of Nemytskii operators $F$ this already restricts us to the class of affine linear $F$ by Proposition \ref{prop:NemytskiiGateauxAffine}, which does not constitute an interesting class of nonlinearities. This motivates the somewhat technical distinction between Gâteaux and Fréchet differentiability on $X$ and $Y$ in our setting. Here, the combination of Assumptions \ref{ass:FG_rate_I}\ref{assCond:FG_rate_GateauxDiffble} and \ref{ass:FG_rate_II}\ref{assCond:FG_rate_GateauxHoelder} does not imply affine linearity, as $F\colon X \to X$ is merely Gâteaux differentiable in general and $F$ is Fréchet differentiable only from $Y$ to $X$. Indeed, we can treat nonlinear Nemytskii operators $F$ and $G$ for the stochastic transport equation, cf.\ Subsection \ref{subsec:transport}. 
        In general, $F\colon X \to X$ is not Fréchet differentiable. 
    \end{enumerate}
\end{remark}

\section{Well-posedness and Stability}
\label{sec:WPStability}

This section contains the well-posedness and stability results required to prove the convergence result in the next section. Here, well-posedness shall be understood in the sense of existence and uniqueness of mild solutions to \eqref{eq:SEE} satisfying an a priori estimate. Stability refers to pointwise stability of the rational Milstein scheme, i.e.\ moment bounds uniform in the number of time steps.
\begin{definition}
    We call $U\in L_\calP^0(\Omega;C([0,T];X))$ a \emph{mild solution} to \eqref{eq:SEE} if a.s.\ for all $t\in[0,T]$
    \begin{align*}
        \sol{t}=\sg{t}\init + \int_0^t\sg{t-s}F(s,\sol{s})\ds+ \int_0^t\sg{t-s}G(s,\sol{s})\dWHs.
    \end{align*}
\end{definition}
As a consequence of the definition of the mild solution to \eqref{eq:SEE}, for all $s,t\in[0,T]$ with $s\leq t$, \begin{equation}
    \label{eq:mild-solution-consequence}
    \sol{t} = \sg{t-s}\sol{s} + \int_{s}^{t}\sg{t-r}F(r,\sol{r})\dr + \int_{s}^{t}\sg{t-r}G(r,\sol{r}) \dWHr.
\end{equation}
The following standard well-posedness theorem for globally Lipschitz nonlinearities on $X$ (see \cite[Thm.~7.5]{DaPratoZabczyk14} and \cite[Thm.~4.3]{KliobaVeraar24Rate}) is included for completeness. 
Additionally assuming merely linear growth on $Y$ rather than Lipschitz continuity, \cite[Thm.~4.4]{KliobaVeraar24Rate} yields well-posedness on $Y$.

\begin{theorem}[Well-posedness on $X$ and $Y$]
    \label{thm:wellposedness-Z}
    Let $q\in[2,\infty)$, $Z\in \{X,Y\}$, and $F,G,f,g$ be as defined in Assumption \ref{ass:FG_rate_I}. 
    Suppose that Assumptions \ref{ass:spacesSemigroup} and \ref{ass:FG_rate_I}\ref{assCond:FG_rate_Lipschitz} hold as well as  \ref{ass:FG_rate_I}\ref{assCond:FG_rate_Yinvariance} and \ref{assCond:FG_rate_lineargrowthY} if $Z=Y$. Further assume that $\init\in L^q_{\F_0}(\Omega;Z)$, $f\in L^q_\calP(\Omega;L^1(0,T;Z))$, and $g\in L_\calP^q(\Omega;L^2(0,T;\LHZ))$.
    Then there exists a unique mild solution $\sol{}\in L^q(\Omega;C([0,T];Z))$ to \eqref{eq:SEE} and
    \begin{align*}
         \norm[\bigg]{\sup_{t\in [0,T]}\norm{\sol{t}}_Z}_{L^q(\Omega)} 
        &\leq C_{q,Z}^{\mathrm{bdd}} \p[\big]{ 1 + \norm{\init}_{L^q(\Omega;Z)} + \norm{f}_{L^q(\Omega;L^1(0,T;Z))}+ B_q \norm{g}_{L^q(\Omega;L^2(0,T;\LHZ))}},
    \end{align*}
    where $C_{q,Z}^{\mathrm{bdd}}\ce (1 + C_{q,Z}^2T)^{1/2}\exp((1+C_{q,Z}^2T)/2)$ with $C_{q,Z}\ce L_{F,Z} \sqrt{T}+B_q L_{G,Z}$ and $B_q$ as in Theorem \ref{thm:maximal-inequality}. In particular, the bound holds for $\sup_{t\in [0,T]}\|\sol{t}\|_{L^q(\Omega;Z)}$.
\end{theorem}
 
As an inspection of the proofs shows, the statement of the theorem remains true with the same constants if $1$ is added on the left-hand side. Hence, for $q\in [2,\infty)$ and $Z \in \{X,Y\}$, the constants
\begin{align}
\label{eq:defCWPqZ}
    C_{q,Z}^{\mathrm{WP}} &\ce C_{q,Z}^{\mathrm{bdd}}(1 + \|\init\|_{L^q(\Omega;Z)} + \|f\|_{L^q(\Omega;L^1(0,T;Z))} + B_q \norm{g}_{L^q(\Omega;L^2(0,T;\calL_2(H,Z)))})
\end{align} 
are finite under the assumptions of Theorem \ref{thm:wellposedness-Z}. Moreover, the estimate
\begin{equation}
\label{eq:WPboundCWPqZ}
    1+\sup_{t \in [0,T]} \|\sol{t}\|_{L^q(\Omega;Z)} \le C_{q,Z}^{\mathrm{WP}} <\infty
\end{equation}
holds. Note that the additional $1$ on the left-hand side here and in the stability estimate below is included for notational convenience. The following shorthand notation is used throughout the paper.

\begin{notation}
\label{not:fgtildeNorms}
    For $p\in (2,\infty)$, denote by $B_p=4\sqrt{p}$ and $B_2=2$ the constant in Theorem \ref{thm:maximal-inequality}. For $f$, $g$, $\tilde{g}$, and $\init$ in the respective spaces, $p\in [2,\infty)$, $Z$ a Hilbert space, and $t\in [0,T]$ define
\begin{alignat*}{3}
    \norm{f}_{\infty,p,Z} &\ce \norm{f}_{L^\infty(0,T;L^p(\Omega;Z))}, &\norm{\init}_{p,Z} & \ce\norm{\init}_{L^p(\Omega;Z)}, \nonumber \\
    \tnorm{g}_{\infty,p,Z} &\ce \norm{g}_{L^\infty(0,T;L^p(\Omega;\calL_2(H,Z)))},\quad  &\tnorm{g(t)}_{p,Z}& \ce \|g(t)\|_{L^p(\Omega;\calL_2(H,Z))},\\
    \bnorm{\tilde{g}}_{\infty,p,Z}&\ce \norm{\tilde{g}}_{L^\infty(0,T;L^p(\Omega;\calL_2^{(2)}(H,Z)))},
     &\bnorm{\tilde{g}(t)}_{p,Z} &\ce \|\tilde{g}(t)\|_{L^p(\Omega;\calL_2^{(2)}(H,Z))}. \nonumber
\end{alignat*}
\end{notation}

\begin{proposition}[Pointwise stability on $Y$]
    \label{prop:pointwise-stability-Y}
    Suppose that Assumptions \ref{ass:spacesSemigroup} and \ref{ass:schemeR} hold for some $\alpha \in (0,1]$. Let $p\in [2,\infty)$ and $\init\in L^p_{\F_0}(\Omega;Y)$.
    Further assume $Y$-invariance and linear growth as in Assumptions \ref{ass:FG_rate_I}\ref{assCond:FG_rate_Yinvariance}, \ref{ass:FG_rate_I}\ref{assCond:FG_rate_lineargrowthY}, \ref{ass:FG_rate_II}\ref{assCond:FG_rate_GpGYinvariance}, and \ref{ass:FG_rate_II}\ref{assCond:FG_rate_GpGLinGrowth} for $q=p$. 
    Then the rational Milstein scheme $\app{}=(\app{j})_{j=0,\ldots,M}$ satisfies the pointwise strong stability estimate
    \begin{align}
    \label{eq:stabilityEstY}
        1+ \max_{0\le j \le M} \norm{\app{j}}_{L^p(\Omega;Y)} \leq \CstabpY < \infty
    \end{align}
    with $\CstabpY \ce c_{\init,f,g,\tilde{g}} (1+C^2T)^{1/2} \exp((1+C^2T)/2)$, where $C \ce \LFY \sqrt{T} + B_p\LGY  + \frac{1}{\sqrt{2}}B_p^2\LGpGY \sqrt{T}$ and
    \begin{align*}
        c_{\init,f,g,\tilde{g}} &\ce 1+\norm{\init}_{p,Y} + \norm{f}_{\infty,p,Y} T + B_p \tnorm{g}_{\infty,p,Y} \sqrt{T} + \frac{1}{\sqrt{2}} B_p^2\bnorm{\tilde{g}}_{\infty,p,Y} T.
    \end{align*}
\end{proposition}
\begin{proof}
    Let $j\in\{1,\dots,M\}$ and define $\varphi(j) \ce 1 + \norm{\app{j}}_{p,Y}$. By the definition of the rational Milstein scheme $\app{}$ rewritten as in \eqref{eq:rewriteGpG}, we have to bound 
    \begin{align*}
        \varphi(j)
        &\le 1+\norm{R_h^j\init}_{p,Y} + h \norm[\Big]{\sum_{i=0}^{j-1} R_h^{j-i}F(\t{i},\app{i})}_{p,Y} + \norm[\Big]{\sum_{i=0}^{j-1} \int_{\t{i}}^{\t{i+1}} R_h^{j-i}G(\t{i},\app{i})\dWHs}_{p,Y}\\
        &\phantom{\le }+ \norm[\Big]{\sum_{i=0}^{j-1} \int_{\t{i}}^{\t{i+1}} \int_{\t{i}}^s R_h^{j-i}(G'G)(\t{i},\app{i})\dWHr\dWHs}_{p,Y}.
    \end{align*}
    We only detail the treatment of the last term, as the other terms follow by analogous but simpler arguments.
    Using Itô's isomorphism from Theorem \ref{thm:maximal-inequality} twice, the isometry $\calL_2(H,\LHY)\cong \bilinHY$, contractivity of the scheme $R$ and linear growth of $G'G-\tilde{g}$ on $Y$, it can be bounded by 
    \begin{align*}
        &B_p \p[\Big]{\sum_{i=0}^{j-1} \int_{\t{i}}^{\t{i+1}} \norm[\Big]{ \int_{\t{i}}^s R_h^{j-i}(G'G)(\t{i},\app{i})\dWHr}_{p,\LHY}^2\ds}^{1/2}\\
        &\le B_p^2 \p[\Big]{\sum_{i=0}^{j-1} \int_{\t{i}}^{\t{i+1}} \int_{\t{i}}^s \norm{R_h^{j-i}(G'G)(\t{i},\app{i})}_{p,\bilinHY}^2\dr\ds}^{1/2}\\
        &\le \frac{B_p^2}{\sqrt{2}} \p[\Big]{h^2 \sum_{i=0}^{j-1} \norm{(G'G)(\t{i},\app{i})}_{p,\bilinHY}^2}^{1/2}\le \frac{B_p^2}{\sqrt{2}} \p[\Big]{ \h^2\sum_{i=0}^{j-1} \p[\big]{ \LGpGY (1+\norm{\app{i}}_{p,Y}) + \bnorm{\tilde{g}({\t{i}})}_{p,Y} }^2 }^{\frac{1}{2}}\\
        &\le \frac{B_p^2}{\sqrt{2}} \bnorm{\tilde{g}}_{\infty,p,Y} T + \frac{B_p^2}{\sqrt{2}} \LGpGY \sqrt{T}  \p[\Big]{ \h\sum_{i=0}^{j-1}  (1+\norm{\app{i}}_{p,Y})^2 }^{\frac{1}{2}}.
    \end{align*}
    Proceeding analogously for the remaining terms, using the Cauchy--Schwarz inequality for the $F$-terms and $h \le T$, we deduce for the full expression that
    \begin{align*}
        \varphi(j) 
        &\le c_{\init,f,g,\tilde{g}} + C \cdot \p[\Big]{ \h\sum_{i=0}^{j-1}  (1+\norm{\app{i}}_{p,Y})^2 }^{\frac{1}{2}} \leq c_{\init,f,g,\tilde{g}} + C \cdot \p[\Big]{ \h \sum_{i=0}^{j-1} \varphi(i)^2 }^{\frac{1}{2}}.
    \end{align*}
    Since $1+\norm{\app{0}}_{p,Y}=1+\norm{\init}_{p,Y} \leq c_{\init,f,g,\tilde{g}}$, the discrete Grönwall Lemma \ref{lem:gronwall} implies that
    \begin{align*}
        1 + \norm{\app{j}}_{p,Y} \leq c_{\init,f,g,\tilde{g}} \p{ 1+C^2T }^{\frac{1}{2}} \e^{\frac{1}{2}(1+C^2T)}
    \end{align*}
    for all $j\in\{0,\dots,M\}$. Since the right-hand side is independent of $j$, the estimate carries over to the maximum over $j$.
\end{proof}

\begin{remark}
\label{rem:stability}
\begin{enumerate}[label=(\roman*)]
    \item An analogous stability estimate holds on $X$, since Lipschitz continuity implies linear growth. The required regularity $f\in C([0,T];L^p(\Omega;X))$ and likewise for $g$ and $\tilde{g}$ is a consequence of Assumptions \ref{ass:FG_rate_I}\ref{assCond:FG_rate_Yinvariance} (or \ref{assCond:FG_rate_HolderContTime}, cf.\ part \ref{item:fgAssumptions}) 
     and \ref{ass:FG_rate_II}\ref{assCond:FG_rate_GpGYinvariance}.
    \item\label{item:fgAssumptions}
    Assumption \ref{ass:FG_rate_I}\ref{assCond:FG_rate_Yinvariance} implies the regularity of $f,g$ needed for well-posedness on $X$ and $Y$, as 
    \[
        C([0,T];L^q(\Omega;Y)) \seq L^q(\Omega;L^1(0,T;Y)) \seq L^q(\Omega;L^1(0,T;X)).
    \] 
    Temporal Hölder continuity as in Assumption \ref{ass:FG_rate_I}\ref{assCond:FG_rate_HolderContTime} implies that $f\in C^\alpha([0,T]; L^q(\Omega; X))$ and likewise for $g$ and thus the assumptions on $f,g$ (not $\tilde{g}$) for stability on $X$.
    \item\label{remCond:pathwiseUniformStability}
    If additionally $f\in L_\calP^p(\Omega;C([0,T];Y))$ and likewise for $g$ and $\tilde{g}$ we obtain pathwise uniform stability of the scheme.
    Then
    \begin{equation*}
        \norm[\Big]{\max_{0 \le j \le M} \|\app{j}\|_Y}_p \le C_{p,T},
    \end{equation*}
    where $C_{p,T}$ is independent of $M$ and $h$. The proof of this stronger stability statement requires a dilation argument to handle the discrete stochastic convolutions. The proof of~\cite[Prop.~5.1]{KliobaVeraar24Rate} can be adapted in a straightforward manner to include the $(G'G)$-terms.
\end{enumerate}
\end{remark}

\section{Convergence Rates}
\label{sec:convergence}

Our main goal is to establish
pathwise uniform convergence rates for the rational and the exponential Milstein scheme as a temporal approximation of stochastic evolution equations of the form
\begin{equation}\label{eq:SEEconv}\tag{SEE}
	\begin{cases}
		\rmd U + AU \dt &= F(t,U)\dt + G(t,U) \dW\quad \text{on }(0,T], \\
		\sol{0}&=\init \in L_{\F_0}^p(\Omega;X).
	\end{cases}
\end{equation}
on $[0,T]$, where $-A$ generates a contractive $C_0$-semigroup $(S(t))_{t\ge 0}$ on a Hilbert space $X$ and $U$ is the mild solution. The conditions on the nonlinearity $F$ and the multiplicative noise $G$ as well as $A$, $\init$, $X$, and $W$ are as in Section \ref{sec:settingAssumptions}. As we have seen in Section \ref{sec:WPStability}, these ensure the well-posedness of \eqref{eq:SEEconv} both in $X$ and a more regular continuously embedded subspace $Y$ as well as pointwise stability of the rational Milstein scheme from Definition \ref{def:rationalMilstein} in $Y$. We recall Notation \ref{not:fgtildeNorms}.

This section is split into two subsections: The central error estimates are presented, proved, and discussed in Subsection \ref{subsec:errorEstGrid}. Subsequently, the error estimates are generalised to a suitable extension of the rational Milstein scheme to the full time interval $[0,T]$ in Subsection \ref{subsec:errorEstInterval}. Before discussing our main result in Theorem \ref{thm:main}, a pathwise uniform convergence rate for the rational Milstein scheme, we first state two useful lemmas on the path regularity of the mild solution to~\eqref{eq:SEEconv}. Two of the error terms are estimated in the auxiliary Lemmas \ref{lem:E2} and \ref{lem:E3} as part of the proof of Theorem \ref{thm:main}. The error estimate is extended to $p\in [1,2)$ in Corollary \ref{cor:smallP} and improved for exponential Milstein schemes in our second main result, Theorem \ref{thm:expMilstein}. Possible generalisations regarding less regular nonlinearities, weaker differentiability assumptions, or $2$-smooth Banach spaces are discussed in Remarks \ref{rem:alpha012}--\ref{rem:2smooth} and simplifications in the linear case in Corollary \ref{cor:linearCase}. The error estimate remains valid for a suitable extension of the scheme to $[0,T]$, as shown in Theorem \ref{thm:extensionInterval}.

\subsection{Main error estimates at the grid points}
\label{subsec:errorEstGrid}
\begin{lemma}
    \label{lem:regularity-UtDiffVst-Z} 
    Let $Z\in\{X,Y\}$. Suppose that the assumptions of Theorem \ref{thm:wellposedness-Z} hold for some $q\in [2,\infty)$.
    Moreover, assume that $f\in L^\infty(0,T;L^q(\Omega;Z))$ and $g\in L^\infty(0,T;L^q(\Omega;\calL_2(H,Z)))$.
    Then for all $0 \le s \le t \le T$
    \begin{equation*}
        \norm{ \sol{t}-\sg{t-s}\sol{s} }_{L^q(\Omega;Z)} \leq L_{q,Z,1} (t-s) + L_{q,Z,2} (t-s)^{\frac{1}{2}},
    \end{equation*}
   where $L_{q,Z,1} \ce L_{F,Z}\CWPqZ+\|f\|_{\infty,q,Z}$ and $L_{q,Z,2} \ce B_q(L_{G,Z}\CWPqZ+\tnorm{g}_{\infty,q,Z})$ with $\CWPqZ$ as in \eqref{eq:defCWPqZ}.
\end{lemma}
\begin{proof}
 We prove the statement for $Z=Y$, noting that the same argument works for $Z=X$, since Lipschitz continuity implies linear growth. Applying the triangle inequality and Itô's isomorphism from Theorem \ref{thm:maximal-inequality} to the mild solution formula \eqref{eq:mild-solution-consequence}, we can employ contractivity of the semigroup, linear growth on $Y$, and the a priori estimate on $Y$ from Theorem \ref{thm:wellposedness-Z} to deduce
    \begin{align*}
        \norm{\sol{t}-\sg{t-s}\sol{s}}_{q,Y} &= \Big\|\int_{s}^t\sg{t-r}F(r,\sol{r})\dr + \int_{s}^t \sg{t-r}G(r,\sol{r}) \dWHr\Big\|_{q,Y}\\
        &\le (t-s) \sup_{r\in[0,T]} \norm{F(r,\sol{r})}_{q,Y} + B_q \p[\Big]{\int_{s}^t \norm{G(r,\sol{r})}_{q,\LHY}^2\dr}^{1/2}\\
        &\le (t-s) (\LFY\CWPqY+\|f\|_{\infty,q,Y}) + B_q (t-s)^{1/2}(\LGY\CWPqY+\tnorm{g}_{\infty,q,Y}).\qedhere
    \end{align*}
\end{proof}

In order to 
achieve decay of order higher than $\frac{1}{2}$,
an additional stochastic integral term is taken into account in the difference. This yields decay of order $\min\{\alpha+\frac{1}{2},1\}$ on the space $X$. 
\begin{lemma}
    \label{lem:regularity-DiffWithStInt-X} 
    Suppose that the assumptions of Theorem \ref{thm:wellposedness-Z} hold for $q=p$ on both $Z=X$ and $Z=Y$.
    Moreover, assume that $f\in L^\infty(0,T;L^p(\Omega;X))$ and $g\in L^\infty(0,T;L^p(\Omega;\LHX))$.
    Then for all $0 \leq s \leq t \leq T$ and $V_{t,s}\ce \sg{t-s}\sol{s}$
    \begin{equation*}
        \norm[\Big]{ \sol{t}-V_{t,s} - \int_{s}^{t} G(s,V_{t,s}) \dWHr }_{L^p(\Omega;X)}
        \leq L_1(t-s)+L_2(t-s)^{3/2}+L_3(t-s)^{\alpha+\frac{1}{2}}
    \end{equation*}
    with $L_1 \ce \LFX\CWPpX +\norm{f}_{\infty,p,X} +\frac{1}{\sqrt{2}}B_p \LGX\LpXtwo$, $L_2 \ce \frac{1}{\sqrt{3}}B_p \LGX\LpXone$, and
      \begin{align*}
        L_3 \ce \frac{ B_p}{\sqrt{2\alpha+1}} \p[\big]{2\CY\p{ (\LGY+\LGX)\CWPpY+\tnorm{g}_{\infty,p,Y}}+\CalphaG}, 
    \end{align*}
    where $\CWPpX,\CWPpY$ and $\LpXone,\LpXtwo$ are as defined in \eqref{eq:defCWPqZ} and Lemma~\ref{lem:regularity-UtDiffVst-Z} with $Z=X$ and $q=p$, respectively.
\end{lemma}

\begin{proof}
    Similar to the proof of Lemma \ref{lem:regularity-UtDiffVst-Z} we insert the mild solution formula \eqref{eq:mild-solution-consequence} and first make use of the contractivity of the semigroup, the linear growth of $F-f$ on $X$ as well as Itô's isomorphism from Theorem \ref{thm:maximal-inequality} via
    \begin{align*}
         &\norm[\Big]{ \sol{t}-\sg{t-s}\sol{s} - \int_{s}^{t} G(s,\sg{t-s}\sol{s}) \dWHr }_{p,X} \nonumber\\
        &\leq \int_{s}^{t} \norm{ \sg{t-r} F(r,\sol{r}) }_{p,X} \dr+ \Big\| \int_{s}^{t} \sg{t-r}G(r,\sol{r}) - G(s,\sg{t-s}\sol{s}) \dWHr \Big\|_{p,X} \nonumber\\
        &\leq \p{ \LFX\CWPpX +\norm{f}_{\infty,p,X} } \cdot(t-s) + B_p \p[\Big]{ \int_{s}^{t} \norm{ \sg{t-r}G(r,\sol{r}) - G(s,\sg{t-s}\sol{s}) }_{p,\LHX}^2 \dr }^{\frac{1}{2}}.
    \end{align*}

    To estimate the second term, we first focus on the integrand.
    Using Lemma \ref{lem:semigroup-interpolation}, the linear growth of $G-g$ on $Y$, the temporal Hölder continuity of $G$, the spatial Lipschitz continuity of $G$ on $X$ and Lemma \ref{lem:regularity-UtDiffVst-Z} with $Z=X$, we obtain
    \begin{align*}
        \norm{&\sg{t-r}G(r,\sol{r}) - G(s,\sg{t-s}\sol{s}) }_{p,\LHX} \leq \norm{ \br{\sg{t-r}-\Id} G(r,\sol{r}) }_{p,\LHX}
        \nonumber\\
        &+ \norm{ G(r,\sol{r})-G(s,\sol{r}) }_{p,\LHX} +\LGX\p[\big]{\norm{\sol{r}-\sg{r-s}\sol{s}}_{p,X}+ \big\|\br{\sg{r-s} - \sg{t-s}}\sol{s}\big\|_{p,X}} \nonumber\\
        \begin{split}
        &\leq 2\CY \p{ \LGY\CWPpY+\tnorm{g}_{\infty,p,Y}}\cdot (t-r)^\alpha + \CalphaG (r-s)^\alpha \\
        &\phantom{\le }+ \LGX\p[\big]{\p{\LpXone (r-s) + \LpXtwo (r-s)^{\frac{1}{2}}}
        +2\CY\CWPpY(t-r)^\alpha}.
        \end{split}
    \end{align*}

    Inserting this inequality back into the integral, we can use the triangle inequality in\\ $L^2(s,t;L^p(\Omega;\LHX))$ to deduce
    \begin{align*}
        \MoveEqLeft B_p \p[\Big]{ \int_{s}^{t} \norm{ \sg{t-r}G(r,\sol{r}) - G(s,\sg{t-s}\sol{s}) }_{p,\LHX}^2 \dr }^{\frac{1}{2}} \nonumber\\
        \begin{split}
        &\le \frac{ B_p}{\sqrt{2\alpha+1}} \p[\big]{2\CY\p{ (\LGY+\LGX)\CWPpY+\tnorm{g}_{\infty,p,Y}}+\CalphaG}(t-s)^{\alpha+\frac{1}{2}} \\
        &\phantom{\le }
        + B_p \LGX \p[\Big]{ \frac{\LpXone}{\sqrt{3}} (t-s)^{3/2} +\frac{\LpXtwo}{\sqrt{2}} (t-s) }.
        \end{split}
    \end{align*}
    
    Finally, the lemma follows from combining the different estimates.
\end{proof}

We are now in the position to state and prove the main result of this paper.

\begin{theorem}
\label{thm:main}
    Let $p\in[2,\infty)$ and $\alpha \in (\frac{1}{2},1]$ such that Assumptions \ref{ass:spacesSemigroup}, \ref{ass:FG_rate_I}, and \ref{ass:schemeR} hold for $\alpha$ and $q=2\alpha p$, and Assumption \ref{ass:FG_rate_II} holds for $\alpha$ and $q=p$. Let $\xi \in L_{\F_0}^{2\alpha p}(\Omega;Y)$. Denote by $U$ the solution of \eqref{eq:SEEconv} and by $u=(u_j)_{j=0,\ldots,M}$ the rational Milstein scheme.
    
    Then the rational Milstein scheme converges at rate $\alpha$ up to a logarithmic correction factor as $h \to 0$ and for $M\ge 2$
    \begin{equation}
    \label{eq:mainEstimate}
        \bigg\|\max_{0 \le j \le M} \|\sol{\t{j}}-\app{j}\|_X \bigg\|_{L^p(\Omega)} \le C_1 h+(C_2+C_3\max\big\{\sqrt{\log(T/h)},\sqrt{p}\big\})h^\alpha
    \end{equation}
    with constants $C_i \ce C_\e \tilde{C}_i$ for $i\in \{1,2,3\}$, where $C_\e\ce (1+C_4^2T)^{1/2}\exp(\frac{1}{2}(1+C_4^2T))$, $C_4\ce \LFX\sqrt{T}+B_p\LGX+\frac{1}{\sqrt{2}}B_p^2\LGpGX\sqrt{T}$, $B_2=2$, $B_p=4\sqrt{p}$ for $p>2$, and $\tilde{C}_i$ are defined in \eqref{eq:constantDefinition}. 
\end{theorem}

The statement remains true if one merely assumes $f\in L_\calP^\infty(0,T;L^{2\alpha p}(\Omega;Y)) \cap C([0,T];L^p(\Omega;Y))$ rather than $f\in C([0,T];L^{2\alpha p}(\Omega;Y))$ as above, and likewise for $g$ with values in $\LHY$. Moreover, it suffices if Assumption \ref{ass:FG_rate_I}\ref{assCond:FG_rate_HolderContTime} holds for $q=p$ or if the semigroup $(S(t))_{t\ge 0}$ is quasi-contractive (cf.\ Definition \ref{def:quasiContractive}), where the latter can be seen by a simple scaling argument. 

Since $\|\phi\|_{L^r(\Omega)} \le \|\phi\|_{L^p(\Omega)}$ for all $1\le r\le p$ and $\phi \in L^p(\Omega)$, the error estimate extends to lower moments $r\in [1,2)$. Note that the regularity assumptions such as integrability of the initial values are required to hold for some $p\ge 2$, with $p=2$ corresponding to the weakest assumptions.
\begin{corollary}
\label{cor:smallP}
    Under the notation and assumptions of Theorem \ref{thm:main} for $p=2$, it holds that for all $r\in [1,2]$ and $M\ge 2$
    \begin{equation*}
        \bigg\|\max_{0 \le j \le M} \|\sol{\t{j}}-\app{j}\|_X \bigg\|_{L^r(\Omega)} \le C_1 h+(C_2+C_3\max\big\{\sqrt{\log(T/h)},\sqrt{2}\big\})h^\alpha.
    \end{equation*}
\end{corollary}
To prove the main error estimate, we define the following auxiliary time-discrete stochastic processes, which allow us to split the error suitably and analyse two of the error terms in the subsequent lemmas.

\begin{definition}
\label{def:vj12}
    Let the time-discrete stochastic processes $v^k = (v_j^k)_{0\le j \le M}$ for $k=1,2$ be given by $v_0^1\ce v_0^2 \ce \init$ and, for $1\le j \le M$,
    \begin{align}
    \label{eq:defvj1}
    v_j^1 &\ce \sg{\t{j}}\init + \sum_{i=0}^{j-1} \int_{\t{i}}^{\t{i+1}} \sg{\t{j}-\t{i}}F(\t{i},\sol{\t{i}}) \ds \\
    &\quad+ \sum_{i=0}^{j-1} \int_{\t{i}}^{\t{i+1}} \sg{\t{j}-\t{i}}\p[\Big]{ G(\t{i},\sol{\t{i}}) + \int_{\t{i}}^{s} (G'G)(\t{i},\sol{\t{i}})\dWHr} \dWHs, \nonumber
\end{align}
and $v_j^2$ by \eqref{eq:defvj1} with all $\sol{\t{i}}$ replaced by $\app{i}$.
\end{definition}

In the following lemma, the error arising from the difference of the nonlinearities evaluated at the solution and at the approximation at the grid points is analysed. It is estimated in terms of sums of the full error at previous grid points, which will finally be dealt with by a discrete Grönwall argument in the proof of Theorem \ref{thm:main}.
\begin{lemma}
\label{lem:E2}
    Let $1\le m \le M$ and $E(m) \ce \norm{\max_{1\leq j\leq m}\norm{\sol{\t{j}} - \app{j} }_X}_p$ as well as $E(0)\ce 0$. Under the assumptions and with the notation and constant $C_4$ from Theorem \ref{thm:main},
    \begin{align*}
        \norm[\Big]{\max_{1\leq j\leq m}\norm{v_j^1 - v_j^2}_X}_{L^p(\Omega)} &\le C_4 \p[\Big]{ \h \sum_{i=0}^{m-1} E(i)^2 }^{\frac{1}{2}}.
    \end{align*}
\end{lemma}
\begin{proof}
  Via the triangle inequality and Definition \ref{def:vj12}, we split the error into
\begin{align*}
    &\norm[\Big]{\max_{1\leq j\leq m}\norm{v_j^1 - v_j^2}_X}_p 
    \leq \norm[\Big]{ \max_{1\leq j\leq m}\norm[\Big]{ \sum_{i=0}^{j-1} \int_{\t{i}}^{\t{i+1}} \sg{\t{j}-\t{i}} \br{ F(\t{i},\sol{\t{i}}) - F(\t{i},\app{i}) } \ds }_X}_p \\
    &+ \norm[\Big]{ \max_{1\leq j\leq m}\norm[\Big]{ \sum_{i=0}^{j-1} \int_{\t{i}}^{\t{i+1}} \sg{\t{j}-\t{i}} \br[\big]{ G(\t{i},\sol{\t{i}}) - G(\t{i},\app{i}) } \dWHs }_X}_p \\
    &+ \norm[\Big]{ \max_{1\leq j\leq m}\norm[\Big]{ \sum_{i=0}^{j-1} \int_{\t{i}}^{\t{i+1}} \sg{\t{j}-\t{i}} \int_{\t{i}}^{s}  \p[\big]{(G'G)(\t{i},\sol{\t{i}})- (G'G)(\t{i},\app{i})}\dWHr \dWHs }_X}_p \\
    &\ec T_{\mathrm{GW},1} + T_{\mathrm{GW},2} + T_{\mathrm{GW},3}.
\end{align*}

Contractivity of the semigroup, the triangle inequality, Lipschitz continuity of $F$ on $X$, the definition of $E(m)$, and the Cauchy--Schwarz inequality yield
\begin{align}
\label{eq:TGW1}
    T_{\mathrm{GW},1} &\leq \norm[\bigg]{ \sum_{i=0}^{m-1}  \int_{\t{i}}^{\t{i+1}} \norm{ F(\t{i},\sol{\t{i}}) - F(\t{i},\app{i}) }_X \ds }_p
    \leq \LFX \h \sum_{i=0}^{m-1} \norm{\sol{\t{i}}-\app{i}}_{p,X}\nonumber\\
    &\leq \LFX  \h \sum_{i=0}^{m-1} E(i) \le \LFX \sqrt{T} \p[\Big]{ \h \sum_{i=0}^{m-1} E(i)^2 }^{\frac{1}{2}}.
\end{align}

For $T_{GW,2}$, we proceed analogously, making use of the maximal inequality from Theorem \ref{thm:maximal-inequality} and Lipschitz continuity of $G$ on $X$. This results in the bound
\begin{align}
\label{eq:TGW2}
    T_{\mathrm{GW},2} &\leq \norm[\bigg]{ \sup_{t\in[0,t_m]} \norm[\Big]{ \int_{0}^{t} \sg{t-\floorh{s}} \br[\big]{ G(\floorh{s},\sol{\floorh{s}}) - G(\floorh{s},\app{\floorh{s}/h}) } \dWHs }_X}_p \nonumber\\
    &\leq B_p  \p[\Big]{ \sum_{i=0}^{m-1} \int_{\t{i}}^{\t{i+1}} \norm[\big]{ \sg{s-\t{i}} \br[\big]{ G(\t{i},\sol{\t{i}}) - G(\t{i},\app{i}) } }_{p,\LHX}^2 \ds }^{\frac{1}{2}}  \nonumber\\
    &\leq B_p \LGX \p[\Big]{ \h \sum_{i=0}^{m-1} \norm{\sol{\t{i}}-\app{i}}_{p,X}^2 }^{\frac{1}{2}} \leq B_p \LGX \p[\Big]{ \h \sum_{i=0}^{m-1} E(i)^2 }^{\frac{1}{2}}.
\end{align}

To estimate the last term, we apply the maximal inequality and Minkowski's integral inequality twice, combined with the isometric isomorphism $\calL_2(H,\LHX) \cong \bilinHX$. Lipschitz continuity of $G'G$ then implies
\begin{align}
\label{eq:TGW3}
    T_{\mathrm{GW},3}
     &\leq B_p \p[\Big]{  \sum_{i=0}^{m-1} \int_{\t{i}}^{\t{i+1}} \norm[\Big]{\sg{s-\t{i}} \p[\Big]{\int_{\t{i}}^{s} (G'G)(\t{i},\sol{\t{i}})
    - (G'G)(\t{i},\app{i}) \dWHr} }_{p,\LHX}^2 \ds }^{1/2}\nonumber\\
    &\leq B_p^2 \p[\Big]{ \sum_{i=0}^{m-1} \int_{\t{i}}^{\t{i+1}} \int_{\t{i}}^{s} \norm[\big]{ (G'G)(\t{i},\sol{\t{i}}) - (G'G)(\t{i},\app{i}) }_{p,\bilinHX}^2 \dr \ds }^{\frac{1}{2}} \nonumber\\
    &\leq B_p^2 \LGpGX \p[\Big]{ \sum_{i=0}^{m-1} \int_{\t{i}}^{\t{i+1}} (s-\t{i}) \norm{ \sol{\t{i}} - \app{i} }_{p,X}^2 \ds }^{\frac{1}{2}} \leq \frac{B_p^2}{\sqrt{2}} \LGpGX \sqrt{T} \p[\Big]{ h \sum_{i=0}^{m-1} E(i)^2 }^{\frac{1}{2}}. 
\end{align}
Finally, the statement can be concluded by adding \eqref{eq:TGW1}, \eqref{eq:TGW2}, and \eqref{eq:TGW3}.
\end{proof}

Next, we estimate the error $E_3$ caused by the rational approximation of the semigroup. To this end, we can no longer make use of the maximal inequality for stochastic convolutions of Theorem~\ref{thm:maximal-inequality}, since the semigroup is approximated by a rational scheme, leading to a discrete stochastic convolution. Instead, we apply the logarithmic square function estimate from Proposition \ref{prop:log-maximal-inequality}, which introduces a logarithmic correction factor in the estimate, but in turn allows us to proceed via the approximation order $\alpha$ of the scheme $R$ and linear growth of $G-g$ on $Y$. 
\begin{lemma}
\label{lem:E3}
    Under the assumptions and with the notation from Theorem \ref{thm:main},
    \begin{align*}
    \norm[\Big]{\max_{1\le j\le M} \|v_j^2-\app{j}\|_X}_{L^p(\Omega)} &\le C_\alpha \p[\big]{ \norm{\init}_{p,Y}+ \p[\big]{\LFY \CstabpY+\|f\|_{\infty,p,Y} }T}\cdot h^\alpha\nonumber\\
    &\phantom{\le }+ \tilde{C}_3\sqrt{\max\{\log(M), p\}} \cdot h^\alpha
\end{align*}
with $\tilde{C}_3\ce  KC_\alpha \p{  \p{\LGY\CstabpY+\tnorm{ g }_{\infty,p,Y}} \sqrt{T}
    +\frac{1}{\sqrt{2}}B_p  \p{ \LGpGY\CstabpY+\bnorm{\tilde{g}}_{\infty,p,Y}}T}$.
\end{lemma}
\begin{proof}
    By 
    Definition \ref{def:vj12}, the difference of $v_j^2$ and $\app{j}$ can be split into
\begin{align*}
    \norm[\bigg]{& \max_{1\leq j\leq m} \norm{ v_j^2 - \app{j} }_X}_p \leq \norm[\bigg]{ \max_{1\leq j\leq m} \norm[\big]{ \br[\big]{ \sg{\t{j}} - R_\h^j } \init }_X}_p \\
    &\phantom{\le }+ \norm[\bigg]{ \max_{1\leq j\leq m} \norm[\Big]{ \sum_{i=0}^{j-1} \int_{\t{i}}^{\t{i+1}} \br[\big]{ \sg{\t{j}-\t{i}} - R_\h^{j-i} } F(\t{i},\app{i}) \ds }_X}_p \\
    &\phantom{\le }+ \norm[\bigg]{ \max_{1\leq j\leq m} \norm[\Big]{ \sum_{i=0}^{j-1} \int_{\t{i}}^{\t{i+1}} \br[\big]{ \sg{\t{j}-\t{i}} - R_\h^{j-i} } G(\t{i},\app{i}) \dWHs }_X}_p \\
    &\phantom{\le }+ \norm[\bigg]{ \max_{1\leq j\leq m} \norm[\Big]{ \sum_{i=0}^{j-1} \int_{\t{i}}^{\t{i+1}} \br[\big]{ \sg{\t{j}-\t{i}} - R_\h^{j-i} } \int_{\t{i}}^{s} (G'G)(\t{i},\app{i})\dWHr \dWHs }_X}_p \\
    &\ec T_{R,1} + T_{R,2} + T_{R,3} + T_{R,4}.
\end{align*}

Since $R$ approximates $S$ to order $\alpha$ on $Y$ by Assumption \ref{ass:schemeR},
\begin{align}
\label{eq:TR1}
    T_{R,1} &\le \norm[\bigg]{ \max_{1\leq j\leq m} \norm[\big]{ \sg{\t{j}} - R_\h^j }_{\calL(Y,X)} \norm{\init }_Y}_p \leq C_\alpha \norm{\init}_{p,Y} \cdot \h^\alpha.
\end{align}

By the same reasoning, further invoking linear growth of $F-f$ and the stability estimate on $Y$ for the rational Milstein scheme from Proposition \ref{prop:pointwise-stability-Y}, we deduce
\begin{align}
\label{eq:TR2}
    T_{R,2} &\le C_\alpha h^\alpha\norm[\bigg]{\sum_{i=0}^{m-1} \int_{\t{i}}^{\t{i+1}} \norm{F(\t{i},\app{i})}_Y \ds }_p\leq C_\alpha h^\alpha \sum_{i=0}^{m-1} h \p[\big]{\LFY (1+\norm{ \app{i} }_{p,Y})+\|f(t_i)\|_{p,Y} }  \nonumber\\
    &\le C_\alpha \p[\big]{\LFY \CstabpY+\|f\|_{\infty,p,Y} }T \cdot h^\alpha.
\end{align}

To estimate $T_{R,3}$, we apply the logarithmic square function estimate from Proposition \ref{prop:log-maximal-inequality} as discussed above. Leveraging the stability estimate of the scheme on $Y$ from Proposition \ref{prop:pointwise-stability-Y}, with the abbreviation $\Phi_p(m)\ce \sqrt{\max\{\log(m),p\}}$ this yields
\begin{align}
\label{eq:TR3}
    T_{R,3} &\leq \norm[\bigg]{ \sup_{t\in[0,\t{m}], 1\leq j\leq m } \norm[\Big]{ \int_{0}^{t} \sum_{i=0}^{j-1} \ind_{[\t{i},\t{i+1})}(s) \br{ \sg{\t{j}-\t{i}} - R_\h^{j-i} } G(\t{i},\app{i}) \dWHs }_X}_p \nonumber\\
    &\leq K \Phi_p(m)\norm[\bigg]{ \max_{1\leq j\leq m} \p[\Big]{ \sum_{\ell=0}^{m-1}\int_{\t{\ell}}^{\t{\ell+1}} \norm[\Big]{ \sum_{i=0}^{j-1} \ind_{[\t{i},\t{i+1})}(s) \br{ \sg{\t{j}-\t{i}}- R_\h^{j-i} } G(\t{i},\app{i}) }_{\LHX}^2 \ds }^{\frac{1}{2}} }_p \nonumber\\
    &\le K \Phi_p(m)\norm[\bigg]{ \max_{1\leq j\leq m} \p[\Big]{ \sum_{\ell=0}^{m-1}\int_{\t{\ell}}^{\t{\ell+1}} \norm[\big]{\br{ \sg{\t{j}-\t{\ell}}- R_\h^{j-\ell} } G(\t{\ell},\app{\ell}) }_{\LHX}^2 \ds }^{\frac{1}{2}} }_p \nonumber\\ 
    &\le K C_\alpha \Phi_p(m) \cdot h^\alpha  \p[\Big]{ \sum_{\ell=0}^{m-1}\int_{\t{\ell}}^{\t{\ell+1}} \norm{ G(\t{\ell},\app{\ell}) }_{p,\LHY}^2 \ds }^{\frac{1}{2}} \nonumber\\
    &\le K C_\alpha \Phi_p(m)\cdot h^\alpha  \p[\Big]{ \sum_{\ell=0}^{m-1}\int_{\t{\ell}}^{\t{\ell+1}} \p{\LGY\p{1+\norm{ \app{\ell} }_{p,Y}}+\norm{ g(\t{\ell}) }_{p,\LHY}}^2 \ds }^{\frac{1}{2}} \nonumber\\
    &\le K C_\alpha \p[\big]{\LGY\CstabpY+\tnorm{ g }_{\infty,p,Y}} \sqrt{T}\cdot \sqrt{\max\{\log(m), p\}} \cdot h^\alpha.   
\end{align}

We reason analogously for $T_{R,4}$, applying first the logarithmic square function estimate, then Itô's isomorphism from Theorem \ref{thm:maximal-inequality}, and take linear growth of $G'G-\tilde{g}$ on $Y$ into account. Thus, we obtain
\begin{align}
\label{eq:TR4}
    T_{R,4} &\leq K C_\alpha \Phi_p(m)\cdot h^\alpha\cdot\p[\Big]{ \sum_{\ell=0}^{m-1}\int_{\t{\ell}}^{\t{\ell+1}} \norm[\Big]{\int_{\t{\ell}}^s (G'G)(\t{\ell},\app{\ell})\dWHr }_{p,\LHY}^2 \ds }^{\frac{1}{2}}\nonumber\\
    &\leq K B_p C_\alpha \Phi_p(m) \cdot h^\alpha\cdot  \p[\Big]{ \sum_{\ell=0}^{m-1}\frac{h^2}{2} \p[\big]{\LGpGY\p{1+\norm{\app{\ell} }_{p,Y}}+\bnorm{\tilde{g}(\t{\ell})}_{p,Y}}^2 }^{\frac{1}{2}}\nonumber\\
    &\leq \frac{K B_p}{\sqrt{2}} C_\alpha \p[\big]{ \LGpGY\CstabpY+\bnorm{\tilde{g}}_{\infty,p,Y}}\sqrt{T} \cdot \sqrt{\max\{\log(m), p\}} \cdot h^{\alpha+\frac{1}{2}}.
\end{align}

Collecting the estimates \eqref{eq:TR1}--\eqref{eq:TR4} and estimating $h\le T$ for higher-order terms as well as $\log(m)\le \log(M)$, the desired estimate follows.
\end{proof}

We can now pass to the proof of the main theorem.

\begin{proof}[Proof of Theorem \ref{thm:main}]
    Let $1 \le m \le M$ and split the pathwise uniform error up to time $\t{m}$ into
    \begin{align}
    \label{eq:defEm123}
    E(m) &\ce \norm[\Big]{\max_{0\leq j\leq m}\norm{\sol{\t{j}} - \app{j} }_X}_p= \norm[\Big]{\max_{1\leq j\leq m}\norm{\sol{\t{j}} - \app{j} }_X}_p \nonumber\\
    &\leq \norm[\Big]{\max_{1\leq j\leq m}\norm{\sol{\t{j}} - v_j^1 }_X}_p + \norm[\Big]{\max_{1\leq j\leq m}\norm{v_j^1 - v_j^2}_X}_p + \norm[\Big]{\max_{1\leq j\leq m}\norm{v_j^2 - \app{j}}_X}_p 
    \ec E_1 + E_2 + E_3,
\end{align}
where $E_k=E_k(m)$ for $k=1,2,3$ and we have used that $\sol{0}=\app{0}=\init$ in the first line. In the following, we omit the dependence on $m$ in the notation of the error terms. 
Lemmas \ref{lem:E2} and \ref{lem:E3} provide error bounds for $E_2$ and $E_3$, which yield
 \begin{align}
\label{eq:E23}
    E_2 + E_3 &\le C_4 \p[\Big]{ \h \sum_{i=0}^{m-1} E(i)^2 }^{\frac{1}{2}} + \tilde{C}_3\sqrt{\max\{\log(M), p\}} \cdot h^\alpha\\
    &\phantom{\le }+ C_\alpha \p[\big]{ \norm{\init}_{p,Y} + \p[\big]{\LFY \CstabpY+\|f\|_{\infty,p,Y} }T}\cdot h^\alpha. \nonumber
\end{align}

It remains to perform the error analysis for $E_1$, which we split into the contributions of $F$ and $G$ using \eqref{eq:defvj1} via
\begin{align}
\label{eq:defTFTG}
    E_1 &\leq \bigg\| \max_{1\leq j\leq m}\Big\| \sum_{i=0}^{j-1} \int_{\t{i}}^{\t{i+1}} \sg{\t{j}-s}F(s,\sol{s})
            - \sg{\t{j}-\t{i}}F(\t{i},\sol{\t{i}}) \ds \Big\|_X\bigg\|_p \nonumber\\
    &\phantom{\le }+ \bigg\| \max_{1\leq j\leq m}\Big\| \sum_{i=0}^{j-1} \int_{\t{i}}^{\t{i+1}} \Big(\sg{\t{j}-s}G(s,\sol{s})- \sg{\t{j}-\t{i}}\nonumber\\
    &\phantom{\le + \|}
            \cdot\p[\Big]{ G(\t{i},\sol{\t{i}}) + \int_{\t{i}}^{s} (G'G)(\t{i},\sol{\t{i}})\dWHr} \Big)\dWHs \Big\|_X\bigg\|_p \ec T_F + T_G.
\end{align}
We further split $T_F$ as
\begin{align*}
    T_F &\leq \bigg\| \max_{1\leq j\leq m} \Big\| \sum_{i=0}^{j-1} \int_{\t{i}}^{\t{i+1}} \br[\big]{\sg{\t{j}-s}-\sg{\t{j}-\t{i}}} F(s,\sol{s}) \ds \Big\|_X\bigg\|_p \\
    &\phantom{\le }+ \bigg\| \max_{1\leq j\leq m} \Big\| \sum_{i=0}^{j-1} \int_{\t{i}}^{\t{i+1}} \sg{\t{j}-\t{i}} \br[\big]{F(s,\sol{s})-F(\t{i},\sol{s})} \ds \Big\|_X\bigg\|_p \\
    &\phantom{\le }+ \bigg\| \max_{1\leq j\leq m} \Big\| \sum_{i=0}^{j-1} \int_{\t{i}}^{\t{i+1}} \sg{\t{j}-\t{i}} \br[\big]{F(\t{i},\sol{s})-F(\t{i},\sg{s-\t{i}}\sol{\t{i}})} \ds \Big\|_X\bigg\|_p \\
    &\phantom{\le }+ \bigg\| \max_{1\leq j\leq m} \Big\| \sum_{i=0}^{j-1} \int_{\t{i}}^{\t{i+1}} \sg{\t{j}-\t{i}} \br[\big]{F(\t{i},\sg{s-\t{i}}\sol{\t{i}})-F(\t{i},\sol{\t{i}})} \ds \Big\|_X\bigg\|_p \\
    &\ec T_{F,1} + T_{F,2} + T_{F,3} + T_{F,4},
\end{align*}
introducing an additional error term by adding zero to facilitate the application of Lemma \ref{lem:regularity-UtDiffVst-Z} with $Z=Y$ for $T_{F,3}$. The semigroup estimate of Lemma \ref{lem:semigroup-interpolation}, linear growth of $F-f$ on $Y$, and the a priori estimate \eqref{eq:WPboundCWPqZ} on $Y$ yield 
\begin{align}
\label{eq:TF1}
    T_{F,1} 
    &\leq \bigg\| \max_{1\leq j\leq m} \sum_{i=0}^{j-1} \int_{\t{i}}^{\t{i+1}} \norm{ \sg{\t{j}-s}-\sg{\t{j}-\t{i}} }_{\calL(Y,X)} \norm{ F(s,\sol{s}) }_Y \ds \bigg\|_p \nonumber\\
    &\leq 2\CY \p[\bigg]{\sup_{r\in[0,T]}\norm{ F(r,\sol{r}) }_{p,Y}}\sum_{i=0}^{m-1} \int_{\t{i}}^{\t{i+1}} \p{ s-\t{i} }^\alpha \ds \leq \frac{2\CY}{\alpha+1} (\LFY\CWPpY+\|f\|_{\infty,p,Y}) T \cdot \h^\alpha.
\end{align}

Temporal Hölder continuity of $F$ combined with contractivity of the semigroup implies the bound
\begin{align}
\label{eq:TF2}
    T_{F,2} &\leq \norm[\bigg]{ \max_{1\leq j\leq m} \sum_{i=0}^{j-1}  \int_{\t{i}}^{\t{i+1}} \norm[\big]{\sg{\t{j}-\t{i}} \br[\big]{F(s,\sol{s})-F(\t{i},\sol{s})}}_X \ds }_p \nonumber\\
    &\leq \sum_{i=0}^{m-1} \int_{\t{i}}^{\t{i+1}} \norm[\bigg]{\sup_{x \in X}\norm{ F(s,x)-F(\t{i},x) }_X}_p \ds 
    \leq \CalphaF \sum_{i=0}^{m-1} \int_{\t{i}}^{\t{i+1}} (s-\t{i})^\alpha \ds\leq \frac{\CalphaF }{\alpha+1} T\cdot \h^\alpha.
\end{align}

Abbreviate $V_s^i\ce\sg{s-\t{i}}\sol{\t{i}}$. The term $T_{F,3}$ is further split using the Taylor expansion
\begin{align}
\label{eq:taylorF}
    F(\t{i},\sol{s}) &= F(\t{i},V_s^i) + F'(\t{i},V_s^i)\br{\sol{s}-V_s^i}\nonumber\\
    &\phantom{= }+ \int_{0}^{1} \big( F'(\t{i},V_s^i+\zeta(\sol{s}-V_s^i))-F'(\t{i},V_s^i) \big)\br{\sol{s}-V_s^i} \dzeta,
\end{align}
which is well-defined by Proposition \ref{prop:TaylorExpansionFrechet}, which is applicable because $\sol{s},V_s^i \in Y$ almost surely by Theorem \ref{thm:wellposedness-Z} with $Z=Y$ and $S(t)\in \calL(Y)$ for all $t\in [0,T]$.
In the second term of the Taylor expansion, we insert the mild solution formula \eqref{eq:mild-solution-consequence} relating $\sol{s}$ and $V_s^i$.
Note that the latter is crucial: estimating $U_s-V_s^i$ directly via Lemma \ref{lem:regularity-UtDiffVst-Z} for $Z=X$ instead of exploiting the specific structure of the term as below only gives decay of order $1/2$. The Taylor expansion results in
\begin{align*}
    &T_{F,3} \leq \bigg\| \max_{1\leq j\leq m}\Big\| \sum_{i=0}^{j-1} \int_{\t{i}}^{\t{i+1}} \sg{\t{j}-\t{i}} F'(\t{i},V_s^i)\br[\Big]{\int_{\t{i}}^{s}\sg{s-r}F(r,\sol{r})\dr} \ds \Big\|_X\bigg\|_p \\
    &+ \bigg\| \max_{1\leq j\leq m}\Big\| \sum_{i=0}^{j-1} \int_{\t{i}}^{\t{i+1}} \sg{\t{j}-\t{i}} F'(\t{i},V_s^i)\br[\Big]{\int_{\t{i}}^{s}\sg{s-r}G(r,\sol{r}) \dWHr} \ds \Big\|_X\bigg\|_p \\
    &+ \bigg\| \max_{1\leq j\leq m}\Big\| \sum_{i=0}^{j-1} \int_{\t{i}}^{\t{i+1}} \sg{\t{j}-\t{i}} \int_{0}^{1} \big( F'(\t{i},V_s^i+\zeta(\sol{s}-V_s^i)) -F'(\t{i},V_s^i) \big) \br{\sol{s}-V_s^i} \dzeta \ds \Big\|_X\bigg\|_p\\
    &\ec T_{F,3,1}+ T_{F,3,2}+ T_{F,3,3}.
\end{align*}

Uniform boundedness of $F'\colon X \to \calL(X)$ as in Lemma \ref{lem:GateauxDerivUnifBdd} can be employed. Here, we use the identification of Gâteaux derivatives of $F|_Y\colon Y\to X$ and $F\colon X \to X$ discussed in the proof of Proposition \ref{prop:TaylorExpansionFrechet}. Hence, via contractivity of the semigroup, Lipschitz continuity and thus linear growth of $F$ on $X$ as well as the a priori estimate on $X$, we obtain
\begin{align}
\label{eq:TF31}
    T_{F,3,1} &\leq \sum_{i=0}^{m-1} \int_{\t{i}}^{\t{i+1}} \norm[\Big]{\norm{ F'(\t{i},V_s^i)}_{\calL(X)} \int_{\t{i}}^{s}\|\sg{s-r}F(r,\sol{r})\|_X\dr }_p \ds \nonumber\\
    &\leq \CFp \sum_{i=0}^{m-1} \int_{\t{i}}^{\t{i+1}} (s-\t{i}) \sup_{t\in[0,T]} \norm{ F(t,\sol{t}) }_{p,X} \ds \leq \frac{\CFp}{2}  (\LFX\CWPpX+\|f\|_{\infty,p,X}) T\cdot \h.
\end{align}
The subsequent term is a particularly interesting one, as estimating it by leveraging the stochastic Fubini theorem enables us to consider pathwise uniform errors for general $p\in [2,\infty)$, which was not feasible using the orthogonality argument employed in \cite[Sec.~6.1]{JentzenRoeckner15_Milstein} for the mean square error. More precisely, Lemma~\ref{lem:linearityHSstochInt} and the stochastic Fubini theorem \cite[Thm.~2.2]{stochasticFubini} allow us to rewrite $T_{F,3,2}$ in such a way that we can make use of the maximal inequality from Theorem~\ref{thm:maximal-inequality}. Consequently, contractivity of the semigroup, uniform boundedness of $F'$ as in Lemma~\ref{lem:GateauxDerivUnifBdd}, Lipschitz continuity of $G$, and the a priori estimate on $Z=X$ from Theorem~\ref{thm:wellposedness-Z} imply 
\begin{align}
\label{eq:TF32}
    T_{F,3,2} &\leq \bigg\| \sup_{t\in[0,t_m]}\Big\| \int_{0}^{t} \sg{t-r} \int_{r}^{\ceilh{r}} \sg{r-\floorh{r}}F'(\floorh{r},\sg{s-\floorh{s}}\sol{\floorh{s}})\br{\sg{s-r}G(r,\sol{r})} \ds \dWHr \Big\|_X\bigg\|_p \nonumber\\
    &\leq B_p \p[\Big]{\int_{0}^{t_m} \Big\|\int_{r}^{\ceilh{r}} \sg{r-\floorh{r}} F'(\floorh{r},\sg{s-\floorh{s}}\sol{\floorh{s}})\br{\sg{s-r}G(r,\sol{r})} \ds\Big\|_{p,\LHX}^2 \dr}^{\frac{1}{2}}  \nonumber\\
    &\leq B_p \p[\Big]{\int_{0}^{t_m} \Big(\int_{r}^{\ceilh{r}} \big\|\|F'(\floorh{r},\sg{s-\floorh{s}}\sol{\floorh{s}})\|_{\calL(X)}
    \|\sg{s-r}G(r,\sol{r})\|_{\LHX}\big\|_p \ds\Big)^2 \dr}^{\frac{1}{2}}\nonumber\\
    &\leq B_p \CFp \p[\Big]{\int_{0}^{t_m} (\ceilh{r}-r)^2 \dr}^{\frac{1}{2}}\bigg(\sup_{t\in[0,T]}\big\|G(t,\sol{t})\big\|_{p,\LHX}\bigg)\nonumber\\
    &\leq \frac{B_p \CFp}{\sqrt{3}} (\LGX\CWPpX+\tnorm{g}_{\infty,p,X})\sqrt{T}\cdot h.
\end{align}
The $(2\alpha-1)$-Hölder continuity of $F'$ becomes relevant when estimating the last term of $T_{F,3}$. Together with contractivity of the semigroup and the regularity result on $Y$ from Lemma \ref{lem:regularity-UtDiffVst-Z} with $q=2\alpha p$ and $Z=Y$, we obtain 
\begin{align}
\label{eq:TF33}     
    T_{F,3,3} &\le \sum_{i=0}^{m-1} \int_{\t{i}}^{\t{i+1}} \int_{0}^{1} \big\| \norm{ F'(\t{i},V_s^i+\zeta(\sol{s}-V_s^i)) - F'(\t{i},V_s^i) }_{\calL(Y,X)} \norm{ \sol{s}-V_s^i }_Y \big\|_p \dzeta \ds \nonumber\\
    &\le \CbetaF \sum_{i=0}^{m-1} \int_{\t{i}}^{\t{i+1}} \int_{0}^{1} \big\| \norm{\zeta( \sol{s} - V_s^i) }_Y^{2\alpha-1} \norm{ \sol{s}-V_s^i }_Y \big\|_p \dzeta \ds \nonumber\\
    &\le \CbetaF \sum_{i=0}^{m-1} \int_{\t{i}}^{\t{i+1}} \int_{0}^{1} \zeta^{2\alpha-1} \norm{ \sol{s}-V_s^i }_{2\alpha p,Y}^{2\alpha} \dzeta \ds \nonumber\\
    &\le \frac{\CbetaF}{2\alpha} \sum_{i=0}^{m-1} \int_{\t{i}}^{\t{i+1}} \big(L_{2\alpha p,Y,1}(s-\t{i})+L_{2\alpha p,Y,2}(s-\t{i})^{1/2}\big)^{2\alpha} \ds \nonumber\\
    &\le \frac{\CbetaF}{\alpha}  T \p[\bigg]{\frac{1}{2\alpha+1}L_{2\alpha p,Y,1}^{2\alpha} \cdot h^{2\alpha}+\frac{1}{\alpha+1}  L_{2\alpha p,Y,2}^{2\alpha}\cdot h^\alpha}.
\end{align}

For $T_{F,4}$, contractivity of the semigroup, Lipschitz continuity of $F$, the semigroup estimate from Lemma \ref{lem:semigroup-interpolation}, and the a priori estimate \eqref{eq:WPboundCWPqZ} on $Y$ imply
\begin{align}
\label{eq:TF4}
    T_{F,4} &\leq \norm[\Big]{ \max_{1\leq j\leq m} \sum_{i=0}^{j-1} \int_{\t{i}}^{\t{i+1}} \big\| \sg{\t{j}-\t{i}} \br[\big]{F(\t{i},\sg{s-\t{i}}\sol{\t{i}})-F(\t{i},\sol{\t{i}})} \big\|_X \ds }_p \nonumber\\
    &\leq \LFX \Big\| \sum_{i=0}^{m-1} \int_{\t{i}}^{\t{i+1}} \norm[\big]{ [\sg{s-\t{i}}-\Id]\sol{\t{i}}}_X \ds \Big\|_p\nonumber\\
    & \leq 2\CY\LFX \sum_{i=0}^{m-1} \int_{\t{i}}^{\t{i+1}}(s-\t{i})^\alpha \norm{ \sol{\t{i}} }_{p,Y} \ds \leq \frac{2\CY\LFX}{\alpha+1}\CWPpY T \cdot \h^\alpha.
\end{align}

Recall the definition of $T_G$ in \eqref{eq:defTFTG}, which is the term we estimate next. After an application of the maximal inequality from Theorem \ref{thm:maximal-inequality}, it can be split using the triangle inequality in $L^2(0,t_m;L^p(\Omega;\LHX))$ and, for $T_{G,2}$ to $T_{G,5}$, contractivity of the semigroup. Due to the presence of the Milstein term in $T_G$, a different splitting than for $T_F$ is obtained. Namely,
\begin{align*}
    T_G 
    &\le B_p \p[\Big]{\int_0^{\t{m}} \Big\|G(s,\sol{s})- \sg{s-\floorh{s}}\p[\Big]{ G(\floorh{s},\sol{\floorh{s}}) + \int_{\floorh{s}}^{s} (G'G)(\floorh{s},\sol{\floorh{s}})\dWHr} \Big\|_{p,\LHX}^2 \ds}^{1/2}\\
    &\le B_p \p[\bigg]{\p[\Big]{\sum_{i=0}^{m-1}\int_{\t{i}}^{\t{i+1}}\big\|[I- \sg{s-\t{i}}]G(s,\sol{s}) \big\|_{p,\LHX}^2 \ds}^{1/2}\\
    &\phantom{\le }+ \p[\Big]{\sum_{i=0}^{m-1}\int_{\t{i}}^{\t{i+1}} \|G(s,\sol{s})- G(\t{i},\sol{s}) \|_{p,\LHX}^2 \ds}^{1/2}\\
    &\phantom{\le }+ \p[\Big]{\sum_{i=0}^{m-1}\int_{\t{i}}^{\t{i+1}} \Big\|G(\t{i},\sol{s})- G(\t{i},V_s^i)- \int_{\t{i}}^{s} (G'G)(\t{i},V_s^i)\dWHr \Big\|_{p,\LHX}^2 \ds}^{1/2}\\
    &\phantom{\le }+ \p[\Big]{\sum_{i=0}^{m-1}\int_{\t{i}}^{\t{i+1}} \|G(\t{i},V_s^i) - G(\t{i},\sol{\t{i}}) \|_{p,\LHX}^2 \ds}^{1/2}\\
    &\phantom{\le }+ \p[\Big]{\sum_{i=0}^{m-1}\int_{\t{i}}^{\t{i+1}} \Big\| \int_{\t{i}}^{s} (G'G)(\t{i},V_s^i)- (G'G)(\t{i},\sol{\t{i}})\dWHr \Big\|_{p,\LHX}^2 \ds}^{1/2}}\\
    &\ec T_{G,1}+T_{G,2}+T_{G,3}+T_{G,4}+T_{G,5}.
\end{align*}
Now, estimates of the first, second, and fourth term can be obtained analogously to the ones of $T_{F,1}$, $T_{F,2}$, and $T_{F,4}$ in \eqref{eq:TF1}, \eqref{eq:TF2}, and \eqref{eq:TF4}, respectively, via the corresponding properties of $G$. This results in
\begin{align}
    \label{eq:TG1}
    T_{G,1} &\le \frac{2B_p\CY}{\sqrt{2\alpha+1}}(\LGY\CWPpY+\tnorm{g}_{\infty,p,Y})\sqrt{T}\cdot \h^\alpha,\\
    \label{eq:TG2}
    T_{G,2}&\leq \frac{B_p \CalphaG}{\sqrt{2\alpha+1}} \sqrt{T} \cdot \h^\alpha,\\
    \label{eq:TG4}
    T_{G,4} &\le \frac{2B_p\CY}{\sqrt{2\alpha+1}}\LGX\CWPpY\sqrt{T}\cdot \h^{\alpha}.
\end{align}

It remains to estimate the novel terms $T_{G,3}$ and $T_{G,5}$. In order to estimate the former, a Taylor expansion of $G(\t{i},\sol{s})$ analogous to \eqref{eq:taylorF} is employed, which exists due to Proposition \ref{prop:TaylorExpansionFrechet}. Note that unlike for $T_{F,3}$, we do not insert the mild solution formula in the Taylor expansion, since we can leverage the Milstein term to obtain higher-order convergence, which was not feasible for $T_{F,3}$. Thus, we can split $T_{G,3}$ as
\begin{align*}
    T_{G,3} &\le B_p\p[\Big]{\sum_{i=0}^{m-1}\int_{\t{i}}^{\t{i+1}} \norm[\Big]{G'(\t{i},V_s^i)\br[\Big]{\sol{s}-V_s^i-\int_{\t{i}}^{s}G(\t{i},V_s^i)\dWHr} }_{p,\LHX}^2 \ds}^{1/2}\\
    &\phantom{\le }+B_p\p[\Big]{\sum_{i=0}^{m-1}\int_{\t{i}}^{\t{i+1}} \norm[\Big]{ \int_0^1 \p[\big]{G'(\t{i},V_s^i+\zeta(\sol{s}-V_s^i))-G'(\t{i},V_s^i)}\br[\big]{\sol{s}-V_s^i} \dzeta }_{p,\LHX}^2 \ds}^{1/2}\\
    &\ec T_{G,3,1}+T_{G,3,2},
\end{align*}
where the $(G'G)$-term was rewritten as in \eqref{eq:rewriteGpG}.
As a consequence of uniform boundedness of $G'\colon X \to \calL(X,\LHX)$ as in Lemma \ref{lem:GateauxDerivUnifBdd}, higher-order decay of the difference of solution and stochastic integral terms in Lemma \ref{lem:regularity-DiffWithStInt-X}, and the triangle inequality in $L^2(0,t_m)$,
\begin{align}
\label{eq:TG31}
    T_{G,3,1} &\le B_p\p[\Big]{\sum_{i=0}^{m-1}\int_{\t{i}}^{\t{i+1}} \Big\|\|G'(\t{i},V_s^i)\|_{\calL(X,\LHX)}\norm[\Big]{\sol{s}-V_s^i-\int_{\t{i}}^{s}G(\t{i},V_s^i)\dWHr}_X \Big\|_p^2 \ds}^{1/2}\nonumber\\
    &\le B_p\CGp\p[\Big]{\sum_{i=0}^{m-1}\int_{\t{i}}^{\t{i+1}} \p[\big]{L_1(s-\t{i})+L_2(s-\t{i})^{3/2}+L_3(s-\t{i})^{\alpha+\frac{1}{2}} }^2 \ds}^{1/2}\nonumber\\
    &\le B_p\CGp \sqrt{T} \p[\Big]{\frac{L_1}{\sqrt{3}} \cdot \h + \frac{L_2}{2} \cdot \h^{3/2} + \frac{L_3}{\sqrt{2\alpha+2}} \cdot \h^{\alpha+\frac{1}{2}}}.
\end{align}

Analogous to the estimate of $T_{F,3,3}$ in \eqref{eq:TF33}, we can estimate $T_{G,3,2}$ by $(2\alpha-1)$-Hölder continuity of $G'\colon Y \to \calL(Y,\LHX)$ and the regularity estimate of Lemma \ref{lem:regularity-UtDiffVst-Z} with $q=2\alpha p$ and $Z=Y$ as
\begin{align}
\label{eq:TG32}
    T_{G,3,2} &\le \frac{B_p\CbetaG}{\alpha}  \sqrt{T} \p[\Big]{\frac{1}{\sqrt{4\alpha+1}}L_{2\alpha p,Y,1}^{2\alpha}\cdot \h^{2\alpha}
    + \frac{1}{\sqrt{2\alpha+1}}L_{2\alpha p,Y,2}^{2\alpha}\cdot \h^\alpha}.
\end{align}

To estimate the last term $T_G$ is composed of, Lipschitz continuity of $G'G$ and the regularity estimate of Lemma \ref{lem:regularity-UtDiffVst-Z} with $q=p$ and $Z=X$ are required. We deduce using Lemma \ref{lem:linearityHSstochInt}, Itô's isomorphism from Theorem \ref{thm:maximal-inequality}, the isometric isomorphism $\calL_2(H,\LHX)\cong \calL_2^{(2)}(H,X)$, and the triangle inequality in $L^2(0,t_m;\R)$ that
\begin{align}
\label{eq:TG5}
    T_{G,5}&\le B_p^2 \p[\Big]{\sum_{i=0}^{m-1}\int_{\t{i}}^{\t{i+1}}  \int_{\t{i}}^{s}\norm{(G'G)(\t{i},V_s^i) - (G'G)(\t{i},\sol{\t{i}}) }_{p,\bilinHX}^2 \dr  \ds}^{1/2}\nonumber\\
    &\le B_p^2 \LGpGX \p[\Big]{\sum_{i=0}^{m-1}\int_{\t{i}}^{\t{i+1}}  (s-\t{i})\norm{V_s^i - \sol{\t{i}}}_{p,X}^2  \ds}^{1/2}\nonumber\\
    &\le B_p^2 \LGpGX \sqrt{T} \p[\Big]{\frac{\LpXone}{2} \cdot \h^{3/2} + \frac{\LpXtwo}{\sqrt{3}}\cdot \h}.
\end{align}

In conclusion of the estimates \eqref{eq:TF1}, \eqref{eq:TF2}, and \eqref{eq:TF31}--\eqref{eq:TF4} for $T_F$ as well as \eqref{eq:TG1}--\eqref{eq:TG5} for $T_G$ combined with the bound \eqref{eq:E23} for $E_2$ and $E_3$ inserted in \eqref{eq:defEm123}, we obtain
\begin{align*}
    E(m)&\le \tilde{C}_1h+\p[\big]{\tilde{C}_2+\tilde{C}_3\sqrt{\max\{\log(M),p\}}}h^\alpha+ C_4 \p[\Big]{ \h \sum_{i=0}^{m-1} E(i)^2 }^{\frac{1}{2}}
\end{align*}
with $C_4$ as defined in Lemma \ref{lem:E2} and constants     
\begin{align}
\label{eq:constantDefinition}
        \tilde{C}_1&\ce \frac{1}{2}\big(\CFp\big( \LFX\CWPpX+\|f\|_{\infty,p,X}\big)   +B_p\p[\big]{\CGp L_2 + B_p \LGpGX\LpXone }\big)T\\
    &\phantom{= }+ \frac{B_p}{\sqrt{3}}\p[\big]{\CFp \big(\LGX\CWPpX+\tnorm{g}_{\infty,p,X}\big) +\CGp L_1 +B_p \LGpGX\LpXtwo }\sqrt{T},\nonumber\\
    \tilde{C}_2 &\ce \p[\Big]{\frac{C_{2,1}}{\alpha+1}+C_{2,2}}T+\frac{B_p C_{2,3}}{\sqrt{2\alpha+1}}\sqrt{T}+\frac{B_p\CbetaG}{\alpha\sqrt{4\alpha+1}}L_{2\alpha p,Y,1}^{2\alpha}T^{\alpha+\frac{1}{2}}+C_\alpha\|\init\|_{p,Y},\nonumber\\
    \tilde{C}_3&\ce  KC_\alpha \p[\Big]{  \p[\big]{\LGY\CstabpY+\tnorm{ g }_{\infty,p,Y}} \sqrt{T}
    +\frac{B_p}{\sqrt{2}}  \p[\big]{ \LGpGY\CstabpY+\bnorm{\tilde{g}}_{\infty,p,Y}}T},\nonumber\\
    C_{2,1}&\ce 2\CY\p[\big]{(\LFX+\LFY)\CWPpY+\|f\|_{\infty,p,Y}}+\CalphaF+\frac{\CbetaF}{\alpha}L_{2\alpha p,Y,2}^{2\alpha},\nonumber\\
    C_{2,2} &\ce \frac{B_p\LGX L_3}{\sqrt{2\alpha+2}} + \frac{\CbetaF}{\alpha(2\alpha+1)}L_{2\alpha p,Y,1}^{2\alpha}T^\alpha+C_\alpha\big(\LFY\CstabpY+\|f\|_{\infty,p,Y}\big),\nonumber\\
    C_{2,3}&\ce 2\CY\p[\big]{(\LGX+\LGY)\CWPpY+\tnorm{g}_{\infty,p,Y}}+\CalphaG+\frac{\CbetaG}{\alpha}L_{2\alpha p,Y,2}^{2\alpha}. \nonumber
\end{align}
Here, $\CWPpX,\CWPpY$ are as defined in \eqref{eq:defCWPqZ}, $\CstabpY$ as in Proposition \ref{prop:pointwise-stability-Y},
$L_{p,X,i},L_{2\alpha p,Y,i}$ for $i=1,2$ as in Lemma \ref{lem:regularity-UtDiffVst-Z}, $L_1,L_2,L_3$ as in Lemma \ref{lem:regularity-DiffWithStInt-X}, and $K=4\exp(1+\frac{1}{2\e})$.

Lastly, the discrete Grönwall Lemma \ref{lem:gronwall} allows us to deduce for all $0 \le m \le M$
\begin{align*}
    E(m) \le (1+C_4^2t_m)^{1/2}\exp\p[\Big]{\frac{1+C_4^2t_m}{2}}\p[\big]{\tilde{C}_1h+\p[\big]{\tilde{C}_2+\tilde{C}_3\sqrt{\max\{\log(M),p\}}}h^\alpha},
\end{align*}
from which the statement of the theorem follows for $m=M$ with $t_M = T$ noting that $E(M) \lesssim \sqrt{\max\{\log(M),p\}} \cdot \h^{\alpha}$.
\end{proof}

The logarithmic correction factor can be omitted for the exponential Milstein scheme, as it is not necessary to use the logarithmic square function estimate for terms not containing the semigroup.

\begin{theorem}[Exponential Milstein scheme]
\label{thm:expMilstein}
 Let $p\in[2,\infty)$ and $\alpha \in (\frac{1}{2},1]$ such that Assumptions~\ref{ass:spacesSemigroup} and~\ref{ass:FG_rate_I} hold for $\alpha$ and $q=2\alpha p$, and Assumption \ref{ass:FG_rate_II} holds for $\alpha$ and $q=p$. Let $\xi \in L_{\F_0}^{2\alpha p}(\Omega;Y)$. Denote by $U$ the solution of \eqref{eq:SEEconv} and by $u=(u_j)_{j=0,\ldots,M}$ the exponential Milstein scheme.
 
    Then the exponential Milstein scheme converges at rate $\alpha$  as $h \to 0$ and for $M\ge 2$
    \begin{equation*}
        \bigg\|\max_{0 \le j \le M} \|\sol{\t{j}}-\app{j}\|_X \bigg\|_{L^p(\Omega)} \le C_1 h+C_{2,\mathrm{expM}}\cdot h^\alpha
    \end{equation*}
    with $C_1$ as defined in Theorem \ref{thm:main} and, using the notation from Theorem \ref{thm:main}, $C_{2,\mathrm{expM}}\ce C_{\mathrm{e}}\tilde{C}_{2,\mathrm{expM}}$,
    \begin{align*}
        \tilde{C}_{2,\mathrm{expM}}
        &\ce \Big(\frac{C_{2,1}}{\alpha+1}+\frac{B_p\LGX L_3}{\sqrt{2\alpha+2}}\Big)T  + \frac{B_pC_{2,3}}{\sqrt{2\alpha+1}}\sqrt{T}
        + \frac{ L_{2\alpha p,Y,1}^{2\alpha}}{\alpha}\Big(\frac{\CbetaF}{2\alpha+1}T^{\alpha+1}+\frac{B_p\CbetaG}{\sqrt{4\alpha+1}}T^{\alpha+\frac{1}{2}}\Big).
    \end{align*}
\end{theorem}
\begin{proof}
    Since $R_h^j=\sg{\t{j}}$ and $R_h^{j-i}=\sg{\t{j}-\t{i}}$ for the exponential Milstein scheme, $v_j^2=\app{j}$ and thus the error term $E_3$ vanishes. Repeating the proof of Theorem \ref{thm:main} for the remaining error terms yields the statement of the theorem.
\end{proof}

We comment on the case of lower rates $\alpha \in (0,1/2]$ and make two observations regarding possible generalisations of Theorems \ref{thm:main} and \ref{thm:expMilstein}.

\begin{remark}
\label{rem:alpha012}
    An analogue of Theorem \ref{thm:main} can be obtained for lower rates of convergence $\alpha \in (0,1/2]$. Suppose that the assumptions of Theorem \ref{thm:main} hold for some $\alpha \in (0,1/2]$ except Assumptions \ref{ass:FG_rate_I}\ref{assCond:FG_rate_GateauxDiffble} on $F$ and \ref{ass:FG_rate_II}\ref{assCond:FG_rate_GateauxHoelder}. Further let $\init\in L^p(\Omega;Y)$ and let Assumption \ref{ass:FG_rate_I} hold for $q=p$. Then \eqref{eq:mainEstimate} holds for some $C_1,C_2,C_3\ge 0$.

    The excluded assumptions were used in the estimates of $T_{F,3}$ and $T_{G,3,2}$ in \eqref{eq:TF31}--\eqref{eq:TF33} and~\eqref{eq:TG32}, respectively. For lower rates $\alpha \le \frac{1}{2}$, a Taylor expansion for $F$ is no longer necessary. Instead, we estimate $T_{F,3}$ directly using Lipschitz continuity of $F$ and Lemma \ref{lem:regularity-UtDiffVst-Z} with $q=p$ and $Z=X$ via
    \begin{align*}
        T_{F,3} &\le \LFX\sum_{i=0}^{m-1}\int_{\t{i}}^{\t{i+1}} \norm{\sol{s}-\sg{s-\t{i}}\sol{\t{i}}}_{p,X} \ds
        \lesssim_{p,T} \sum_{i=0}^{m-1}\int_{\t{i}}^{\t{i+1}} (s-\t{i}) +(s-\t{i})^{1/2} \ds \lesssim h + h^{1/2}.
    \end{align*}
    To avoid a Taylor expansion for $G$, we split $T_G\le T_{G,1}+T_{G,2}+\bar{T}_{G,3}+\bar{T}_{G,4}$ differently. Here,
    \begin{align*}
        \bar{T}_{G,3}&\ce B_p \p[\Big]{\sum_{i=0}^{m-1}\int_{t_i}^{\t{i+1}} \|G(\t{i},\sol{s})-G(\t{i},\sol{\t{i}})\|_{p,\LHX}^2\ds}^{1/2},\\
        &\lesssim \p[\Big]{\sum_{i=0}^{m-1}\int_{\t{i}}^{\t{i+1}}\p[\big]{(s-\t{i})+(s-\t{i})^{1/2}+(s-\t{i})^\alpha}^2\ds}^{1/2} \lesssim \sqrt{T}(h+h^{1/2}+h^\alpha)
    \end{align*}
    can be estimated using Lipschitz continuity of $G$ and, instead of Lemma \ref{lem:regularity-UtDiffVst-Z}, the path regularity result \cite[Lemma~6.2]{KliobaVeraar24Rate}. 
    For $\bar{T}_{G,4}$ we can proceed similarly as for $T_{G,5}$ in \eqref{eq:TG5} without invoking Lemma \ref{lem:regularity-UtDiffVst-Z} to obtain additional decay. Simply applying Itô's isomorphism again, linear growth of $G'G-\tilde{g}$ on $X$, and pointwise stability on $X$ yield 
    \begin{align*}
        \bar{T}_{G,4}&\ce B_p \p[\Big]{\sum_{i=0}^{m-1}\int_{t_i}^{\t{i+1}} \norm[\Big]{\int_{\t{i}}^s (G'G)(\t{i},\sol{\t{i}})\dWHr}_{p,\LHX}^2\ds}^{1/2}
        \lesssim \sqrt{T}h^{1/2}.
    \end{align*} However, the Milstein scheme is not advantageous in this parameter range. Indeed, both the exponential Euler method and A-stable rational schemes without Milstein terms such as implicit Euler or Crank--Nicolson achieve pathwise uniform convergence at rate $\min\{\alpha,1/2\}$ up to a logarithmic correction factor for rational schemes \cite[Thm.~6.4,~Cor.~6.5]{KliobaVeraar24Rate} in the same setting without assumptions on $F',G'$, or $G'G$. At the same time,  their computational effort is lower, since no iterated stochastic integrals need to be simulated.
\end{remark}

\begin{remark}
\label{rem:FGateauxDiffonYenough}       
    The Gâteaux differentiability of $F\colon X \to X$ in all $x \in X$ in Assumption \ref{ass:FG_rate_I}\ref{assCond:FG_rate_GateauxDiffble} can be weakened to Gâteaux differentiability in all $x \in Y$, i.e.\ Gâteaux differentiability of $F\colon Y \to X$, omitting the arguments from $\Omega$ and $[0,T]$ in this remark. Indeed, Proposition \ref{prop:TaylorExpansionFrechet}, which ensures the existence of the Taylor expansion \eqref{eq:taylorF}, remains valid under this weaker assumption. In the estimates of $T_{F,3,1}$ and $T_{F,3,2}$ in \eqref{eq:TF31} and \eqref{eq:TF32}, the operator norm of $F'(u)$ in $\calL(Y,X)$ rather than $\calL(X)$ has to be used. Hence, the $Y$-norm of the argument of $F'(u)$ has to be estimated, which remains possible with the same order of decay in $h$ (albeit with larger constants), since only linear growth, the a priori estimate and the Lemma \ref{lem:regularity-UtDiffVst-Z} were used in the subsequent estimates.    
    This relaxation of Gâteaux differentiability is not possible for $G$, since $G'G\colon X \to \bilinHX$ is required to be Lipschitz in Assumption \ref{ass:FG_rate_II}\ref{assCond:FG_rate_GpGLipschitz}. For $G'G$ to even be defined on the entire space $X$, Gâteaux differentiability of $G$ on $X$ is necessary. 
    
    While the relaxed assumption on $F$ might seem more general at first glance, note that any Nemytskii operator associated to a $C^1$ Lipschitz function $\phi\colon \R\to\R$ is automatically Gâteaux differentiable on the full space $L^2(\calO;\R)$ for bounded domains $\calO \seq \R^d$ by Proposition \ref{prop:NemytskiiGateauxAffine}.   
\end{remark}

\begin{remark}
\label{rem:2smooth}
    Theorem \ref{thm:main} remains valid for Banach spaces $X=L^p(S)$ with $p\in [2,\infty)$ and $S$ a measure space if one replaces all instances of Hilbert--Schmidt operators $\phi \in \calL_2(H,Z)$ by $\gamma$-radonifying operators $\phi\in\gamma(H,Z)$ suitably. More generally, Theorem \ref{thm:main} can be generalised to any $2$-smooth Banach spaces $X,Y$ having Pisier's contraction property (cf.\ \cite[Sec.~2.1]{vNeeVer-Pin} and \cite[Def.~7.5.1]{AnalysisBanachSpacesII} for their definitions).    
    All technical tools used in the convergence proof carry over to this setting as outlined in~\cite[Sec.~3]{KliobaVeraar24Irregular}: Whenever Itô's isomorphism or the Burkholder--Davis--Gundy inequality are applied, one can instead make use of the maximal inequality for stochastic convolutions~\cite[Thm.~4.1]{vNeeVer-Pin}, replacing the constant $B_p$ by $B_{p,D}=10 D \sqrt{p}$ if the underlying space is $(2,D)$-smooth (see also the discussion after (2.3) in \cite{KliobaVeraar24Irregular}). The logarithmic square function estimate used in Lemma \ref{lem:E3} remains valid \cite[Thm.~3.1]{coxVanWInden2024logSquare} noting that $\sqrt{p+\log(j)}\le \sqrt{2}\sqrt{\max\{\log(j),p\}}$. Pisier's contraction principle is required to apply the $\gamma$-Fubini theorem \cite[Thm.~9.4.10]{AnalysisBanachSpacesII}, which entails that $\gamma(H,\gamma(H,X))\cong \gamma(\overline{H \otimes H},X)$ is an isometric isomorphism.  

    Pathwise uniform stability as in Remark \ref{rem:stability}\ref{remCond:pathwiseUniformStability} can also be extended to $2$-smooth Banach spaces, albeit with more care. Adapting the proof strategy based on martingale estimates in \cite[Lem.~3.5, Prop.~3.6]{KliobaVeraar24Irregular} to also incorporate iterated stochastic integral terms, pathwise uniform stability under a linear growth assumption on $G'G$ can be shown, leveraging that the iterated stochastic integrals $(G'G)(\t{i},\app{i})\Delta_2W_{i+1}$ form a martingale difference sequence.
\end{remark}

In the linear case, the necessary assumptions simplify considerably, since for example Lipschitz continuity is immediate and the Gâteaux derivatives $F'(u)[v]=F(v)$ and $G'(u)$ are independent of $u$, and thus trivially Hölder continuous.
In the following we also allow for affine linear coefficients.

\begin{corollary}[Linear case]
\label{cor:linearCase}
    Let $\alpha \in (0,1]$ such that Assumptions \ref{ass:spacesSemigroup} and \ref{ass:schemeR} hold and let $p\in[2,\infty)$. Define $p_\alpha \ce \max\{2\alpha p,p\}$, $q_\alpha\ce \frac{2\alpha p}{2\alpha-1}$ and $r_\alpha \ce q_\alpha$ for $\alpha>\frac{1}{2}$ as well as $q_\alpha \ce \infty$ and $r_\alpha \ce p$ else.  
    For $Z\in\{X,Y\}$, let $F\colon \Omega \times [0,T] \times Z \to Z,\,F(\cdot,\cdot,x)\ce F_0(\cdot,\cdot)x + f$ and $G\colon \Omega \times [0,T] \times Z \to \LHZ,\,G(\cdot,\cdot,x)\ce G_0(\cdot,\cdot)x + g$
    be strongly $\calP \otimes \calB(Z)$-measurable with $F_0\colon \Omega \times [0,T] \to \calL(Z)$, $f\colon \Omega \times [0,T] \to Z$, $G_0\colon \Omega \times [0,T] \to \calL(Z,\LHZ)$, and $g\colon \Omega \times [0,T] \to \LHZ$. 
    Suppose that $\|F_0(\omega,t)\|_{\calL(Z)}$ and $\|G_0(\omega,t)\|_{\calL(Z,\LHZ)}$ are uniformly bounded in $(\omega,t)$ and the mappings $[0,T]\to L^{q_\alpha}(\Omega;\calL(X)), t \mapsto F_0(\cdot,t)$ as well as $[0,T]\to L^{q_\alpha}(\Omega;\calL(X,\LHX)), t \mapsto G_0(\cdot,t)$ are $\alpha$-Hölder continuous. Further, let  
    \begin{align*}
        f&\in C^{\alpha}([0,T]; L^{r_\alpha}(\Omega; X)) \cap C([0,T]; L^{p_\alpha}(\Omega; Y)),\\
        g&\in C^{\alpha}([0,T]; L^{r_\alpha}(\Omega; \LHX)) \cap C([0,T]; L^{p_\alpha}(\Omega; \LHY)),
    \end{align*}
    and $\init \in L_{\F_0}^{p_\alpha}(\Omega;Y)$. Denote by $\sol{}$ the solution of \eqref{eq:SEEconv} and by $\app{}=(\app{j})_{j=0,\ldots,M}$ the rational Milstein scheme. 
    
    Then the rational Milstein scheme converges at rate $\alpha$ up to a logarithmic correction factor as $h \to 0$ and for $M\ge 2$ there are $C_1,C_2,C_3\ge 0$ such that
    \begin{equation*}
        \norm[\bigg]{ \max_{0 \le j \le M} \norm{ \sol{\t{j}}-\app{j} }_X }_{L^p(\Omega)} \le C_1 h+ \p[\big]{ C_2+C_3\max\big\{\sqrt{\log(T/h)},\sqrt{p}\big\} } h^\alpha.
    \end{equation*}
    For the exponential Milstein scheme, the statement holds with $C_3=0$, i.e.\ the convergence rate $\alpha$ is attained without logarithmic correction factor.
\end{corollary}
\begin{proof}
    First consider $\alpha \in (\frac{1}{2},1]$. We verify that Assumptions \ref{ass:FG_rate_I} and \ref{ass:FG_rate_II} hold for $F$ and $G'G$, the claims for $G$ follow analogously. Denote by $L_{F_0,Z}$ the uniform boundedness constant of $\|F_0(\omega,t)\|_{\calL(Z)}$ for $Z\in\{X,Y\}$ and define $L_{G_0,Z}$ likewise.
    Since $F$ is affine linear, Lipschitz continuity on $X$ and linear growth of $F-f$ on $Y$ as in Assumptions \ref{ass:FG_rate_I}\ref{assCond:FG_rate_Lipschitz} and \ref{assCond:FG_rate_lineargrowthY} follow immediately with constants $L_{F_0,X}$ and $L_{F_0,Y}$.
    
    To allow $F_0,G_0$ to depend on time, we employ the relaxed temporal Hölder continuity assumption~\eqref{eq:weakenedTemporalHoelderLinGrowth} discussed in Remark \ref{rem:weakenedHoelderAss}. Hölder continuity of $F_0$ with constant $C_{\alpha,F_0}$ implies \eqref{eq:weakenedTemporalHoelderLinGrowth} via
    \begin{align*}
        \norm[\bigg]{ \sup_{x\in X} &\frac{\norm{F(\cdot,t,x)-F(\cdot,s,x)}_X }{1+\norm{x}_X} }_{q_\alpha} \le \norm[\bigg]{ \sup_{x\in X} \norm{F_0(\cdot,t)-F_0(\cdot,s)}_{\calL(X)}\frac{\norm{x}_X }{1+\norm{x}_X} }_{q_\alpha}\\
        &\phantom{\le }+ \norm[\bigg]{ \sup_{x\in X} \frac{\norm{f(\cdot,t)-f(\cdot,s)}_{X} }{1+\norm{x}_X} }_{q_\alpha}\le (t-s)^\alpha \p[\big]{C_{\alpha,F_0} + \norm{f}_{C^\alpha([0,T];L^{q_\alpha}(\Omega;X))}}<\infty.
    \end{align*}
   Gâteaux differentiability is satisfied due to $F'(\omega,t,x) = F_0(\omega,t)\in \calL(X)$.
    Since the Gâteaux derivative does not depend on $x$, spatial Hölder continuity of $F'$ as in Assumption \ref{ass:FG_rate_II}\ref{assCond:FG_rate_GateauxHoelder} is immediate.
    By definition of $G$, we can compute
    \begin{equation*}
        (G'G)(\omega,t,x) = G_0(\omega,t)\circ G(\omega,t,x) = G_0(\omega,t)\circ G_0(\omega,t)x + G_0(\omega,t)\circ g(\omega,t)
    \end{equation*}
    and define $\tilde{g}(\omega,t)\ce (G'G)(\omega,t,0) = G_0(\omega,t)\circ g(\omega,t)$. 
    Lipschitz continuity and linear growth of $G'G$ as required in Assumptions \ref{ass:FG_rate_II}\ref{assCond:FG_rate_GpGLipschitz} and \ref{assCond:FG_rate_GpGLinGrowth} with constants $L_{G_0,X}^2$ and $L_{G_0,Y}^2$, respectively, are immediate. 
    By uniform boundedness of $G_0\colon \Omega \times [0,T]\to\calL(Y,\LHY)$, 
    Assumption \ref{ass:FG_rate_II}\ref{assCond:FG_rate_GpGYinvariance} on $\tilde{g}$ is satisfied.
    The claimed error estimate and convergence rate now follow directly from Theorem \ref{thm:main} for the rational Milstein scheme and Theorem \ref{thm:expMilstein} for the exponential Milstein scheme, respectively.
    
    Now consider $\alpha\in (0,\frac{1}{2}]$. Only Assumptions \ref{ass:FG_rate_I}\ref{assCond:FG_rate_HolderContTime} and \ref{ass:FG_rate_II}\ref{assCond:FG_rate_GateauxHoelder} depend on $\alpha$ and the latter can be replaced by the $p_\alpha$-dependent regularity assumptions for this range of $\alpha$, as discussed in Remark~\ref{rem:alpha012}. The temporal Hölder continuity assumption \ref{ass:FG_rate_I}\ref{assCond:FG_rate_HolderContTime} was used to estimate $T_{F,2}$ and $T_{G,2}$ in \eqref{eq:TF2} and \eqref{eq:TG2}, respectively. These terms can be estimated directly using the $\alpha$-Hölder continuity of $[0,T]\ni t \mapsto F_0(\cdot,t) \in L^\infty(\Omega;\calL(X))$ and the a priori estimate \eqref{eq:defCWPqZ} with $Z=X$ and $q=p_\alpha=p$ by
    \begin{align*}
    &T_{F,2}
    \le \sum_{i=0}^{m-1} \int_{t_i}^{t_{i+1}} \big\|[F_0(\cdot,s)-F_0(\cdot,t_i)]U_s\big\|_{p,X}\ds + \sum_{i=0}^{m-1} \int_{t_i}^{t_{i+1}} \|f(\cdot,s)-f(\cdot,t_i)\|_{p,X}\ds\\
    &\le \sum_{i=0}^{m-1} \int_{t_i}^{t_{i+1}} \big\|F_0(\cdot,s)-F_0(\cdot,t_i)\big\|_{\infty,\calL(X)}\ds \cdot \bigg(\sup_{t\in[0,T]} \|U_t\|_{p,X}\bigg) + \|f\|_{C^\alpha([0,T];L^p(\Omega;X))}T h^\alpha \lesssim Th^\alpha
\end{align*}
    and an analogous argument for $T_{G,2}$. Together with the previous observations for $\alpha >\frac{1}{2}$ this finishes the proof.
\end{proof}

\subsection{Error estimates on the full time interval}
\label{subsec:errorEstInterval}
The error estimates of Theorems \ref{thm:main} and \ref{thm:expMilstein} hold for the difference between the solution $\sol{t}$ at the grid points $t=t_j$ and the scheme $\app{j}$. The convergence rate obtained there remains valid for arbitrary times $t\in [0,T]$ if one chooses a suitable extension of the Milstein scheme to general times $t$. For the extension introduced in Definition \ref{def:extension}, such an error estimate is contained in Theorem \ref{thm:extensionInterval}, whose proof relies on the maximal estimates for (iterated) stochastic integral increments of Lemma \ref{lem:smallStochasticInt}.
\begin{definition}
\label{def:extension}
    Let $M\ge 2$. For $t\in[0,T)$ and $0\le \ell\le M-1$ such that $t\in [t_\ell, t_{\ell+1})$, define the extension $\calR_h(t)\ce R_{t-t_\ell}R_h^\ell$ and $\calR_h(T)\ce R_h^M$ of the rational scheme to $[0,T]$, where we set $R_0\ce I$. Then we define the process $\bar{u}=(\bar{u}_t)_{t\in [0,T]}$ by 
    \begin{align*}
        \bar{u}_t &\ce \calR_h(t)\init+\int_0^t \calR_h(t-\lfloor s \rfloor) F(\lfloor s \rfloor, \bar{u}_{\lfloor s \rfloor})\ds\\
        &\phantom{\ce }+\int_0^t \calR_h(t-\lfloor s \rfloor) \br[\Big]{G(\lfloor s \rfloor, \bar{u}_{\lfloor s \rfloor})+\int_{\floors}^s (G'G)(\lfloor s \rfloor, \bar{u}_{\lfloor s \rfloor})\dWHr}\dWHs,\quad t\in[0,T].
    \end{align*}
\end{definition}

\begin{lemma}
\label{lem:smallStochasticInt}
    Let $\calX$ and $Z$ be Hilbert spaces, $T,h>0$, $M\in \N$ with $M\ge 2$ such that $hM=T$, $t_j \ce jh$ for $0 \le j \le M$, and $p\in [2,\infty)$. Further, let $\phi \ce \p{\phi_j}_{j=0}^{M-1}\in L^p(\Omega; \ell_M^\infty(\calL_2(H,\calX)))$ and $\psi \ce \p{\psi_j}_{j=0}^{M-1} \in L^p(\Omega; \ell_M^\infty(\calL_2^{(2)}(H,Z)))$ be finite sequences such that $\phi_j$ and $\psi_j$ are strongly $\F_{t_j}$-measurable. Then, with $K$ as in Proposition \ref{prop:log-maximal-inequality}, the following estimates hold:
    \begin{align}
    \label{eq:lemSmall}
        \norm[\bigg]{\max_{0 \le j \le M-1}\sup_{t \in [t_j,t_{j+1}]}\norm[\Big]{\int_{t_j}^t \phi_j\dWHs}_{\calX}}_{L^p(\Omega)}&\le K \sqrt{\max\{\log(M),p\}} \cdot h^{1/2} \|\phi\|_{L^p(\Omega;\ell_M^\infty(\calL_2(H,\calX)))},\\
        \norm[\bigg]{\max_{0 \le j \le M-1}\sup_{t \in [t_j,t_{j+1}]}\norm[\Big]{\int_{t_j}^t \int_{t_j}^s\psi_j\dWHr\dWHs}_{Z}}_{L^p(\Omega)}&\le K^2 \max\{\log(M),p\} \cdot h \|\psi\|_{L^p(\Omega;\ell_M^\infty(\calL_2^{(2)}(H,Z)))}.\label{eq:lemSmallIterated}
    \end{align}
\end{lemma}
\begin{proof}
Abbreviate $\Phi_p(M)\ce \sqrt{\max\{\log(M),p\}}$. By the logarithmic square function estimate of Proposition \ref{prop:log-maximal-inequality},
    \begin{align*}
        &\norm[\bigg]{\max_{0 \le j \le M-1}\sup_{t \in [t_j,t_{j+1}]}\norm[\Big]{\int_{t_j}^t \phi_j\dWHs}_{\calX}}_p=\norm[\bigg]{\sup_{t \in [0,T],\;0\le j \le M-1}\norm[\Big]{\int_0^t \1_{[t_j,t_{j+1})}(s)\phi_j\dWHs}_{\calX}}_p\\
        &\le K \Phi_p(M) \norm[\Big]{\max_{0\le j \le M-1}\p[\Big]{\int_0^T \1_{[t_j,t_{j+1})}(s)\|\phi_j\|_{\calL_2(H,\calX)}^2\ds}^{1/2}}_p
        = K \Phi_p(M) h^{1/2} \|\phi\|_{L^p(\Omega;\ell_M^\infty(\calL_2(H,\calX)))}.
    \end{align*}
    For the iterated stochastic integral increments, we apply Proposition \ref{prop:log-maximal-inequality} to obtain
    \begin{align*}
        \norm[\bigg]{\max_{0 \le j \le M-1}\sup_{t \in [t_j,t_{j+1}]}&\norm[\Big]{\int_{t_j}^t\int_{t_j}^s \psi_j\dWHr\dWHs}_Z}_p
        \le\norm[\bigg]{\sup_{t \in [0,T],0\le j \le M-1}\norm[\Big]{\int_0^t \1_{[t_j,t_{j+1})}(s)\int_{t_j}^s\psi_j\dWHr\dWHs}_Z}_p\\
        &\le K \Phi_p(M) \norm[\Big]{\max_{0\le j \le M-1}\p[\Big]{\int_0^T \1_{[t_j,t_{j+1})}(s)\norm[\Big]{\int_{t_j}^s\psi_j\dWHr}_{\calL_2(H,Z)}^2\ds}^{1/2}}_p\\
        &\le K \Phi_p(M) h^{1/2} \norm[\bigg]{\max_{0\le \ell \le M-1}\sup_{t\in[t_\ell,t_{\ell+1}]}\norm[\Big]{\int_{t_\ell}^t\psi_\ell\dWHr}_{\calL_2(H,Z)}}_p,
    \end{align*}
    from which \eqref{eq:lemSmallIterated} follows by applying \eqref{eq:lemSmall} with $\calX=\calL_2(H,Z)$ and invoking the isomorphism $\calL_2(H,\calL_2(H,Z))\cong \calL_2^{(2)}(H,Z)$.
\end{proof}
Lemma \ref{lem:smallStochasticInt} also holds for $\calX$ and $Z$ being $2$-smooth Banach spaces if one replaces Hilbert--Schmidt by $\gamma$-radonifying operators and $K$ by some space-dependent constant (see also Remark \ref{rem:2smooth}).

\begin{theorem} \label{thm:extensionInterval}
    Let the assumptions of Theorem \ref{thm:main} hold and let $M\ge 2$. Further, assume that $f\in L_\calP^p(\Omega;C([0,T];Y))$, $g\in L_\calP^p(\Omega;C([0,T];\LHY))$, and $\tilde{g}\in L_\calP^p(\Omega;C([0,T];\bilinHY))$. Then the process $\bar{u}=(\bar{u}_t)_{t\in [0,T]}$ from Definition \ref{def:extension}
    is an extension of the rational Milstein scheme $u=(\app{j})_{j=0,\ldots,M}$ defined in \eqref{eq:defRationalMilstein} in the sense of $\bar{u}_{t_j}=u_j$ for all $0 \le j \le M$. Moreover, it satisfies the error estimate
    \begin{align*}
        \norm[\bigg]{\sup_{t\in[0,T]}\|\sol{t}-\bar{u}_t\|_X}_{L^p(\Omega)} \le \bar{C}_1h+(\bar{C}_2+\bar{C}_3\max\{\sqrt{\log(T/h)},\sqrt{p}\})h^\alpha
    \end{align*}
    for some constants $\bar{C}_1,\bar{C}_2,\bar{C}_3\ge 0$. For the extension of the exponential Milstein scheme, the estimate holds with $\bar{C}_3=0$.
\end{theorem}
\begin{proof}
    Note that due to $\calR_h(t_j)=R_h^j$ for all $0\le j \le M$, $\bar{u}_t$ is indeed an extension of the rational Milstein scheme $u$. First, we show the claim for the extension of the exponential Milstein scheme, where $\calR_h(t)=S(t)$ by the semigroup property of $S$. The error can be split into
    \begin{align}
    \label{eq:errorSplitExtension}
        &\sol{t}-\bar{u}_t = \sum_{i=1}^4\tilde{T}_{F,i}+\int_0^t \sg{t-\floors}[F(\floors,\sol{\floors})-F(\floors,\bar{u}_{\floors})]\ds +\sum_{i=1}^5\tilde{T}_{G,i}\\
        &+\int_0^t \sg{t-\floors}\br[\Big]{G(\floors,\sol{\floors})-G(\floors,\bar{u}_{\floors})+\int_{\floors}^s \p[\big]{(G'G)(\floors,\sol{\floors})-(G'G)(\floors,\bar{u}_{\floors})}\dWHr}\dWHs \nonumber
    \end{align}
    with the continuous-time analogue 
    \begin{align*}
        \tilde{T}_{F,1} \ce \int_0^t \br[\big]{\sg{t-s}-\sg{t-\floors}}F(s,\sol{s})\ds \text{ of } \int_0^{t_j} \br[\big]{\sg{t_j-s}-\sg{t_j-\floors}}F(s,\sol{s})\ds,
    \end{align*} 
    whose norm $T_{F,1}$ was estimated in \eqref{eq:TF1} and likewise for the other terms. Define $\tilde{T}_{F,5}$ and $\tilde{T}_{G,6}$ to be the first and second integral terms in \eqref{eq:errorSplitExtension}, respectively. A close inspection of the proofs of Theorem \ref{thm:main} as well as Lemmas \ref{lem:regularity-UtDiffVst-Z} and \ref{lem:regularity-DiffWithStInt-X} reveals that the estimates for $T_{F,1},\ldots,T_{G,5}$ carry over to the $L^p(\Omega;C([0,T];X))$-norms of $\tilde{T}_{F,1},\ldots,\tilde{T}_{G,5}$, since the proofs do not rely on the terminal time being a grid point. Instead of the Gronwall argument employed for their grid-point analogues, we estimate $\tilde{T}_{F,5}$ and $\tilde{T}_{G,6}$ directly. Because $\bar{u}$ and $u$ agree at the grid points, we can leverage the grid point estimate of Theorem \ref{thm:expMilstein} and Lipschitz continuity of $F$ to deduce
    \begin{align*}
        \norm[\bigg]{\sup_{0 \le t \le T}\|\tilde{T}_{F,5}\|_X}_p &\lesssim \norm[\bigg]{\sup_{0 \le t \le T}\int_0^t\|\sol{\floors}-\bar{u}_{\floors}\|_X\ds}_p \lesssim T\norm[\Big]{\max_{0 \le j \le M}\|\sol{t_j}-\app{j}\|_X}_p \lesssim T(h+h^\alpha).
    \end{align*}
    An analogous estimate is obtained for $\tilde{T}_{G,6}$ after splitting into contributions of $G$ and $G'G$.
    
    Second, we prove the error estimate for the extension of the rational Milstein scheme. Let $t\in [0,T]$ and $0\le \ell\le M-1$ such that $t\in [t_\ell,t_{\ell+1})$ or set $\ell=M$ if $t=T$. Since $\calR_h(t)=\calR_h(t-t_\ell)\calR_h(t_\ell)=R_{t-t_\ell}R_h^\ell$ by definition, $R$ and $S$ are contractive on $X$ and $Y$, and $R$ approximates $S$ to rate $\alpha$ on $Y$ for arbitrary time steps by Assumption \ref{ass:schemeR},
    \begin{align}
    \label{eq:calRrate}
        \|\sg{t}-\calR_h(t)\|_{\calL(Y,X)} 
        &\le\|[\sg{t-t_{\ell}}-R_{t-t_\ell}]\sg{t_\ell}\|_{\calL(Y,X)}+\|R_{t-t_\ell}[\sg{t_{\ell}}-R_h^\ell]\|_{\calL(Y,X)}\nonumber\\
        &\le C_\alpha(t-t_\ell)^\alpha\|S(t_\ell)\|_{\calL(Y)} + \|S(t_\ell)-R_h^\ell\|_{\calL(Y,X)} \le 2C_\alpha h^\alpha.
    \end{align}
    Define the continuous-time version $\bar{v}_t^2$ of $(v_j^2)_j$ from Definition \ref{def:vj12} analogously as the continuous-time error terms and abbreviate $\tau \ce t-t_\ell \in [0,h)$. Then $\sol{t}-\bar{v}_t^2$ admits the same splitting as $\sol{t}-\bar{u}_t$ in \eqref{eq:errorSplitExtension} and can be estimated using Theorem \ref{thm:main} instead of \ref{thm:expMilstein} to estimate $\tilde{T}_{F,5}$ and $\tilde{T}_{G,6}$. 
    The remaining error can be decomposed as
    \begin{align*}
        \bar{v}_t^2-\bar{u}_t &= S(\tau)(v_{t_\ell}^2-\bar{u}_{t_\ell})+\br[\big]{\sg{\tau}-\calR_h(\tau)}\\
        &\phantom{= }~~\p[\Big]{\app{\ell}+\tau F(\t{\ell},\app{\ell})+\int_{t_\ell}^t \p[\Big]{G(\t{\ell},\app{\ell})+\int_{\t{\ell}}^s (G'G)(\t{\ell},\app{\ell})\dWHr}\dWHs}\ec \tilde{T}_{R,1}+\tilde{T}_{R,2}.
    \end{align*}
    Due to $v_{t_\ell}^2-\bar{u}_{t_\ell}=v_\ell^2-\app{\ell}$, the $L^p(\Omega;C([0,T];X))$-norm of the first term is precisely the term already estimated in Lemma \ref{lem:E3} after using the contractivity of the semigroup. This results in
    \begin{equation*}
        \norm[\bigg]{\sup_{t\in [0,T]} \|\sol{t}-\bar{v}_t^2\|_X}_p + \norm[\bigg]{\sup_{t\in [0,T]} \norm[\big]{\tilde{T}_{R,1}}_X}_p \lesssim h+(1+\max\{\sqrt{\log(M)},\sqrt{p}\})h^\alpha.
    \end{equation*}
    Thus, it remains to estimate the norm of $\tilde{T}_{R,2}$. Since $\tau\leq h$ and by \eqref{eq:calRrate}, we have
    \begin{align*}
        \norm[\bigg]{\sup_{t\in [0,T]} \norm[\big]{\tilde{T}_{R,2}}_X}_p&\le C_\alpha h^\alpha \p[\bigg]{\norm[\Big]{\max_{0 \le \ell\le M-1} \|\app{\ell}\|_Y}_p + h \norm[\Big]{\max_{0 \le \ell\le M-1} \|F(\t{\ell},\app{\ell})\|_Y}_p\\
        &\phantom{\le }+ \norm[\Big]{\max_{0 \le \ell\le M-1}\sup_{t\in[t_\ell,t_{\ell+1})} \norm[\Big]{\int_{t_\ell}^t \p[\Big]{G(\t{\ell},\app{\ell})+\int_{\t{\ell}}^s (G'G)(\t{\ell},\app{\ell})\dWHr}\dWHs}_Y}_p}.
    \end{align*}
    By pathwise uniform stability of $u$ as discussed in Remark \ref{rem:stability}\ref{remCond:pathwiseUniformStability} and linear growth of $F$, the first two terms together with the prefactor decay at rate $\alpha$ and $\alpha+1$. After another triangle inequality, Lemma \ref{lem:smallStochasticInt} can be applied with $\phi_\ell\ce G(t_\ell,\app{\ell})$, $\psi_\ell\ce (G'G)(\t{\ell},\app{\ell})$, and $\calX=Z=Y$, since the resulting norms are finite by linear growth of $G$ and $G'G$ on $Y$ and pathwise uniform stability. Altogether, this results in
    \begin{align*}
        &\norm[\bigg]{\sup_{t\in [0,T]} \norm[\big]{\tilde{T}_{R,2}}_X}_p\lesssim h^\alpha \p[\big]{ 1+ h+ \sqrt{\max\{\log(M),p\}}\cdot h^{1/2}+ \max\{\log(M),p\}\cdot h}\lesssim_T h^\alpha,
    \end{align*}
    where we have used uniform boundedness of the expression in the bracket for $h\in (0,T/2]$.
\end{proof}
Note that $\bar{u}$ is a stochastic continuous-time extension of the rational Milstein scheme $u$ but not an interpolation constructed only from its grid values $\app{j}$. While the nonlinearities $F,G,G'G$ are only evaluated at the grid points $\t{j}$ and at $\app{j}$ for all $j$ such that $\t{j}\le t$ to compute $\bar{u}_t$,  the scheme $\calR_h(t)$, the Wiener increments, and the iterated stochastic integral increments are used inside the intervals. By contrast, interpolations of the grid values are easier to compute. However, the error of a piecewise constant extension of the rational Milstein scheme only decays at the lower rate $h^{1/2}\sqrt{\log(T/h)}$ due to the path regularity of $U$, which can be seen by combining the proof of \cite[Theorem~6.13]{KliobaVeraar24Rate} with Theorem \ref{thm:main}. For the scalar SDE $\rmd U_t=\rmd W_t$, the same holds for a piecewise affine interpolation, since its error is the Brownian bridge (see \cite[Chapter~5.6.B]{KaratzasShreve1998}), whose supremum is of order $h^{1/2}\sqrt{\log(M)}$ over $M$ intervals of length $h$. Moreover, for the SDE the iterated stochastic integral over $[t_j,t_{j+1}]$ does not provide additional information, since it equals $\frac{1}{2}((W_{\t{j+1}}-W_{\t{j}})^2-h)$.

\section{Applications to Hyperbolic Equations}
\label{sec:examples}

We consider a linear and a nonlinear stochastic Schrödinger equation, the stochastic Maxwell's equations, and the stochastic transport equation with a Nemytskii-type nonlinearity in Subsections~\ref{subsec:schroedinger-linear}, \ref{subsec:schroedingerConv}, \ref{subsec:maxwell}, and~\ref{subsec:transport}, respectively. Throughout this section, we abbreviate the exponential Euler, Crank--Nicolson, and implicit Euler methods (cf.\ Subsection \ref{subsec:approxSemigroup} for their definition) by EXE, CN, and IE. The corresponding Milstein schemes are referred to as EX-Milstein (or exponential Milstein), CN-Milstein, and IE-Milstein.

In the rest of the paper, we consider SPDEs with coloured noise driven by a $Q$-Wiener process $W_Q$ with trace class operator $Q$ (cf.\ Subsection \ref{subsec:stochasticIntegration} for the definition) rather than an $H$-cylindrical Brownian motion $W$. This allows for an easier comparison with results from the literature \cite{AC18, KliobaVeraar24Rate, CCHS20}, where $Q$-Wiener processes were used. Recall from Subsection \ref{subsec:stochasticIntegration} that every $Q$-Wiener process can be written in the form $W_Q(t)=Q^{1/2}W_H(t)$ for some $H$-cylindrical Brownian motion $W_H$ for all $t\in[0,T]$. Hence, the results from the previous sections are applicable to stochastic evolution equations with noise term $\mathbb{G}(t,u)\,\rmd W_Q$ by considering $G(t,u) \ce \mathbb{G}(t,u) Q^{1/2}$.

\subsection{Linear stochastic Schrödinger equation}
\label{subsec:schroedinger-linear}

Denote by $\T^d$ the $d$-dimensional torus for $d\in \N$ and let $L^q\ce L^q(\T^d;\mathbb{C})$ for $q\in[2,\infty]$ as well as $H^\sigma \ce H^\sigma(\T^d;\C)$ for $\sigma \ge 0$ in this and the subsequent subsection. Consider the linear stochastic Schrödinger equation
\begin{equation}
    \label{eq:schroedinger-linear}\tag{SE}
    \begin{cases}
        \rmd U &= -\iu(\Delta U + V\cdot U) \dt - \iu U \dW_Q \quad \text{on }(0,T], \\
        U_0 &= \init
    \end{cases}
\end{equation}
on the Hilbert space $X\ce H^\sigma$, $\sigma\ge 0$, with a potential $V\colon\T^d\to\C$ and a $Q$-Wiener process $W_Q$ with trace class operator $Q$. In the abstract setting, we thus consider $H\ce L^2$ and the Schrödinger operator $A\ce \iu\Delta$ with domain $\dom(A)=\{f\in H^\sigma: \iu \Delta f \in H^\sigma\}=H^{\sigma+2}$ to account for $X=H^\sigma$, the linear operator $F_0=-\iu M_V$ given by multiplication by the potential $V$ and the noise $G_0 = [u \mapsto -\iu M_u Q^{1/2}]$. In the following cases, we verify that $F_0$ and $G_0$ satisfy the conditions of Corollary~\ref{cor:linearCase} for the exponential Milstein scheme with $Y\ce H^{\sigma+2\alpha}$ for some $\alpha \in (0,1]$. These cases are obtained by combining all possible cases in which $X=H^\sigma$ and $Y=H^{\sigma+2\alpha}$ satisfy \cite[Ass.~6.14]{KliobaVeraar24Rate} and rearranging by dimension.

\begin{assumption}
\label{ass:SoeLinExpMilstein}
Consider the exponential Milstein scheme, i.e.\ $R=S$. Let $\sigma \ge 0$, $d\in \N$, $\alpha \in (0,1]$, $\beta>0$, $V \in H^\beta$, $Q^{1/2}\in \calL_2(L^2,H^\beta)$ such that one of the following mutually exclusive cases applies.
    \begin{enumerate}[label=(\roman*)]
        \item\label{case:SoLinBanAlg} $d\in \N$, $\sigma >\frac{d}{2}$, $\alpha \in (0,1]$, $\beta=\sigma+2\alpha$ (Banach algebra case)
        \item\label{case:SoLin1D} $d=1$, $\sigma \in [0,\frac{1}{2})$, $\alpha\in (\frac{1}{4}-\frac{\sigma}{2},1]$, $\beta=\sigma+2\alpha$
        \item\label{case:SoLin23D} $d\in \{2,3\}$, $\sigma \in [0,1]$, $\alpha\in (\frac{d}{4}-\frac{\sigma}{2},1]$, $\beta=\sigma+2\alpha$
        \item\label{case:SoLin45D} $d\in \{4,5\}$, $\sigma \in (\frac{d}{2}-2,1]$, $\alpha\in (\frac{d}{4}-\frac{\sigma}{2},1]$, $\beta=\sigma+2\alpha$.
    \end{enumerate}
\end{assumption}

We could also include $d=1$, $\sigma \in [0,\frac{1}{2})$, $\alpha\in (0,\frac{1}{4}-\frac{\sigma}{2})$, and $\beta>\frac{1}{2}$ as a fifth case, or as a sixth case $d\ge 2$, $\sigma \in [0,1)$, $\alpha\in (0,\frac{1-\sigma}{2}]$, and $\beta>\frac{d}{2}$ with analogous arguments to the ones presented below. However, both cases imply $\alpha\le 1/2$, which is the parameter regime in which Milstein schemes are not advantageous (see Remark \ref{rem:alpha012}). 

\begin{theorem}
    \label{thm:schroedingerLinearExp}
    Let $\sigma \ge 0$, $d\in \N$, $\alpha\in (0,1]$, $V\in H^\beta$, and $Q^{1/2}\in\calL_2(L^2,H^\beta)$ satisfy Assumption~\ref{ass:SoeLinExpMilstein} and let $p \in [2,\infty)$. Suppose that
    $\init\in L_{\F_0}^{p_\alpha}(\Omega;H^{\sigma+2\alpha})$ for $p_\alpha\ce\max\{2\alpha p,p\}$, $U$ is the mild solution to~\eqref{eq:schroedinger-linear},
    and $(\app{j})_{j=0,\ldots,M}$ is the exponential Milstein scheme. Then there is a constant $C \ge 0$ depending on $(V,Q^{1/2},\beta,\init,T,p,\alpha, \sigma, d)$ such that for $M \geq 2$
    \begin{equation*}
        \left\| \max_{0 \le j \le M} \|\sol{\t{j}}-\app{j}\|_{H^\sigma} \right\|_{L^p(\Omega)} \le C h^\alpha.
    \end{equation*}
    In particular, the exponential Milstein scheme converges at rate $1$ as $h \to 0$ if $V\in H^{\sigma+2}$, $Q^{1/2}\in\calL_2(L^2,H^{\sigma+2})$, $\sigma>d/2$, and $\init\in L_{\F_0}^{2p}(\Omega;H^{\sigma+2})$. 
\end{theorem}
\begin{proof}
    Let $X= H^\sigma$, $Y = H^{\sigma+2\alpha}$, $F_0=-\iu M_V$ and $G_0 = [u \mapsto -\iu M_u Q^{1/2}]$. By \cite[Lem.~2.2]{AC18}, $-A=-\iu\Delta$ generates a contractive $C_0$-semigroup on $H^s$ for any $s\ge 0$, in particular on $X$ and $Y$. Assumption \ref{ass:spacesSemigroup} is thus satisfied and due to $R=S$, Assumption \ref{ass:schemeR} is, too. The temporal Hölder continuity as well as the assumptions on $f$ and $g$ in Corollary \ref{cor:linearCase} are trivially satisfied, since $F_0$ and $G_0$ are independent of time, $f\equiv 0$, and $g\equiv 0$. We verify boundedness of $F_0$ in $\calL(X)$, that is,
    \begin{equation}
    \label{eq:F0proofSoe}
        \|F_0\|_{\calL(H^\sigma)}= \|-\iu M_V\|_{\calL(H^\sigma)} 
        =\sup_{\|u\|_{H^\sigma}=1} \|V u\|_{H^\sigma} \le C_V
    \end{equation}
    for some $C_V\ge 0$. In case \ref{case:SoLinBanAlg}, this estimate follows with $C_V=\|V\|_{H^\sigma}$ because $X=H^\sigma$ is a Banach algebra for $\sigma>\frac{d}{2}$. In case  $\sigma=0$ in \ref{case:SoLin1D} or \ref{case:SoLin23D}, this follows with $C_V\simeq \|V\|_{H^\beta}$ from Hölder's inequality $\|Vu\|_{L^2}\le \|V\|_{L^\infty}\|u\|_{L^2}$ and the observation that $H^\beta \hra L^\infty$ since $\beta>\frac{d}{2}$ by choice of the parameter ranges. If $\sigma=1$ in \ref{case:SoLin23D} or \ref{case:SoLin45D}, the choice of parameters ensures that the embeddings $H^1 \hra L^q$ for $q=\frac{2\beta}{\beta-1}$, $H^\beta\hra L^\infty$, and $H^\beta \hra H^{1,2\beta}$ hold. An explicit calculation of the $H^1$-norm and Hölder's inequality with parameters $(2\beta,q)$ and $(\infty,2)$ then results in  (see \cite[p.~2078]{KliobaVeraar24Rate})
    \begin{align*}
        \|Vu\|_{H^1}^2 
        &\lesssim \|\nabla V\|_{L^{2\beta}}^2\|u\|_{L^q}^2+\|V\|_{L^\infty}^2(\|\nabla u\|_{L^2}^2+\|u\|_{L^2}^2)\lesssim \|V\|_{H^\beta}^2\|u\|_{H^1}^2.
    \end{align*}
    Hence, the claim is true with $C_V \simeq \|V\|_{H^\beta}$. In the remaining cases \ref{case:SoLin1D}--\ref{case:SoLin45D} with $\sigma \in (0,1)$ (or subsets thereof), we make use of Lemma 3.6 in \cite{KliobaVeraar24Rate}. It states that for $\sigma \in (0,1)$, $d>2\sigma$, and $V\in H^\beta$ for some $\beta>\frac{d}{2}$, $\|Vu\|_{H^\sigma}\le C_V \|u\|_{H^\sigma}$, which suffices to prove \eqref{eq:F0proofSoe}. Indeed, this is applicable because $d>2\sigma$ and $\beta>\frac{d}{2}$ in the cases considered for $\sigma\in (0,1)$. 

    Next, we verify boundedness of $F_0$ on $\calL(Y)=\calL(H^{\sigma+2\alpha})$. This is significantly easier, since $\sigma+2\alpha>\frac{d}{2}$ and thus $H^{\sigma+2\alpha}$ is a Banach algebra in all cases. Hence, 
    \begin{equation*}
        \|F_0\|_{\calL(H^{\sigma+2\alpha})} = \sup_u \|Vu\|_{H^{\sigma+2\alpha}} \le \sup_u \|V\|_{H^{\sigma+2\alpha}} \|u\|_{H^{\sigma+2\alpha}} = \|V\|_{H^{\sigma+2\alpha}},
    \end{equation*}
    where the supremum is taken over all $u\in H^{\sigma+2\alpha}$ with $\|u\|_{H^{\sigma+2\alpha}}=1$. In summary, we have shown that $\|F_0\|_{\calL(Z)}=\|M_V\|_{\calL(Z)}\lesssim \|V\|_{H^\beta}$ for $Z\in \{X,Y\}$.

    It remains to show boundedness of $G_0$ on $\calL(Z,\LHZ)$ for $Z\in \{X,Y\}$. We reduce this to the boundedness of $F_0$ and regularity of the noise via
    \begin{align*}
        \|&G_0\|_{\calL(Z,\LHZ)} = \sup_u \|M_uQ^{1/2}\|_\LHZ \le  \sup_u \|M_u\|_{\calL(H^\beta,Z)}\|Q^{1/2}\|_{\calL_2(L^2,H^\beta)}\\
        &\le  \sup_u \sup_v\|vu\|_Z\|Q^{1/2}\|_{\calL_2(L^2,H^\beta)} 
        \le \sup_u \sup_v\|M_v\|_{\calL(Z)}\|u\|_Z\|Q^{1/2}\|_{\calL_2(L^2,H^\beta)}\\
        &\lesssim  \sup_u \sup_v\|v\|_{H^\beta}\|u\|_Z\|Q^{1/2}\|_{\calL_2(L^2,H^\beta)}  = \|Q^{1/2}\|_{\calL_2(L^2,H^\beta)},
    \end{align*}
    where the suprema are taken over all $u\in Z$ and $v\in H^\beta$ of unit norm in the respective spaces and the estimates for $F_0$ were used in the last inequality. The statement of the theorem is obtained from an application of Corollary \ref{cor:linearCase}.
\end{proof}

For rational Milstein schemes, the additional assumption that $R$ approximates $S$ to rate $\alpha\in (0,1]$ on $Y$ enters, leading to higher regularity assumptions on $Y$. Namely, we can choose $Y=H^{\sigma+\ell \alpha}$ with $\ell=4$ for IE-Milstein and $\ell=3$ for CN-Milstein. To see this, recall the definition of IE and CN from Subsection \ref{subsec:approxSemigroup}. Note that IE approximates $S$ to order $\alpha$ on $\dom(A^{2\alpha})$, where the fractional domain is given by $\dom((\iu\Delta)^{2\alpha}) = H^{\sigma+4\alpha}$ for the Schrödinger equation. Likewise, CN approximates $S$ to order $\alpha$ on $\dom(A^{3\alpha/2})=H^{\sigma+3\alpha}$. This observation allows us to generalise Assumption \ref{ass:SoeLinExpMilstein} to rational Milstein schemes.

\begin{assumption}
\label{ass:SoeLinRatMilstein}
Let $\sigma \ge 0$, $d\in \N$, $\ell \in [2,\infty)$, $\alpha \in (0,1]$, $\beta>0$, $V \in H^\beta$, $Q^{1/2}\in \calL_2(L^2,H^\beta)$ such that one of the following mutually exclusive cases holds.
    \begin{enumerate}[label=(\roman*)]
        \item\label{case:SoLinRatBanAlg} $d\in \N$, $\sigma >\frac{d}{2}$, $\alpha \in (0,1]$, $\beta=\sigma+\ell\alpha$ (Banach algebra case) 
        \item\label{case:SoLinRat1DOpt} $d=1$, $\sigma \in [0,\frac{1}{2})$, $\alpha\in (\frac{1}{\ell}(\frac{1}{2}-\sigma),1]$, $\beta=\sigma+\ell\alpha$
         \item\label{case:SoLinRat1DSubOpt} $d=1$, $\sigma \in [0,\frac{1}{2})$, $\alpha\in (0,\frac{1}{\ell}(\frac{1}{2}-\sigma))$, $\beta>\frac{1}{2}$ (suboptimal)
        \item $\sigma \in [0,1],  d \in [2, 2(\ell+\sigma)), \alpha \in (\frac{1}{\ell}(\frac{d}{2}-\sigma),1], \beta=\sigma+\ell\alpha$
         \item $d\ge 2$, $\sigma \in [0,1)$, $\alpha\in (0,\frac{1}{\ell}(1-\sigma)]$, $\beta>\frac{d}{2}$ (suboptimal)
    \end{enumerate}
\end{assumption}

Here, the cases that only allow for suboptimal rates $\alpha\le \frac{1}{2}$ are included. This can be useful when comparing the rates of convergence in regularity regimes where the exponential Milstein method attains the optimal rate $1$ but for e.g.\ IE-Milstein, the regularity of $V$ and $Q$ limit us to rate $\frac{1}{2}$.

Assumption \ref{ass:SoeLinExpMilstein} for EX-Milstein corresponds to $\ell=2$ in the above. Even then, it can be beneficial to consider $\ell>2$ in order to treat higher dimensions. If the stronger assumptions for $\ell>2$ are satisfied, the $L^2$-error ($\sigma=0$) could be considered in dimension four, which was not feasible in the setting of Assumption~\ref{ass:SoeLinExpMilstein}\ref{case:SoLin45D}.

\begin{theorem}
    \label{thm:schroedingerLinearRat}
    Let $(R_h)_{h>0}$ be the IE method or the CN method and $\ell_0 \ce 4$ or $\ell_0\ce 3$, respectively. Suppose that $\sigma \ge 0$, $d\in \N$, $\alpha\in (0,1]$, $V\in H^\beta$, and $Q^{1/2}\in\calL_2(L^2,H^\beta)$ satisfy Assumption \ref{ass:SoeLinRatMilstein} for some $\ell\ge\ell_0$. Let $p \in [2,\infty)$ and
    $\init\in L_{\F_0}^{p_\alpha}(\Omega;H^{\sigma+\ell\alpha})$ for $p_\alpha\ce\max\{2\alpha p,p\}$. Denote by $U$ the mild solution to  
    \eqref{eq:schroedinger-linear}
    and by $(\app{j})_{j=0,\ldots,M}$ the rational Milstein scheme. Then there exists a constant $C \ge 0$ depending on $(V,Q^{1/2},\beta,\init,T,p,\alpha, \sigma, d,\ell)$ such that for $M \geq 2$
    \begin{equation*}
        \left\| \max_{0 \le j \le M} \|\sol{\t{j}}-\app{j}\|_{H^\sigma} \right\|_{L^p(\Omega)} \le C\sqrt{\log(T/h)}\cdot h^\alpha.
    \end{equation*}
    In particular, the IE-Milstein and CN-Milstein schemes converge at rate $1$ up to a logarithmic correction factor as $h\to 0$ if $V\in H^{\sigma+\ell}$, $Q^{1/2}\in\calL_2(L^2,H^{\sigma+\ell})$, $\sigma>\frac{d}{2}$, and $\init \in L_{\F_0}^{2p}(\Omega;H^{\sigma+\ell})$ with $\ell=4$ and $\ell=3$, respectively. 
\end{theorem}
\begin{proof}
    By the discussion above, IE and CN satisfy Assumption \ref{ass:schemeR}. It remains to verify boundedness of $F_0$ and $G_0$ on $Z=Y$, but this is immediate from the proof of Theorem \ref{thm:schroedingerLinearExp} by the Banach algebra property of $Y=H^{\sigma+\ell\alpha}\seq H^{\sigma+2\alpha}$.
\end{proof}

We point out that, up to the logarithmic factor, rational Milstein schemes attain the optimal rate of convergence $1$ for a smoother potential and noise compared to the exponential Milstein scheme. This improves the rates of up to $1/2$ obtained for the exponential Euler method in \cite[Thm.~5.5]{AC18} as well as \cite[Thm.~6.16]{KliobaVeraar24Rate} and for IE and CN (without Milstein term) in \cite[Thm.~6.17]{KliobaVeraar24Rate}.

\subsection{Stochastic Schrödinger equation with a nonlocal nonlinearity}
\label{subsec:schroedingerConv}

As a second example, we again consider the stochastic Schrödinger equation, now with nonlinear $F$ and $G$, which are nonlocal due to convolution. More precisely, consider
\begin{equation}
    \label{eq:nl-schroedinger}\tag{NLSE}
    \begin{cases}
        \rmd U &= -\iu\p[\big]{\Delta U + \eta\ast [\phi(U)] } \dt - \iu \kappa \ast [\psi(U)] \dW_Q\quad \text{on }(0,T], \\
        U_0 &= \init
    \end{cases}
\end{equation}
on $H^\sigma=H^\sigma(\T^d;\C)$ for some $\sigma\ge 0,~d\in \N$, nonlinearities $\phi,\psi\colon\C\to\C$, convolution kernels $\eta,\kappa\colon\T^d\to\C$, and a $Q$-Wiener process $W_Q$.

In the abstract setting, this amounts to $X\ce H^\sigma$, $Y\ce H^{\sigma+\ell\alpha}$, where $\alpha\in (0,1]$ and $\ell\in [2,\infty)$, and $H\ce L^2$ with non-Nemytskii-type nonlinearities
\begin{align*}
    F\colon H^\sigma &\to H^\sigma,\quad
    u \mapsto -\iu \eta \ast [\phi(u)]=\br[\Big]{ x \mapsto - \iu \int_{\T^d} \eta(y)\phi(u(x-y)) \,\rmd y }
\end{align*}
and $G\colon H^\sigma \to \calL_2(L^2,H^\sigma)$, $G(u)\ce -\iu M_{\kappa \ast [\psi(u)]}Q^{1/2}$. Under the following assumption, we obtain pathwise uniform convergence at rate $\alpha$. 

\begin{assumption}
\label{ass:SoeNLconv}
    Let $d\in \N$, $\sigma \ge 0$, $\alpha \in (0,1]$, and $\ell\in [2,\infty)$ such that Assumption \ref{ass:SoeLinRatMilstein} is satisfied for some $\beta>0$, $V\equiv 0$, and $Q^{1/2}\in \calL_2(L^2,H^\beta)$. Suppose that $\eta,\kappa \in C_c^\infty\ce C_c^\infty(\T^d;\C)$. Let $\phi,\psi\colon\C\to\C$ be Lipschitz continuous and once real-differentiable (identifying $\C\cong\R^2$ in the usual way), and $\psi\colon\C\to\C$ be bounded. Further, assume that $\psi'\colon\R^2\to\R^{2\times2}$ is Lipschitz continuous and, if $\alpha>\frac{1}{2}$, that $\phi'\colon\R^2\to\R^{2\times2}$ is $(2\alpha-1)$-Hölder continuous, where $\R^2$ is equipped with the $2$-norm $|\cdot|_2$ and $\R^{2\times2}$ by the induced matrix norm $|\cdot|_{2\times 2}$.
\end{assumption}

\begin{theorem}
\label{thm:SoeNLconv}
    Let $(R_h)_{h>0}$ be the EXE, CN, or IE method and $\ell_0 \ce 2$, $\ell_0\ce 3$, or $\ell_0 \ce 4$, respectively. Suppose that $\sigma \ge 0$, $d\in \N$, $\alpha\in (0,1]$, $\phi,\psi\colon\C\to\C$, $\eta,\kappa\colon \T^d\to\C$, $\beta>0$, and $Q^{1/2}\in\calL_2(L^2,H^\beta)$ satisfy Assumption \ref{ass:SoeNLconv} for some $\ell\ge\ell_0$. Let $p \in [2,\infty)$ and
    $\init\in L_{\F_0}^{p_\alpha}(\Omega;H^{\sigma+\ell\alpha})$ for $p_\alpha\ce\max\{2\alpha p,p\}$. Denote by $U$ the mild solution to \eqref{eq:nl-schroedinger}
    and by $\app{}=(\app{j})_{j=0,\ldots,M}$ the EX-, CN-, or IE-Milstein scheme, respectively. Then there exists a constant $C \ge 0$ depending on $(\eta,\kappa,\phi,\psi,Q^{1/2},\init,T,p,\alpha, \beta, \sigma, d,\ell)$ such that for $M \geq 2$
    \begin{equation*}
        \left\| \max_{0 \le j \le M} \|\sol{\t{j}}-\app{j}\|_{H^\sigma} \right\|_{L^p(\Omega)} \le C \sqrt{\log(T/h)}\cdot h^\alpha.
    \end{equation*}
    If $R$ is the EXE method, the estimate holds without the logarithmic factor.
    In particular, the EX-, CN-, and IE-Milstein schemes converge at rate $1$ in $H^\sigma$ for $\sigma>\frac{d}{2}$ as $h\to 0$ up to a logarithmic factor for CN- and IE-Milstein if $Q^{1/2}\in\calL_2(L^2,H^{\sigma+\ell})$ and $\init \in L_{\F_0}^{2p}(\Omega;H^{\sigma+\ell})$ with $\ell=2$, $\ell=3$, and $\ell=4$, respectively.
\end{theorem}
\begin{proof}
    We have already verified Assumptions \ref{ass:spacesSemigroup} and \ref{ass:schemeR} in Subsection \ref{subsec:schroedinger-linear}. It remains to check Assumptions \ref{ass:FG_rate_I} and \ref{ass:FG_rate_II}. Denote the Lipschitz constants of $\phi, \psi, \psi'$ by $C_\phi, C_\psi, C_{\psi'}$ and the Hölder constant of $\phi'$ by $C_{\phi',\alpha}$. Due to $\eta\in C_c^\infty\seq \mathcal{S}\seq B_{1,2}^\sigma$, the convolution estimate in Besov spaces from~\cite[Thm. 2.1]{KuehnSchilling2022}, and the Lipschitz continuity of $\phi$, $F\colon H^\sigma\to H^\sigma$ is Lipschitz continuous. Indeed, 
    \begin{align*}
        \norm{F(u)-F(v)}_{H^\sigma} &\leq \norm{\eta}_{B_{1,2}^\sigma} \norm{\phi(u)-\phi(v)}_{L^2} \leq C_\phi \norm{\eta}_{B_{1,2}^\sigma} \norm{u-v}_{L^2}
    \end{align*}
    for all $u,v \in H^\sigma$. For the noise term, we can reduce to the same estimate by
    \begin{align*}
        \|G(u)-G(v)\|_{\calL_2(L^2,H^\sigma)} &\le \|M_{\kappa\ast[\psi(u)-\psi(v)]}\|_{\calL(H^\beta,H^\sigma)} \|Q^{1/2}\|_{\calL_2(L^2,H^\beta)}\\
        &\lesssim \|\kappa\ast[\psi(u)-\psi(v)]\|_{H^\sigma} \|Q^{1/2}\|_{\calL_2(L^2,H^\beta)},
    \end{align*}
    where Assumption \ref{ass:SoeLinRatMilstein} ensures the estimate of the multiplication operator by the discussion in Subsection \ref{subsec:schroedinger-linear}. By the same reasoning, we deduce linear growth of $F-f$ on $H^{\sigma+\ell\alpha}$ from
    \begin{align*}
        \|F(u)-F(0)\|_{H^{\sigma+\ell\alpha}}&=\|\eta\ast[\phi(u)-\phi(0)]\|_{H^{\sigma+\ell\alpha}} \le \|\eta\|_{B^{\sigma+\ell\alpha}_{1,2}}\|\phi(u)-\phi(0)\|_{L^2}\\
        &\le C_\phi\|\eta\|_{B^{\sigma+\ell\alpha}_{1,2}}\|u\|_{L^2} \le C_\phi\|\eta\|_{B^{\sigma+\ell\alpha}_{1,2}}\|u\|_{H^{\sigma+\ell\alpha}}
    \end{align*}
    as well as linear growth of $G-g$, noting that $H^{\sigma+\ell\alpha}$ is a Banach algebra by Assumption \ref{ass:SoeLinRatMilstein}.
    
    We claim that for $\alpha >\frac{1}{2}$, $F\colon H^\sigma \to H^\sigma$ is Gâteaux differentiable with $F'(u)[v]=-\iu \eta \ast [\phi'(u)v]$. Indeed, the previous estimates imply 
    \begin{equation*}
        \|F(u+\varepsilon v)-F(u)-\varepsilon F'(u)[v]\|_{H^\sigma}\le \|\eta\|_{B^{\sigma}_{1,2}}\|\phi(u+\varepsilon v)-\phi(u)-\varepsilon \phi'(u)v\|_{L^2}
    \end{equation*}
    and by Proposition \ref{prop:NemytskiiGateauxAffine}, the map $L^2 \to L^2,~u \mapsto \phi(u)$ is Gâteaux differentiable with Gâteaux derivative at $u$ in direction $v$ given by $\phi'(u)v$. The term $\phi'(u)v$ is to be understood in the sense of 
    \begin{equation*}
        [\phi'(u)v](x) = \iota^{-1}(\phi'(\iota(u(x)))\cdot \iota(v(x))), \quad x\in \T^d,
    \end{equation*}
    where $\cdot$ denotes matrix-vector multiplication and $\iota\colon \C \to \R^2,~a+b\iu\mapsto (a,b)^\top$ is the canonical isomorphism. The analogous statement for $G$ follows for $\alpha>0$ with $G'(u)[v]=-\iu M_{\kappa \ast [\psi'(u)v]}Q^{1/2}$.
    
    Next, we verify Hölder continuity of $F'$ for $\alpha>\frac{1}{2}$. Again, we reduce to estimates on $L^2$ via
    \begin{align*}
        \|&F'(u)-F'(\tilde{u})\|_{\calL(H^{\sigma+\ell\alpha},H^\sigma)}  \le \sup_{\|v\|_{H^{\sigma+\ell\alpha}}=1} \|\eta\|_{B^{\sigma}_{1,2}} \|[\phi'(u)-\phi'(\tilde{u})]v\|_{L^2}.
    \end{align*}
    Now, Hölder's inequality and the $(2\alpha-1)$-Hölder continuity of $\phi'$ yield
    \begin{align*}
        \sup_v \|[\phi'(u)-\phi'(\tilde{u})]v\|_{L^2} &\le \p[\Big]{\int_{\T^d} \big|\phi'(\iota(u(x)))-\phi'(\iota(\tilde{u}(x)))\big|_{2\times 2}^2 |\iota(v(x))|_2^2\dx}^{1/2}\\
        & \le \sup_v C_{\phi',\alpha} \p[\Big]{\int_{\T^d} |\iota(u(x))-\iota(\tilde{u}(x))|_2^{2(2\alpha-1)} \dx}^{1/2} \|v\|_{L^\infty} \\
        &\le \sup_v C_{\phi',\alpha} \|u-\tilde{u}\|_{L^{2(2\alpha-1)}}^{2\alpha-1} \|v\|_{H^{\sigma+\ell\alpha}} \le C_{\phi',\alpha} \|u-\tilde{u}\|_{H^{\sigma}}^{2\alpha-1},
    \end{align*}
    where the supremum is taken over all $v\in H^{\sigma+\ell\alpha}$ of unit norm. For $G'$, we can again reduce to this estimate by regularity of $Q^{1/2}$ and estimating the multiplication operator as in the linear case. The $(2\alpha-1)$-Hölder continuity estimate simplifies due to Lipschitz continuity of $\psi'$.

    Finally, we check the assumptions on $G'G\colon H^\sigma\to\calL_2^{(2)}(L^2,H^\sigma)$.
    Collecting the definitions from above, we see that
    \begin{equation*}
        (G'G)(u) = \big[(h_1,h_2) \mapsto - M_{\kappa \ast[\psi'(u)\cdot(\kappa \ast[\psi(u)])\cdot Q^{1/2}h_1]} Q^{\frac{1}{2}}h_2 \big].
    \end{equation*}
    Hence, Lipschitz continuity of $G'G$ can be reduced to estimating the last term in
    \begin{align*}
        \MoveEqLeft \norm{(G'G)(u)-(G'G)(v)}_{\calL_2^{(2)}(L^2,H^\sigma)}^2 \\
        &= \sum_{m\in\N} \sum_{n\in\N} \norm[\big]{\p[\big]{ \kappa\ast\big[\psi'(u)\cdot(\kappa\ast[\psi(u)])\cdot Q^{\frac{1}{2}}h_m\big] - \kappa\ast\big[\psi'(v)\cdot(\kappa\ast[\psi(v)])\cdot Q^{\frac{1}{2}}h_m\big]} Q^{\frac{1}{2}}h_n }_{H^\sigma}^2 \\
        &\leq \sum_{n\in\N} \norm[\big]{Q^{\frac{1}{2}}h_n}_{H^\beta}^2 
        \sum_{m\in\N} \norm[\big]{M_{ \kappa\ast[ (\psi'(u)\cdot(\kappa\ast[\psi(u)])-\psi'(v)\cdot(\kappa\ast[\psi(v)]))\cdot Q^{1/2}h_m ] } }_{\calL(H^\beta,H^\sigma)}^2 \\
        &\leq \norm{Q^{\frac{1}{2}}}_{\calL_2(L^2,H^\beta)}^2 
         \sum_{m\in\N} \norm[\big]{{ \kappa\ast\big[ (\psi'(u)\cdot(\kappa\ast[\psi(u)])-\psi'(v)\cdot(\kappa\ast[\psi(v)]))\cdot Q^{\frac{1}{2}}h_m \big] } }_{H^\sigma}^2 \\
        &\leq \norm{Q^{\frac{1}{2}}}_{\calL_2(L^2,H^\beta)}^2 
        \sum_{m\in\N} \norm{\kappa}_{B_{1,2}^\sigma}^2 \norm{\psi'(u)\cdot(\kappa \ast[\psi(u)])-\psi'(v)\cdot(\kappa \ast[\psi(v)])}_{L^2}^2 \norm{Q^{\frac{1}{2}}h_m}_{L^\infty}^2 \\
        &\leq \norm{Q^{\frac{1}{2}}}_{\calL_2(L^2,H^\beta)}^4 \norm{\kappa}_{B_{1,2}^\sigma}^2 \norm{\psi'(u)\cdot(\kappa \ast[\psi(u)])-\psi'(v)\cdot(\kappa \ast[\psi(v)])}_{L^2}^2,
    \end{align*}
    where we used that $\norm{M_f}_{\calL(H^\beta,H^\sigma)}\leq\norm{f}_{H^\sigma}$ (cf.~Subsection~\ref{subsec:schroedinger-linear}) and $H^\beta\hookrightarrow L^\infty$. 
    
    Hölder's inequality, Young's convolution inequality, Lipschitz continuity of $\psi'$ and $\psi$ as well as boundedness of $\psi$ allow us to estimate the last factor by
    \begin{align*}
        &\|[\psi'(u)-\psi'(v)]\cdot(\kappa \ast[\psi(u)])\|_{L^2} + \|\psi'(v)\cdot(\kappa \ast[\psi(u)-\psi(v)])\|_{L^2}\\
        &\le \|\psi'(u)-\psi'(v)\|_{L^2(\T^d;\R^{2\times 2})}\|\kappa \ast[\psi(u)]\|_{L^\infty}+ \|\psi'(v)\|_{L^\infty(\T^d;\R^{2\times 2})}\|\kappa \ast[\psi(u)-\psi(v)]\|_{L^2}\\
        & \le \p[\big]{C_{\psi'}\|\psi\|_{L^\infty(\C;\C)} + C_\psi^2} \|\kappa\|_{L^1} \|u-v\|_{L^2},
    \end{align*}
    which implies Lipschitz continuity due to $H^\sigma \hra L^2$. Replacing the norm of $\kappa$ by $\|\kappa\|_{B^{\sigma+\ell\alpha}_{1,2}}$ and setting $v=0$, the same argument gives linear growth of $G'G-\tilde{g}$ on $H^{\sigma+\ell\alpha}$. An application of Theorem \ref{thm:main}, Theorem \ref{thm:expMilstein} for exponential Milstein, and Remark \ref{rem:alpha012} for $\alpha \le \frac{1}{2}$ finishes the proof.
\end{proof}

\subsection{Stochastic Maxwell's equations}
\label{subsec:maxwell}

The stochastic Maxwell's equations
\begin{align}\tag{ME}
\label{eq:Maxwell}
    \Bigg\{\begin{split} \rmd U + AU \dt &= F(t,U) \,\rmd t+ G(U)\dW_Q\quad \text{on }(0,T],\\
    \sol{0}&=\init=(\bfE_0^\top,\bfH_0^\top)^\top
    \end{split}
\end{align}
with boundary conditions of a perfect conductor as in \cite{CCHS20} describe how the electric and magnetic field $\bfE$ and $\bfH$ behave. This example is treated in \cite{CCHS20} for the exponential Euler method and in \cite[Sec.~6.6]{KliobaVeraar24Rate} additionally for rational schemes, all without Milstein terms. For the sake of readability, we repeat some arguments in our analysis of Milstein schemes for \eqref{eq:Maxwell}.

For a bounded, simply connected domain $\calO \seq \R^3$ with smooth boundary, let $X\ce L^2(\calO)^6=L^2(\calO)^3\times L^2(\calO)^3$ be equipped with the weighted scalar product
\begin{equation}
\label{eq:defScalarProdMaxwell}
    \left\langle \begin{pmatrix} \bfE_1\\\bfH_1 \end{pmatrix},\begin{pmatrix} \bfE_2\\\bfH_2 \end{pmatrix}\right\rangle_X
    \ce \int_\calO \big(\mu(x) \langle \bfH_1(x),\bfH_2(x)\rangle + \varepsilon(x) \langle \bfE_1(x),\bfE_2(x)\rangle\big)\,\rmd x,
\end{equation}
where $\langle\cdot,\cdot\rangle$ denotes the standard Euclidean scalar product in $\R^3$ and the permittivity and permeability $\varepsilon, \mu \in L^\infty(\calO)$ are assumed to be uniformly positive, i.e.\ $\varepsilon,\mu \ge \kappa >0$ for some constant $\kappa\in(0,\infty)$. The Maxwell operator  $A\colon  \dom(A) \to X$ is given by
\begin{equation*}
    A\begin{pmatrix} \bfE\\ \bfH \end{pmatrix} \ce \begin{pmatrix} 0 & -\varepsilon^{-1}\nabla\times\\\mu^{-1}\nabla \times&0 \end{pmatrix}\begin{pmatrix} \bfE\\ \bfH \end{pmatrix} = \begin{pmatrix} -\varepsilon^{-1}\nabla\times \bfH \\\mu^{-1}\nabla \times \bfE \end{pmatrix}
\end{equation*}
on $\dom(A) \ce H_0(\curl,\calO) \times H(\curl,\calO)$ with $H(\curl,\calO) \ce \{\bfE \in (L^2(\calO))^3\,:\,\nabla \times \bfE \in L^2(\calO)^3\}$ and $H_0(\curl,\calO)$ the subspace of $H(\curl,\calO)$ with vanishing tangential trace $\mathbf{n} \times \mathbf{E}\vert_{\partial\calO}$, where $\mathbf{n}$ is the unit outward normal vector. 
Moreover, $W_Q$ is a $Q$-Wiener process for symmetric, non-negative $Q$ with finite trace such that $Q^{1/2}\in \calL_2(H,X)$ has further regularity specified below. We equip $H\ce L^2(\calO)^6$ with the canonical norm.
In \eqref{eq:Maxwell}, $F\colon [0,T] \times X \to X$ is the linear drift term given by
\begin{equation}
\label{eq:defFMaxwell}
    F(t,V)(x)=\begin{pmatrix} \sigma_1(x,t) \bfE_V(x)\\\sigma_2(x,t) \bfH_V(x)\end{pmatrix},\quad V=\begin{pmatrix}\bfE_V\\\bfH_V\end{pmatrix}\in L^2(\calO)^6,~~ x\in \calO,
\end{equation}
for $\sigma_1,\sigma_2\colon  \calO \times [0,T] \to \R$, whose smoothness is specified later.
In the notation of Corollary \ref{cor:linearCase}, the error estimate for the linear case, this corresponds to $f\equiv 0$ and
\begin{equation*}
    F_0\colon \Omega\times [0,T]\to \calL(L^2(\calO)^6),~ F_0(\omega,t)(V)\ce (\sigma_1(\cdot,t)\bfE_V^\top,\sigma_2(\cdot,t)\bfH_V^\top)^\top.
\end{equation*} 
As noise $G(V)$ for $V=(\bfE_V^\top,\bfH_V^\top)^\top\in L^2(\calO)^6$, we consider the Nemytskii map associated to $\diag((-\varepsilon^{-1}\bfE_V^\top,-\mu^{-1}\bfH_V^\top))Q^{1/2}$. We can recast this in the linear setting by choosing $g\equiv 0$ and $G_0\colon\Omega \times [0,T]\to \calL(X,\calL_2(H,X))$ defined via
\begin{align}
\label{eq:defG0Maxwell}
    (G_0(\omega,t)(V))(h) &\ce  \begin{pmatrix}
        -\varepsilon^{-1}\diag(\bfE_V)&0\\
        0&-\mu^{-1}\diag(\bfH_V) 
    \end{pmatrix}
    (Q^{1/2}h)\in L^2(\calO)^6
\end{align} 
for all $h\in L^2(\calO)^6$. To verify the assumptions of Corollary \ref{cor:linearCase} for Maxwell's equations, regularity of $\sigma_1$ and $\sigma_2$ is required in addition to the assumption on $\mu,\varepsilon$, which we repeat. 
\begin{assumption}
\label{ass:Maxwell}
    Suppose that $\sigma_j(x,\cdot)$ is Lipschitz continuous uniformly in $x\in \calO$ with Lipschitz constant $C_{\sigma_j}$ and let $\partial_{x_i}\sigma_j, \sigma_j \in L^\infty(\calO \times [0,T])$ for $i=1,2,3$ and $j=1,2$. Further, assume that the permittivity and permeability $\varepsilon, \mu \in L^\infty(\calO)$ are uniformly positive, i.e.\ $\varepsilon,\mu \ge \kappa >0$ for some constant $\kappa\in(0,\infty)$.
\end{assumption}

Indeed, uniform boundedness of $F_0$ in $\calL(X)$-norm by $C_{F_0,X}\ce \max\{\|\sigma_1\|_\infty,\|\sigma_2\|_\infty\}$ follows via
\begin{align*}
    \sup_{t\in[0,T]} \|F_0(t)\|_{\calL(X)}^2 
    &= \sup_{t\in[0,T]} \sup_{\|V\|_X=1} \int_\calO \left( \mu(x) \|\sigma_2(x,t)\bfH_V(x)\|^2+\varepsilon(x) \|\sigma_1(x,t)\bfE_V(x)\|^2\right)\dx\\
    &\le \sup_{\|V\|_X=1} \sup_{t\in[0,T]} \max\{\|\sigma_1(\cdot,t)\|_\infty,\|\sigma_2(\cdot,t)\|_\infty\}^2 \|V\|_X^2 = C_{F_0,X}^2,
\end{align*}
where $\|\cdot\|$ denotes the Euclidean norm in $\R^3$. As the space $Y$, we choose $Y\ce \dom(A)$. Since
\begin{equation*}
    \|F_0(t)\|_{\calL(Y)}^2 = \sup_{\|V\|_{\dom(A)}=1} \|F_0(t)V\|_Y^2 = \sup_{\|V\|_{\dom(A)}=1} \p[\big]{\|AF_0(t)V\|_X^2+\|F_0(t)V\|_X^2}
\end{equation*}
and the arguments above allow us to estimate the second term by $C_{F_0,X}^2\|V\|_X^2$, it suffices to estimate the first term. By an explicit calculation of the curl operator,
\begin{align*}
    \|AF_0(t)V\|_X^2
    &\le \kappa^{-2} \int_\calO \mu(x) \|(\nabla \times (\sigma_1(\cdot,t) \bfE_V))(x)\|^2 + \varepsilon(x) \|(\nabla \times(\sigma_2(\cdot,t)\bfH_V))(x)\|^2\dx\\
    &\le 3\kappa^{-2}\p[\Big]{C_{F_0,X}^2 \|AV\|_X^2+2\p[\Big]{\max_{j=1,2}\max_{i=1,2,3} \|\partial_{x_i}\sigma_j\|_\infty^2}\|V\|_X^2}.
\end{align*}
Thus, we conclude uniform boundedness of $\|F_0(t)\|_{\calL(Y)}$ by
\begin{equation*}
    C_{F_0,Y}\ce \max\Big\{\sqrt{3}\kappa^{-1}C_{F_0,X}, \sqrt{6}\kappa^{-1}\max_{j=1,2}\max_{i=1,2,3} \|\partial_{x_i}\sigma_j\|_\infty+C_{F_0,X}\Big\}.
\end{equation*}
Moreover, $t \mapsto F_0(t)$ is Lipschitz continuous, since by Hölder's inequality in $L^2(\calO)$ and uniform Lipschitz continuity of $\sigma_1$, $\sigma_2$, we have
\begin{align*}
    \|F_0(t)-F_0(s)\|_{q_\alpha,\calL(X)}^2 
    &= \sup_{\|V\|_X=1} \int_\calO \mu\|(\sigma_2(\cdot,t)-\sigma_2(\cdot,s))\bfH_V\|^2+\varepsilon\|(\sigma_1(\cdot,t)-\sigma_1(\cdot,s))\bfE_V\|^2\dx\\
    &\le \sup_{\|V\|_X=1} \max_{j=1,2}\|\sigma_j(\cdot,t)-\sigma_j(\cdot,s)\|_{L^\infty(\calO)}^2 \|V\|_X^2
    \le \max_{j=1,2}\, C_{\sigma_j}^2(t-s)^2,
\end{align*}
where we have omitted the integration variable $x$ in the first line.
Uniform boundedness of $G_0$ w.r.t.\ the norm in $\calL(X,\LHX)$ by $C_{G_0,X}\ce \kappa^{-1}C_{H^\beta \hra L^\infty}\|Q^{1/2}\|_{\calL_2(L^2(\calO)^6,H^\beta(\calO)^6)}$ follows from the definitions, uniform positivity of the coefficients $\mu,\varepsilon$, Hölder's inequality, and the embedding $H^\beta(\calO) \hra L^\infty(\calO)$ for any $\beta>\frac{3}{2}$. W.r.t.\ the norm in $\calL(Y,\LHY)$, this is a consequence of $Q^{1/2} \in \calL_2(L^2(\calO)^6,H^{1+\beta}(\calO)^6)$ again for $\beta>\frac{3}{2}$ and the linear growth estimate \cite[Formula~(7)]{CCHS20}
\begin{align*}
    \|G(V)\|_{\calL_2(H,\dom(A))}
    &\le C \|Q^{1/2}\|_{\calL_2(L^2(\calO)^6,H^{1+\beta}(\calO)^6)}(1+\|V\|_{\dom(A)})
\end{align*}
for some $C>0$ as detailed in \cite[p.~2111]{KliobaVeraar24Rate}. Since $G_0$ is independent of $t$, temporal Hölder continuity is trivially satisfied. Reasoning as in \cite[Thm.~6.20]{KliobaVeraar24Rate} based on \cite[p.~410]{MaxwellBanachDA} and \cite[Formula~(3)]{CCHS20}, we see that $X$, $Y$, and $(S(t))_{t\ge 0}$ fulfil Assumption \ref{ass:spacesSemigroup}. We can thus improve the convergence rates $1/2$ obtained for Maxwell's equations in \cite[Thm.~3.3]{CCHS20} for EXE and $1/2$ up to logarithmic correction in~\cite[Thm.~6.21]{KliobaVeraar24Rate} for IE and CN to up to $1$ for Milstein schemes.

\begin{theorem}
    \label{thm:maxwell}
    Let $p \in [2,\infty)$, $X= L^2(\calO)^6$ equipped with the weighted scalar product from
    \eqref{eq:defScalarProdMaxwell}, and $F,G$ as introduced in \eqref{eq:defFMaxwell} and \eqref{eq:defG0Maxwell}, respectively. Let Assumption \ref{ass:Maxwell} hold. Suppose that
    $\init=(\bfE_0^\top,\bfH_0^\top)^\top\in L_{\F_0}^{2p}(\Omega;\dom(A))$ and $Q^{1/2} \in \calL_2(L^2(\calO)^6,H^{1+\beta}(\calO)^6)$ for some $\beta >\frac{3}{2}$. Denote by $U$ the mild solution to \eqref{eq:Maxwell}
    and by $(\app{j})_{j=0,\ldots,M}$ the exponential Milstein scheme. Then there exists a constant $C \ge 0$ depending on $(\sigma_1,\sigma_2,Q^{1/2},\beta,\init,T,p, \varepsilon,\mu, \kappa)$ such that for $M \geq 2$
    \begin{equation*}
        \left\| \max_{0 \le j \le M} \|\sol{\t{j}}-\app{j}\|_{X} \right\|_{L^p(\Omega)} \le C h,
    \end{equation*}
    i.e.\ the EX-Milstein scheme converges at rate $1$ as $h \to 0$. If $(\app{j})_{j=0,\ldots,M}$ is the CN-Milstein scheme, it converges at rate $2/3$ up to a logarithmic factor as $h\to 0$ for all $\init\in L_{\F_0}^{4p/3}(\Omega;\dom(A))$ and 
    \begin{equation*}
        \left\| \max_{0 \le j \le M} \|\sol{\t{j}}-\app{j}\|_{X} \right\|_p \le C \sqrt{\log(T/h)} \cdot h^{2/3}.
    \end{equation*}
\end{theorem}
\begin{proof}
    The claim for the exponential Milstein scheme follows directly from Corollary \ref{cor:linearCase} with $\alpha=1$ and $Y=\dom(A)$, which is applicable by the above considerations. For the CN-Milstein scheme, we additionally recall that the Crank--Nicolson method approximates the Maxwell semigroup to rate $\alpha$ on $\dom(A^{3\alpha/2})$ (cf.\ Subsection \ref{subsec:approxSemigroup}), that is, to rate $2/3$ on $Y=\dom(A)$.
\end{proof}

This theorem illustrates that rational Milstein schemes result in worse rates of convergence compared to the exponential Milstein scheme given the same regularity of the coefficients and noise. The IE-Milstein scheme only converges at rate $1/2$ up to a logarithmic correction factor, just like the IE method without Milstein term \cite[Thm.~6.21]{KliobaVeraar24Rate}, which should be preferred in numerical simulations (see Remark~\ref{rem:alpha012}). Increasing the regularity such that the assumptions of Corollary \ref{cor:linearCase} are satisfied for $Y=\dom(A^2)$ and $Y=\dom(A^{3/2})$, we expect the IE- and CN-Milstein, respectively, to achieve the optimal rate of convergence $1$ up to a logarithmic correction factor. However, to the authors' best knowledge, no explicit characterisations of $\dom(A^\alpha)$ or $\dom_A(\alpha,\infty)$ are available for non-integer $\alpha$ for the Maxwell operator.

\subsection{Nonlinear stochastic transport equation}
\label{subsec:transport}

Finally, we analyse a first order equation for which our framework yields optimal convergence rates of the exponential Milstein scheme, even for nonlinear Nemytskii operators. Consider
\begin{equation}
    \label{eq:transport}\tag{TE}
    \begin{cases}
        \rmd U &= \br[\big]{\nabla U + \phi(U)} \dt + \psi(U)\dW_Q \quad \text{on }(0,T], \\
        U_0 &= \init
    \end{cases}
\end{equation}
in dimension $d=1$ with nonlinearities $\phi\colon\R\to\R$ and $\psi\colon\R\to\R$, and a $Q$-Wiener process $W_Q$.
In this subsection, we abbreviate $L^q\ce L^q(\R;\R)$ for $q\in[2,\infty)$ and $H^\alpha \ce H^\alpha(\R;\R)$ for $\alpha\in [0,1]$. Since the best known Lipschitz estimate for $\sigma \in (0,1)$ even with smooth Lipschitz $\phi\in C_b^2$ is 
\begin{equation*}
    \|\phi(u)-\phi(v)\|_{H^\sigma} \lesssim \|u-v\|_{H^\sigma}+ \p[\big]{1+\|u\|_{H^\sigma}+\|v\|_{H^\sigma}}\|u-v\|_{L^\infty},
\end{equation*}
cf. \cite[Prop.~2.7.2]{Taylor2000}, which is nonlinear in $u$ and $v$, we restrict our considerations to the case $\sigma=0$. 
We set $X=H=L^2$, $Y=H^\alpha$ for some $\alpha\in (0,1]$, and consider the linear operator $Au = -u'$ for $u\in \dom(A) \ce H^1$. 
Note that $\dom(A^\alpha)=H^\alpha$ for $\alpha\in[0,1]$ and $-A$ generates the left shift semigroup~\cite[Prop.~II.2.10.1(ii)]{EngelNagel}, which is contractive on both $L^2$ and $H^\alpha$.
Furthermore, $F$ and $G$ are given as Nemytskii operators $F\colon L^2 \to L^2 $, $u \mapsto \phi \circ u$ and $G\colon L^2 \to \calL_2(L^2,L^2)$, $G(u)=[h \mapsto M_{\psi \circ u} Q^{1/2}h]$.

To show convergence, we need to assume sufficient regularity of $\phi$ and $\psi$ and the covariance operator.
We make this precise in the next theorem. Before, we recall useful estimates of composition and multiplication operators in the one-dimensional case that have been used in the previous subsections.
\begin{lemma}
    \label{lem:taylor-tools}
    Denote by $M_f$ the multiplication operator associated with $f\colon\R\to\R$.
    \begin{enumerate}[label=(\roman*)]
        \item\label{lemItem:composEst} Let $s\in[0,1]$ and $f$ be Lipschitz continuous with $f(0)=0$.
        Then $\norm{f\circ u}_{H^s} \lesssim \norm{u}_{H^s}$ for all $u\in H^s$.

        \item\label{lemItem:multOpEst} Let $s>\frac{1}{2}$, $r\in \{0,s\}$, and $f\in H^r$.
        Then $\norm{M_f}_{\calL(H^s,H^r)} \lesssim \norm{f}_{H^r}$.
    \end{enumerate}
\end{lemma}

\begin{proof}
    Part (i) for $s\in(0,1)$ is proved as in \cite[Thm.~5.5.1/1]{RunstSickel1996}.
    For $s=0$ and $s=1$ this follows by direct calculation, in the latter case using \cite[Prop.~II.6.1]{Taylor2000} in
        \begin{align*}
            \norm{f\circ u}_{H^1}^2 &= \norm{f\circ u}_{L^2}^2 + \norm{(f\circ u)'}_{L^2}^2 \leq C^2 \norm{u}_{L^2}^2 + C^2 \norm{u'}_{L^2}^2 = C^2 \norm{u}_{H^1}^2.
        \end{align*}
        
        Part (ii) for $r=0$ is a direct consequence of Hölder's inequality and the Sobolev embedding $H^s \hra L^\infty$, since $\norm{fg}_{L^2} \le \norm{f}_{L^2} \norm{g}_{L^\infty} \lesssim \norm{f}_{L^2} \norm{g}_{H^s}$.
        For $r=s$, instead of Hölder's inequality we can use the Banach algebra property of $H^s$ to estimate $\|fg\|_{H^s}\le \|f\|_{H^s}\|g\|_{H^s}$. Taking the supremum over $g\in H^s$ of unit norm yields the claim for $r\in \{0,s\}$.
\end{proof}

\begin{theorem}
    Let $p\in[2,\infty)$ and $\alpha\in(\frac{1}{2},1]$.
    Assume that $\init\in L_{\F_0}^{2\alpha p}(\Omega;H^\alpha)$ and $Q^{1/2}\in\calL_2(L^2,H^\alpha)$.
    Suppose that $\phi,\psi\colon\R\to\R$ are differentiable with bounded and $(2\alpha-1)$-Hölder continuous derivatives such that $\psi'\psi$ is Lipschitz continuous and $\phi(0)=0$.
    
    Then the exponential Milstein scheme $(\app{j})_{j=0,\ldots,M}$ converges at rate $\alpha$ as $h\to0$ to the mild solution $U$ of \eqref{eq:transport}.
    In particular, there is a constant $C\geq0$ such that
    \begin{equation*}
        \norm[\bigg]{ \max_{0 \le j \le M} \norm{ \sol{\t{j}}-\app{j} }_{L^2} }_{L^p(\Omega)} \le C \h^\alpha.
    \end{equation*}    
\end{theorem}
\begin{proof}
    Note that by assumption there exists a real number $K\geq0$ such that $\varphi\in\{\phi,\psi\}$ is Lipschitz with constant $K$ and for all $x,y \in \R$,
    \begin{align*}
        \abs{ \varphi'(x) } \leq K,~~
         \abs{ \varphi'(x)-\varphi'(y) } \leq K \abs{ x-y }^{2\alpha-1},~~
        \abs{\psi'(x)\psi(x)-\psi'(y)\psi(y)} \leq K \abs{x-y}.
    \end{align*}

    We now verify Assumptions \ref{ass:spacesSemigroup} to \ref{ass:FG_rate_II} for $X=L^2$ and $Y=H^\alpha$ in order to apply Theorem \ref{thm:expMilstein}. 
    As discussed above, Assumption \ref{ass:spacesSemigroup} is fulfilled and thus also Assumption \ref{ass:schemeR} for the exponential Euler method. 

    Since $\phi$ is Lipschitz continuous with $\phi(0)=0$, $F$ is well-defined as a mapping into $L^2$.
    To see that $G$ is well-defined, we use Lemma~\ref{lem:taylor-tools}\ref{lemItem:multOpEst} to observe that
    \begin{align*}
        \norm{G(u)}_{\calL_2(L^2,L^2)} 
        &\leq \norm{M_{\psi\circ u-\psi(0)}}_{\calL(H^\alpha,L^2)} \norm[\big]{Q^{\frac{1}{2}}}_{\calL_2(L^2,H^\alpha)} + \lvert\psi(0)\rvert\norm[\big]{Q^{\frac{1}{2}}}_{\calL_2(L^2,L^2)} \\
        &\lesssim \norm{\psi\circ u-\psi(0)}_{L^2} \norm[\big]{Q^{\frac{1}{2}}}_{\calL_2(L^2,H^\alpha)} + \lvert\psi(0)\rvert\norm[\big]{Q^{\frac{1}{2}}}_{\calL_2(L^2,L^2)},
    \end{align*}
    which is finite for all $u\in L^2$ because $\psi$ is Lipschitz continuous and $Q^{1/2}\in\calL_2(L^2,H^\alpha)$.

    Similarly, the $Y$-invariance of $F$ follows from Lemma~\ref{lem:taylor-tools}\ref{lemItem:composEst} and for $G$ from Lemma~\ref{lem:taylor-tools}\ref{lemItem:multOpEst}, as
    \begin{align*}
        \norm{G(u)}_{\calL_2(L^2,H^\alpha)}
        &\leq \norm{M_{\psi\circ u-\psi(0)}}_{\calL(H^\alpha,H^\alpha)} \norm[\big]{Q^{\frac{1}{2}}}_{\calL_2(L^2,H^\alpha)} + \lvert\psi(0)\rvert\norm[\big]{Q^{\frac{1}{2}}}_{\calL_2(L^2,H^\alpha)} \\
        &\lesssim \norm{\psi\circ u-\psi(0)}_{H^\alpha} \norm[\big]{Q^{\frac{1}{2}}}_{\calL_2(L^2,H^\alpha)} + \lvert\psi(0)\rvert\norm[\big]{Q^{\frac{1}{2}}}_{\calL_2(L^2,H^\alpha)},
    \end{align*}
    which is finite for all $u\in H^\alpha$ by Lemma~\ref{lem:taylor-tools}\ref{lemItem:composEst} applied to $f(x)=\psi(x)-\psi(0)$.

    The Lipschitz continuity of $\phi$ is easily seen to imply the same of $F$.
    Indeed, for $u,v\in L^2$,
    \begin{align*}
        \norm{F(u)-F(v)}_{L^2}^2 &= \int_\R \abs{\phi(u(x))-\phi(v(x))}^2 \dx \leq K^2 \int_\R \abs{u(x)-v(x)}^2 \dx = K^2 \norm{u-v}_{L^2}^2.
    \end{align*}
    Lemma \ref{lem:taylor-tools}\ref{lemItem:composEst} with $f(x)=\phi(x)-\phi(0)$ implies linear growth of $F-F(0)$ on $Y$.
    Lipschitz continuity of $u \mapsto G(u)=M_{\psi(u)}Q^{\frac{1}{2}}$ follows similarly, using Lemma~\ref{lem:taylor-tools}\ref{lemItem:multOpEst} with $r=0$ and Lipschitz continuity of $\psi$ to deduce
    \begin{align*}
        \norm{&G(u)-G(v)}_{\calL_2(L^2,L^2)} 
        \leq \norm{M_{\psi\circ u-\psi\circ v}}_{\calL(H^\alpha,L^2)} \norm[\big]{Q^{\frac{1}{2}}}_{\calL_2(L^2,H^\alpha)} \lesssim 
        K \norm{u - v}_{L^2} \norm[\big]{Q^{\frac{1}{2}}}_{\calL_2(L^2,H^\alpha)}.
    \end{align*}
    Analogously, linear growth on $Y$ can be shown using Lemma \ref{lem:taylor-tools}\ref{lemItem:multOpEst} with $r=s$.

    Gâteaux differentiability of $F$ follows from Proposition \ref{prop:NemytskiiGateauxAffine} because $\phi$ is continuously differentiable with bounded derivative.
    The Gâteaux derivative given by $F'(u)[v] = \br[\big]{x\mapsto \phi'(u(x))\cdot v(x)}$ for $u,v\in L^2$ is Hölder continuous from $H^\alpha$ to $\calL(H^\alpha,L^2)$, since Hölder's inequality and  $H^\alpha\hra L^{4\alpha}$ for $\alpha>\frac{1}{2}$ imply
    \begin{align*}
        \norm{&F'(u)-F'(\tilde{u})}_{\calL(H^\alpha,L^2)} 
        = \sup_{\|v\|_{H^\alpha}=1} \norm{(\phi'\circ u - \phi'\circ \tilde{u})\cdot v}_{L^2}\\
        &\leq \norm{\phi'\circ u - \phi'\circ \tilde{u}}_{L^{\frac{4\alpha}{2\alpha-1}}} \sup_{\|v\|_{H^\alpha}=1} \norm{v}_{L^{4\alpha}} \lesssim K \norm{u - \tilde{u}}_{L^{4\alpha}}^{2\alpha-1} \lesssim K\norm{u - \tilde{u}}_{H^\alpha}^{2\alpha-1}
    \end{align*}
    for $u,\tilde{u}\in H^\alpha$.
    We claim that $G\colon L^2\to\calL_2(L^2,L^2)$ is Gâteaux differentiable with derivative $G'(u)[v] = M_{\psi'(u)\cdot v}Q^{\frac{1}{2}}$ for $u,v\in L^2$. Indeed, by Lemma \ref{lem:taylor-tools}\ref{lemItem:multOpEst} with $r=0$,
    \begin{align*}
        \MoveEqLeft \lim_{\tau\to0} \frac{1}{|\tau|} \norm{G(u+\tau v)-G(u)-G'(u)[\tau v]}_{\calL_2(L^2,L^2)} \\
        &\lesssim  \norm[\big]{Q^{\frac{1}{2}}}_{\calL_2(L^2,H^\alpha)} \lim_{\tau\to 0} \frac{1}{|\tau|} \norm{\psi\circ(u+\tau v)-\psi\circ u-(\psi'\circ u)\cdot \tau v}_{L^2},
    \end{align*}
    and the limit vanishes by Proposition \ref{prop:NemytskiiGateauxAffine} for $\psi$. As for $F'$, we check that this Gâteaux derivative is Hölder continuous on $H^\alpha$. 

    Finally, we check the assumptions on $G'G$ given by $(G'G)(u) = M_{\psi'(u)\cdot\psi(u)} Q^{\frac{1}{2}} \otimes Q^{\frac{1}{2}}$, denoting by $Q^{\frac{1}{2}} \otimes Q^{\frac{1}{2}}$ the bilinear operator $(h_1,h_2)\mapsto \br[\big]{ x\mapsto \p{Q^{\frac{1}{2}}h_1}(x)\cdot \p{Q^{\frac{1}{2}}h_2}(x) }$.
    Note that, by the Banach algebra property of $H^\alpha$, we have
    \begin{equation*}
        \norm[\big]{Q^{\frac{1}{2}} \otimes Q^{\frac{1}{2}}}_{\calL_2^{(2)}(L^2,H^\alpha)} \leq \norm[\big]{Q^{\frac{1}{2}}}_{\calL_2(L^2,H^\alpha)}^2 < \infty.
    \end{equation*}
    The Lipschitz continuity of $G'G$ is essentially implied by the assumed Lipschitz continuity of $\psi'\psi$ noting that 
    \begin{align*}
        \MoveEqLeft \norm{(G'G)(u)-(G'G)(v)}_{\calL_2^{(2)}(L^2,L^2)} \leq \norm{M_{(\psi'\psi)(u)-(\psi'\psi)(v)}}_{\calL(H^\alpha,L^2)} \norm[\big]{Q^{\frac{1}{2}}}_{\calL_2(L^2,H^\alpha)}^2
    \end{align*}
    for $u,v\in L^2$ and continuing as in the Lipschitz estimate of $G$. Likewise, linear growth of $G'G-\tilde{g}$ follows from Lemma \ref{lem:taylor-tools}\ref{lemItem:composEst} with $f(x) = \psi'(x)\psi(x)-\psi'(0)\psi(0)$.
\end{proof}

\begin{remark}
    The proof also gives convergence for rational Milstein schemes, albeit with a lower rate. Since the composition estimate of Lemma \ref{lem:taylor-tools}\ref{lemItem:composEst} is only available for $s\in [0,1]$, the Assumptions \ref{ass:FG_rate_I} and \ref{ass:FG_rate_II} can only be verified on $Y=H^\alpha$ for $\alpha\le 1$. Recall that IE and CN approximate the shift semigroup to orders $\frac{\alpha}{2}$ and $\frac{2\alpha}{3}$, respectively. This limits the convergence rates of the IE-Milstein scheme and the CN-Milstein scheme for the nonlinear transport equation \eqref{eq:transport} to at most $\frac{1}{2}$ and $\frac{2}{3}$, respectively. This limitation does not apply in the linear case.
\end{remark}

\section{Numerical Simulations of the Stochastic Schrödinger Equation}
\label{sec:simulations}

To complement the theoretical findings presented above, we include simulations of three variants of the stochastic Schrödinger equation. First, the linear case is considered with a potential and multiplicative trace-class noise in Subsection \ref{subsec:simuLin}. Second, a nonlocal nonlinearity as considered in Subsection \ref{subsec:schroedingerConv} is simulated in Subsection \ref{subsec:simuConv}. Lastly, we include simulations with a nonlinearity not covered by our results as an outlook in Subsection \ref{subsec:simuNLout}.

All three are considered on $L^2 \ce L^2(\T;\C)$ with the orthonormal Fourier basis functions $e_\ell(x)=(2\pi)^{-1/2}\e^{\iu\ell x}$ for $x\in \T$, $\ell\in\Z$ and up to the final time $T=\frac{1}{2}$.
We abbreviate $H^s \ce H^s(\T;\C)$ for $s\in[0,\infty)$. For the spatial and the noise discretisation we used a spectral Galerkin approach utilising the first $2^{10}$ Fourier coefficients, i.e.\ we only considered the components with $\ell\in\{-2^9+1,\dots,2^9\}$.
As initial value, we chose $\init=\sum_{\ell\in\Z} \mu_\ell e_\ell$ with $\mu_\ell=(1+\lvert \ell\rvert^{2.51})^{-1}$, $\ell\in\Z$, being the Fourier coefficients of $\init$. 
Using the Fourier characterisation $\|\init\|_{H^s}^2 \simeq \sum_{\ell\in\Z} (1+|\ell|^2)^s|\mu_\ell|^2$ of the $H^s$-norm due to $ \norm{e_\ell}_{H^s}^2 \simeq (1+\abs{\ell}^{2})^s $, we see that $\init\in H^s$ for all $s\in[0,\frac{5.02-1}{2})=[0,2.01)$, in particular, $\init\in H^2$.

Similarly, we define the covariance operator $Q = \sum_{\ell\in\Z} \lambda_\ell \langle\cdot,e_\ell\rangle e_\ell$ with $\lambda_\ell = (1+\abs{\ell}^{5.1})^{-1}$ for $\ell\in\Z$.
The exponent $5.1$ determines the regularity of the covariance operator and thus of the driving noise.
We have $Q^{\frac{1}{2}}\in \calL_2(L^2,H^s)$ for all $s\in[0,\frac{5.1-1}{2})=[0,2.05)$, in particular, $Q^{\frac{1}{2}}\in \calL_2(L^2,H^2)$.
Indeed, reasoning as for $\init$, we can bound
\begin{align*}
    \norm[\big]{Q^{\frac{1}{2}}}_{\calL_2(L^2,H^s)}^2 &= \sum_{\ell\in\Z} \norm[\big]{Q^{\frac{1}{2}}e_\ell}_{H^s}^2
    \leq \sum_{\ell\in\Z} \lambda_\ell \norm{e_\ell}_{H^s}^2
    \lesssim \sum_{\ell\in\Z} \frac{1}{1+\abs{\ell}^{5.1-2s}},
\end{align*}
which is finite in the given range of $s$.

All simulations were performed with linear multiplicative noise given by $G(u)=-\iu M_uQ^{\frac{1}{2}}$.
This implies that $(G'G)(u)$ is symmetric as a bilinear Hilbert--Schmidt operator for all $u\in L^2$ (also referred to as the \emph{commutativity condition}).
Thus, we can 
use  
the identity (cf.\ \cite[Eq. (80)]{JentzenRoeckner15_Milstein})
\begin{align*}
    (G'G)(\app{i}) \Delta_2 W_{i+1}^K
    &= \frac{1}{2} G'(\app{i})\br[\big]{G(\app{i})\Delta W_{i+1}^K}\Delta W_{i+1}^K - \frac{1}{2}h \sum_{\ell=-K/2+1}^{K/2} G'(\app{i})\br[\big]{G(\app{i})e_\ell}e_\ell
\end{align*}
to simulate the iterated integral terms required to implement Milstein schemes.
Here, the superscript $K$ denotes the truncation of the Wiener process to the first $K$ Fourier coefficients.
In general, i.e.\ if the operator does not fulfil the commutativity condition, it is currently not known how to simulate the iterated integrals exactly.
Therefore, an approximation of the iterated integrals becomes necessary.
This means that in practice different schemes may have substantially different computational costs per time step and therefore it is instructive to consider the convergence order not with respect to the step size but with respect to the computational cost of an algorithm.
This is called the \emph{effective order of convergence} and was introduced in \cite{Roessler2010}.
In the case of a Milstein scheme in the non-commutative setting, estimates for the computational cost to simulate the iterated integrals with a desired precision are required.
Here, we refer to \cite{KastnerRoessler23} for a recent survey of suitable approximation algorithms in the finite-dimensional case as well as \cite{Leonhard2018} for the best-known estimates in the infinite-dimensional case.
For a detailed analysis of the effective order of convergence for Milstein-type schemes for SPDEs, we refer the interested reader to \cite{vonHallernRoessler23DerivativeFreeMilstein} and \cite{Leonhard2018enhancing}.

We denote by EXE, CN, and IE the exponential Euler, the Crank--Nicolson, and the implicit Euler schemes, respectively, and by EX-Milstein, CN-Milstein, and IE-Milstein (or simply EXM, CNM, IEM) the corresponding exponential or rational Milstein schemes. Approximations computed using EXM, CNM, and IEM are compared with the EXE and the rational schemes CN and IE simulated in~\cite[Sec.~6.5]{KliobaVeraar24Rate} for five different step sizes $\h\in\{2^{-5},2^{-6},\dots,2^{-9}\}$. 
To estimate the pathwise uniform error for $p=2$, we computed the root-mean-square error over $100$ samples.
As no closed form of the analytical solution is available for the equations considered, the errors are computed w.r.t.\ a reference solution, which was generated using the exponential Euler scheme with the significantly smaller step size $\h=2^{-16}$.

In Subsections \ref{subsec:simuLin} and \ref{subsec:simuConv}, we make use of a centred bump function $\chi\colon\T\to\C$ on the torus defined by
\begin{align}
\label{eq:defBump}
    f(x)&=\begin{cases}
        C\exp\p[\big]{\frac{1}{x^2-c^2}}, & \lvert x\rvert < c, \\
        0 & \text{else},
    \end{cases}
    &
    \chi(x)&=\begin{cases}
        f(x), & x\in[0,\pi), \\
        f(x-2\pi), & x\in[\pi,2\pi)
    \end{cases}
\end{align}
for $c\in(0,\pi)$.
Here, $C=\exp(1/c^2)$ is a normalisation constant such that $\chi(0)=f(0)=1$, and we chose $c=\frac{\pi}{2}$ for the simulations.

Finally, we mention that the simulations were performed using Julia v1.11.7 \cite{Bezanson2017} on a standard laptop. 
The code is publicly available under \cite{KastnerKlioba2026}.

\subsection{A linear example}
\label{subsec:simuLin}

We consider the linear Schrödinger equation \eqref{eq:schroedinger-linear} investigated in Section~\ref{subsec:schroedinger-linear} on $X=L^2$ with the potential $V=\chi$, i.e.\ $d=1$ and $\sigma=0$.
Since $\xi\in H^2$, $V\in H^2$, and $Q^{\frac{1}{2}}\in \calL_2(L^2,H^2)$, we can choose $Y=H^2$ and Assumption \ref{ass:SoeLinExpMilstein}\ref{case:SoLin1D} is satisfied with $\alpha=1$ for EXM. Hence, Theorem \ref{thm:schroedingerLinearExp} shows that we should expect a convergence rate of one.
For the rational Milstein schemes CNM and IEM, the expected convergence rates are $2/3$ and $1/2$, respectively, by Theorem~\ref{thm:schroedingerLinearRat}.
For all schemes without Milstein terms, we expect a rate of $1/2$ due to \cite[Thms.~6.16 and 6.17]{KliobaVeraar24Rate}. 

Figure \ref{fig:all-plots}(a) illustrates the numerical errors obtained for different step sizes. The resulting experimental rates of convergence are contrasted with the expected convergence rates discussed above in Table \ref{tab:rates-two-schroedinger}. The simulations confirm the higher rates of convergence of some Milstein schemes compared to the exponential Euler scheme and the rational schemes, whose rates are essentially limited by $1/2$ also in the simulations.
In contrast, higher numerical rates of convergence close to $1$ and $2/3$ are obtained for EXM and CNM, respectively, while the numerical rate for IEM is indeed close to $1/2$, thus validating the theory. 

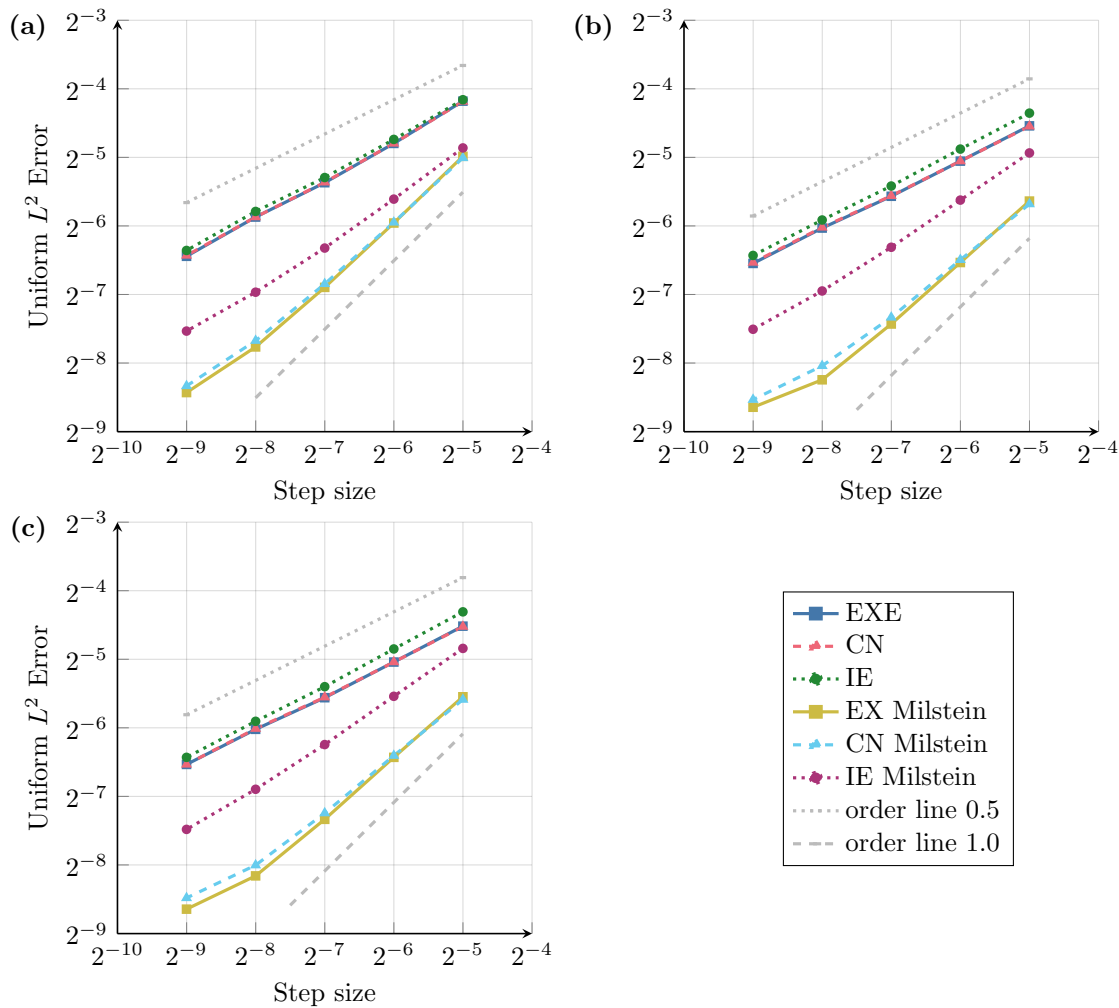
\begin{figure}
    \centering
    \begin{tikzpicture}
    	\begin{groupplot}[xlabel={Step size},ylabel={Uniform $L^2$ Error},
    		xmode=log, ymode=log, log basis x=2, log basis y=2,
    		xmin=2^-10, xmax=2^-4,
    		ymin=2^-10, ymax=2^-2,
    		legend cell align=left,
    		legend columns=1,
    		cycle multiindex* list={
    			[of colormap=PTbright]\nextlist
    			mark=square*\\mark=triangle*\\mark=*\\\nextlist
    			solid\\dashed\\dotted\\
    		},
    		/tikz/mark options={solid,scale=0.6},
    		group style={group size=2 by 2},
    		group/vertical sep=1.2cm, group/horizontal sep=2cm,
    		group/xlabels at=all, group/ylabels at=edge left,
    		width=0.5\textwidth-0.6cm, height=0.5\textwidth-0.6cm, scale only axis=false,
    		]
    		
    		\nextgroupplot[legend to name=legend-linear-schroedinger] 
    		
    		\foreach \method/\name in {EXE/EXE,CNE/CN,LIE/IE,EXM/EX Milstein,CNM/CN Milstein,LIM/IE Milstein}{
    			\addplot+ table [
    			x=Stepsize,
    			y=\method,
    			row sep=\\
    			] { 
    				Stepsize	LIE	CNE	EXE	LIM	CNM	EXM\\
    				0.03125	0.0855	0.0655	0.054	0.0751	0.0487	0.0313\\
    				0.015625	0.0619	0.0435	0.0361	0.0543	0.0297	0.016\\
    				0.0078125	0.0444	0.0295	0.025	0.0387	0.0187	0.0082\\
    				0.00390625	0.0312	0.0195	0.0168	0.0275	0.0113	0.0043\\
    				0.001953125	0.0222	0.0133	0.0118	0.0196	0.007	0.0023\\
    			};
    			\addlegendentryexpanded{\name{}}
    		}
    		
    		\addplot [dotted,PTgrey,domain=2^-9:2^-5,samples=2] {0.0855*8*x^0.5};
    		\addlegendentry{order line 0.5}
    		\addplot [dashed,PTgrey,no markers,domain=2^-8:2^-5,samples=2] {0.0313*22.63*x};
    		\addlegendentry{order line 1.0}

    		\nextgroupplot[] 
    		
    		\foreach \method/\name in {EXE/EXE,CNE/CN,LIE/IE,EXM/EX Milstein,CNM/CN Milstein,LIM/IE Milstein}{
    			\addplot+ table [
    			x=Stepsize,
    			y=\method,
    			row sep=\\
    			] { 
    				Stepsize	LIE	CNE	EXE	LIM	CNM	EXM\\
    				0.03125	0.09	0.061	0.0476	0.0809	0.0445	0.0213\\
    				0.015625	0.0631	0.0401	0.0318	0.0572	0.0275	0.011\\
    				0.0078125	0.0453	0.0282	0.0234	0.0403	0.0179	0.0057\\
    				0.00390625	0.0318	0.0191	0.0164	0.0285	0.011	0.0031\\
    				0.001953125	0.0225	0.0131	0.0116	0.0201	0.0069	0.0018\\
    			};
    		}
    		
    		\addplot [dotted,PTgrey,domain=2^-9:2^-5,samples=2] {0.09*8*x^0.5};
    		\addplot [dashed,PTgrey,no markers,domain=2^-8:2^-5,samples=2] {0.0213*22.63*x};

    		\nextgroupplot[] 
    		
    		\foreach \method/\name in {EXE/EXE,CNE/CN,LIE/IE,EXM/EX Milstein,CNM/CN Milstein,LIM/IE Milstein}{
    			\addplot+ table [
    			x=Stepsize,
    			y=\method,
    			row sep=\\,
    			] { 
    				Stepsize	LIE	CNE	EXE	LIM	CNM	EXM\\
    				0.03125	0.094	0.062	0.0485	0.0853	0.0456	0.0229\\
    				0.015625	0.0656	0.0405	0.0321	0.0599	0.028	0.0118\\
    				0.0078125	0.0467	0.0284	0.0236	0.0419	0.0181	0.0061\\
    				0.00390625	0.0326	0.0192	0.0165	0.0293	0.0111	0.0033\\
    				0.001953125	0.0229	0.0132	0.0116	0.0205	0.0069	0.0019\\
    			};
    		}
    		
    		\addplot [dotted,PTgrey,domain=2^-9:2^-5,samples=2] {0.094*8*x^0.5};
    		\addplot [dashed,PTgrey,no markers,domain=2^-8:2^-5,samples=2] {0.0229*22.63*x};

    		\nextgroupplot[
    		group/empty plot,
    		height=0pt,
    		width=0pt,
    		scale only axis,
    		]

    	\end{groupplot}
    	
    	\node[anchor=south east, font=\bfseries]
    	at ($(group c1r1.north west) + (-8mm,-4mm)$) {(a)};
    	
    	\node[anchor=south east, font=\bfseries]
    	at ($(group c2r1.north west) + (-8mm,-4mm)$) {(b)};
    	
    	\node[anchor=south east, font=\bfseries]
    	at ($(group c1r2.north west) + (-8mm,-4mm)$) {(c)};
    	
    	\node at
    	($(current bounding box.south west)!0.5!(current bounding box.south east)
    	+ (44mm,38mm)$)
    	{\pgfplotslegendfromname{legend-linear-schroedinger}};
    	
    \end{tikzpicture}
    \caption{Numerical errors for the stochastic Schrödinger equation with\\(a) a potential, (b) a nonlocal nonlinearity, and (c) a Nemytskii-type nonlinearity. 
    }
    \label{fig:all-plots}
\end{figure}

\begin{table}
    \centering
    \renewcommand{\arraystretch}{1.2}
    \caption{Expected and numerical convergence rates 
    for the stochastic Schrödinger equation with (a) a potential and (b) a nonlocal nonlinearity. }
    \label{tab:rates-two-schroedinger}
    \begin{tabular}{@{}lllllll@{}}
        \toprule
         & EXE & CN & IE & EXM & CNM & IEM \\
        \midrule
        Expected Rate & $1/2$ & $1/2$ & $1/2$ & $1$ & $2/3$ & $1/2$ \\
        Numerical Rate for (a) & 0.55 & 0.58 & 0.49 & 0.94 & 0.70 & 0.48 \\
        Numerical Rate for (b) & 0.51 & 0.55 & 0.50 & 0.89 & 0.67 & 0.50 \\
        \bottomrule
    \end{tabular}
\end{table}

\subsection{A nonlinear example}
\label{subsec:simuConv}

Next, we present a simulation in the setting of Section~\ref{subsec:schroedingerConv} with a nonlocal drift $F(u)=-\iu \eta\ast\phi(u)$ that is not of Nemytskii-type but given by the convolution of a Nemytskii-type nonlinearity with a smooth kernel $\eta$.
As convolution kernel, we choose the centred bump function $\eta=\chi$ from \eqref{eq:defBump} and $\phi(z)=z(1+\abs{z}^2)^{-1}$. They fulfil Assumption \ref{ass:SoeNLconv} with $\alpha=\frac{2}{\ell_0}$, $\ell_0=2,3,4$ for EXM, CNM, and IEM, respectively, and $\beta=2$ such that Theorem \ref{thm:SoeNLconv} is applicable for the drift part.
For simplicity, we again consider linear noise $G(u)=-\iu M_uQ^{\frac{1}{2}}$ with $Q^{1/2}$ as in Subsection \ref{subsec:simuLin}.
To estimate the part of the error related to $G$, we apply Theorems~\ref{thm:schroedingerLinearExp}~and~\ref{thm:schroedingerLinearRat}, also with $\alpha=\frac{2}{\ell_0}$ and $\beta=2$.
An inspection of the proofs of Theorems~ \ref{thm:SoeNLconv},~\ref{thm:schroedingerLinearExp}, and \ref{thm:schroedingerLinearRat} shows that we can combine the different error estimates for the different $F$- and $G$-terms.

The numerical errors as well as the expected and numerical rates of convergence are presented in Figure \ref{fig:all-plots}(b) and Table \ref{tab:rates-two-schroedinger}, respectively. 
Also for this nonlinear example, the advantage of the EXM and CNM can be clearly discerned.
The numerical rates conform closely to the expected rates.

\subsection{A further nonlinear example}
\label{subsec:simuNLout}

As an outlook, we consider the nonlinear Nemytskii operator $F(u)=-\iu \phi\circ u$ associated with $\phi(z) =z(1+\abs{z}^2)^{-1}$, again with linear multiplicative noise $G(u) = -\iu M_uQ^{\frac{1}{2}}$. On $Y=H^2$, $F$ is of quadratic growth, which is not covered by our setting. We could apply Theorem~\ref{thm:SoeNLconv} on $Y=H^1$, as $F$ is of linear growth on this space, but this limits the expected convergence rates to $1/2$ for Milstein schemes. 

The numerical simulations, whose results can be found in Figure \ref{fig:all-plots}(c) and Table \ref{tab:rates-nonlinear-schroedinger}, indicate that, nonetheless, the Milstein schemes EXM and CNM converge at rates higher than $1/2$. It is an interesting open question whether error estimates for Milstein schemes can also be obtained under weaker conditions on $F$ and $G$ such as polynomial rather than linear growth assumptions.

\begin{table}
    \centering
    \renewcommand{\arraystretch}{1.2}
    \caption{Numerical convergence rates for the stochastic Schrödinger equation with a Nemytskii-type nonlinearity.}
    \label{tab:rates-nonlinear-schroedinger}
    \begin{tabular}{@{}lllllll@{}}
        \toprule
         & EXE & CN & IE & EXM & CNM & IEM \\
        \midrule
        Numerical Rate & 0.52 & 0.56 & 0.51 & 0.90 & 0.68 & 0.51 \\
        \bottomrule
    \end{tabular}
\end{table}

\FloatBarrier

\end{document}